\documentclass[10pt]{amsart}
\usepackage[dvipdfmx]{graphicx}
\usepackage{amsmath,amscd,amsthm,amsxtra,yhmath,stackrel}
\usepackage{epsfig,graphics,color,colortbl}
\usepackage{amssymb,latexsym}
\usepackage{mathrsfs}
\usepackage{bm}
\usepackage[poly,all]{xy}
\usepackage{hyperref}
\usepackage{tikz-cd}
\usepackage[draft]{todonotes}
\usepackage{upgreek}
\usepackage{ytableau}
\usepackage{stmaryrd}
\usepackage{verbatim}
\usepackage{mathtools}
\usepackage[title]{appendix}
\usepackage{enumerate}
\usepackage[normalem]{ulem}
\usepackage{datetime}
\usepackage[pagewise]{lineno} 

\newtheorem{thm}{\bf Theorem}[section]
\newtheorem{df}[thm]{\bf Definition}
\newtheorem{prop}[thm]{\bf Proposition}
\newtheorem{cor}[thm]{\bf Corollary}
\newtheorem{lem}[thm]{\bf Lemma}
\newtheorem{rem}[thm]{\bf Remark}
\newtheorem{ex}[thm]{\bf Example}

\numberwithin{equation}{section}

\newcommand{\mc}{\mathcal}
\newcommand{\mf}{\mathfrak}
\newcommand{\ms}{\mathscr}
\newcommand{\mb}{\bm}

\newcommand{\pf}{\noindent{\bfseries Proof. }}
\newcommand{\ov}{\overline}

\newcommand{\K}{{\mc K}}
\newcommand{\U}{{\mc U}}

\newcommand{\Boson}{B_q}
\newcommand{\cmB}{\Delta_{\Boson}}

\newcommand{\cP}{\mathscr{P}}
\newcommand{\cO}{\mc{O}}
\newcommand{\I}{\mathbb{I}}

\newcommand{\bi}{{\bf i}}
\newcommand{\bj}{{\bf j}}
\newcommand{\ttq}{\texttt{q}}

\newcommand{\Z}{\mathbb{Z}}
\newcommand{\Q}{\mathbb{Q}}

\newcommand{\e}{\epsilon}
\newcommand{\de}{\delta}

\newcommand{\te}{\tilde{e}}
\newcommand{\tf}{\tilde{f}}

\newcommand{\gl}{\mf{gl}}

\newcommand{\La}{\Lambda}

\newcommand{\la}{\lambda}

\newcommand{\hf}{\frac{1}{2}}

\newcommand{\red}[1]{{\color{red}#1}}

\newcommand{\ot}{\otimes}

\newcommand{\si}{(-1)^{\e_i}}

\newcommand{\bq}{{\bf q}}
\newcommand{\bff}{{\bf f}}

\begin{document}
\title
[Crystal base of the negative half of the quantum superalgebra]{Crystal base of the negative half of the quantum superalgebra $U_q(\gl(m|n))$}

\author{IL-SEUNG JANG}
\address{Department of Mathematics, Incheon National University, Incheon 22012, Korea}
\email{ilseungjang@inu.ac.kr}

\author{JAE-HOON KWON}
\address{Department of Mathematical Sciences and RIM, Seoul National University, Seoul 08826, Korea}
\email{jaehoonkw@snu.ac.kr}

\author{AKITO URUNO}
\address{Department of Mathematical Sciences, Seoul National University, Seoul 08826, Korea}
\email{aki926@snu.ac.kr}

\thanks{This work is supported by the National Research Foundation of Korea(NRF) grant funded by the Korea government(MSIT) (No.\,2019R1A2C1084833 and 2020R1A5A1016126).}

\begin{abstract}
We construct a crystal base of $U_q(\gl(m|n))^-$, the negative half of the quantum superalgebra $U_q(\gl(m|n))$. We give a combinatorial description of the associated crystal $\ms{B}_{m|n}(\infty)$, which is equal to the limit of the crystals of the ($q$-deformed) Kac modules $K(\la)$. We also construct a crystal base of a parabolic Verma module $X(\la)$ associated with the subalgebra $U_q(\gl_{0|n})$, and show that it is compatible with the crystal base of $U_q(\gl(m|n))^-$ and the Kac module $K(\la)$ under the canonical embedding and projection of $X(\la)$ to $U_q(\gl(m|n))^-$ and $K(\la)$, respectively.
\end{abstract}

\maketitle
\setcounter{tocdepth}{1}

\noindent

\section{Introduction}

The crystal base theory for the quantized enveloping algebra $U_q(\mf{g})$ associated to a symmetrizable Kac-Moody algebra $\mf{g}$ has been one of the most important tools in the representation theory of $U_q(\mf{g})$ \cite{Kas91}, reflecting its fundamental combinatorial structure.
 
For a classical Lie superalgebra $\mf{g}$, there have been several works on the crystal base of a representation of $U_q(\mf{g})$, where the representation theory of $U_q(\mf{g})$ is no longer parallel to that of a symmetrizable Kac-Moody algebra, and hence the existence of a crystal base is not easily expected. It is shown that  an irreducible polynomial representation of $U_q({\mf g})$ has a crystal base when $\mf{g}$ is a general linear Lie superalgebra $\gl(m|n)$ \cite{BKK}, and a queer Lie superalgebra $\mf{q}(n)$ \cite{GJKKK}.
When $\mf{g}$ is an orthosymplectic Lie superalgebra $\mf{osp}(m|n)$ with $m\ge 2$, the existence of a crystal base in the sense of \cite{BKK} is shown in \cite{K15,K16} for a family of irreducible representations, which corresponds to the integrable highest weight representations of the classical Lie algebras from a viewpoint of super duality \cite{CLW}.

Let $U_q(\gl(m|n))$ be the quantized enveloping algebra associated to $\gl(m|n)$ \cite{Ya94}, and let $U_q(\gl(m|n))^-$ be its negative half.
The purpose of this paper is to construct a crystal base of $U_q(\gl(m|n))^-$.

Let $P$ be the integral weight lattice for $\gl(m|n)$ and let $P^+$ be the set of dominant integral weights for the even subalgebra $\gl(m|n)_{0}=\gl(m|0)\oplus \gl(0|n)$. For $\la\in P^+$, let $V(\la)$ be the irreducible highest weight $U_q(\gl(m|n))$-module with highest weight $\la$, which is finite-dimensional.

A tensor power of the natural representation of $U_q(\gl(m|n))$ is semisimple 
and an irreducible representation appearing here is called an irreducible polynomial representation, whose highest weight $\la$ is parametrized by $\cP_{m|n}$, the set of $(m|n)$-hook partitions. 
In \cite{BKK}, it is shown that for $\la\in \cP_{m|n}$, $V(\la)$ has a crystal base $(\ms{L}(\la),\ms{B}(\la))$, and an explicit combinatorial description of its crystal $\ms{B}(\la)$ is given. An important feature of $(\ms{L}(\la),\ms{B}(\la))$ is that it is a lower (resp. upper) crystal base of $V(\la)$ as a $U_q(\gl(m|0))$-module (resp. $U_q(\gl(0|n))$-module) so that the resulting crystal structure on $\ms{B}(\la)$ and hence on $\ms{B}(\mu)\ot \ms{B}(\nu)$ for $\mu,\nu\in \cP_{m|n}$ are more involved than the case of $\gl(m+n)$ (cf.~\cite{KK}).

%
As in the case of a symmetrizable Kac-Moody algebra \cite{Kas91}, one may expect to construct a crystal of $U_q(\gl(m|n))^-$ by taking a limit of  $\ms{B}(\la)$ for $\la\in \cP_{m|n}$. But the crystals $\ms{B}(\la)$ do not seem to admit naturally a directed system of $U_q(\gl(m|n))$-crystals as $\la$ goes to infinity. Moreover, the upper crystal lattice of  an integrable highest weight $U_q(\gl(0|n))$-module is not compatible with the projection from $U_q(\gl(0|n))^-$ onto it. These are two main differences from the case of symmetrizable Kac-Moody algebras. 


In this paper, we use the crystal base of ($q$-deformed) Kac modules to construct a crystal base of $U_q(\gl(m|n))^-$.
A Kac module $K(\la)$ is the $U_q(\gl(m|n))$-module induced from an integrable highest weight $U_q(\gl(m|n)_{0})$-module $V_{m|0}(\la_+)\ot V_{0|n}(\la_-)$ with highest weight $\la=\la_++\la_-\in P^+$. It forms another important class of finite-dimensional indecomposable $U_q(\gl(m|n))$-modules.
It is shown in \cite{K14} that $K(\la)$ has a crystal base $(\ms{L}(K(\la)),\ms{B}(K(\la)))$, and when $\la\in\cP_{m|n}$ it is compatible with $(\ms{L}(\la),\ms{B}(\la))$ under the canonical projection from $K(\la)$ to $V(\la)$, that is, it preserves the crystal lattices and induces a morphism of $U_q(\gl(m|n))$-crystals from $\ms{B}(K(\la))$ onto $\ms{B}(\la)$.

We remark that we use in this paper the generalized quantum group $\U_{m|n}$ \cite{KOS}, which is isomorphic to $U_q(\gl(m|n))$ as a $\Q(q)$-algebra under mild extensions, and which has a comultiplication equal to that of a usual quantum group. The irreducible polynomial modules and Kac modules together with their crystal bases are well-defined for $\U_{m|n}$. Hence the problem is replaced by constructing a crystal base of $\U_{m|n}^-$.

We first consider a directed system of $\ms{B}(K(\la))$ as crystals over $\gl(m|n)_0$, and show that its limit $\ms{B}_{m|n}(\infty)$ has a well-defined abstract $\U_{m|n}$-crystal structure (Theorem \ref{thm: isomorphism kappa}), though $\{\,\ms{B}(K(\la))\,|\,\la\in P^+\,\}$ itself does not form a directed system of $\U_{m|n}$-crystals. 
As a set, the limit $\ms{B}_{m|n}(\infty)$ can be identified with  
\begin{equation}\label{eq:limit crystal}
\ms{B}(\K_{m|n})\times \ms{B}_{m|0}(\infty) \times \ms{B}_{0|n}(\infty), 
\end{equation}
where $\ms{B}(\K_{m|n})$ is the crystal of $K(0)$ or the subalgebra $\K_{m|n}$ of $\U_{m|n}^-$ spanned by the set of PBW-type monomials in odd root vectors, and 
$\ms{B}_{m|0}(\infty)$ and $\ms{B}_{0|n}(\infty)$ are the crystals of the subalgebras $\U_{m|0}^-\cong U_q(\gl(m))^-$ and $\U_{0|n}^-\cong U_{-q^{-1}}(\gl(n))^-$ at $q=0$, respectively. 
As to the (abstract) Kashiwara or crystal operators on $\ms{B}_{m|n}(\infty)$, the operators for $\gl(m|1)\subset \gl(m|n)$ act on the first two components following explicit combinatorial rules, and those for $\gl(0|n)\subset \gl(m|n)$ act only on the last component as usual.

Let $A_t$ be the subring of $f(q)\in \Q(q)$ regular at $q=t$ ($t=0, \infty$). 
Let 
\begin{equation*}
\begin{split}
\ms{L}(\infty) & = \ms{L}(\K_{m|n})\cdot \ms{L}_{m|0}(\infty)\cdot \ms{L}_{0|n}(\infty)
\end{split}
\end{equation*}
be an $A_0$-lattice of $\U_{m|n}^-$, where $\ms{L}(\K_{m|n})$ is the $A_0$-span of PBW-type monomials in odd root vectors, $\ms{L}_{m|0}(\infty)$ is the crystal lattice for $U_{q}(\gl(m))^-$ at $q=0$, and $\ms{L}_{0|n}(\infty)$ is the $A_0$-lattice corresponding to the $A_{\infty}$-dual crystal lattice of $U_{q}(\gl(n))^-$ at $q=\infty$. Then we introduce the crystal operators on $\U_{m|n}^-$ such that $(\ms{L}_{m|n}(\infty),\ms{B}_{m|n}(\infty))$ is a crystal base and the induced crystal is isomorphic to $\ms{B}_{m|n}(\infty)$ \eqref{eq:limit crystal} (Theorem \ref{thm:crystal base of U^-}).

To define the crystal operators on $\U_{m|n}^-$, especially for $\gl(m|0)$, we use the $\U_{m|0}$-comodule structure on the algebra $B_q$ of $q$-bosons of type $A_{m-1}$, by which  we realize the subalgebra $\K_{m|n}\cdot\U_{m|0}^-$ of $\U_{m|n}^-$ as a $B_q$-module, and then apply the crystal base theory for $B_q$-modules in \cite{Kas91}.
The crystal operators for $\gl(0|n)$ are given by twisting the usual ones with respect to the dual bar involution on $\U_{0|n}^-$.

The crystal base $(\ms{L}_{m|n}(\infty),\ms{B}_{m|n}(\infty))$ is compatible with the crystal bases of $K(\la)$ in the following sense: Let $X(\la)$ be a parabolic Verma module induced from an integrable highest weight $\U_{0|n}$-module $V_{0|n}(\la_-)$ with highest weight $\la_-$. 
Instead of the projection from $\U_{m|n}^-$ to $K(\la)$,
we consider
\begin{equation}\label{eq:projections 0}
 \xymatrixcolsep{3pc}\xymatrixrowsep{0pc}
 \xymatrix{
  \U_{m|n}^- \ar@{<-}^{\pi^\vee_-}[r] &  X(\la) \ar@{->}^{\pi_+}[r] & K(\la)
  },
\end{equation}
where $\pi_+$ is the canonical projection and $\pi_-^\vee$ is the map obtained by applying the dual of the projection\, $\U_{0|n}^- \longrightarrow V_{0|n}(\la_-)$ to the projection $\pi_-: \U_{m|n}^- \longrightarrow X(\la)$. 
We show that $X(\la)$ has a crystal base (Theorem \ref{thm:crystal base of parabolic Verma}), and the diagram \eqref{eq:projections 0} is compatible with the crystal bases of $\U_{m|n}^-$, $X(\la)$ and $K(\la)$ in the sense that it induces well-defined maps on their crystal lattices and the associated crystals at $q=0$, and the crystal operators partially commute with the maps (Corollaries \ref{cor:compatibility with U} and \ref{cor:compatibility with K}). 
The lattice $\ms{L}_{m|n}(\infty)$ is slightly different from the lattice spanned by the PBW-type basis in \cite{CHW}  which can be obtained by applying the dual bar involution on $\U_{0|n}^-$ to $\ms{L}_{0|n}(\infty)\subset \ms{L}_{m|n}(\infty)$.

We further discuss the structure of $\ms{B}_{m|n}(\infty)$.
We show that $\ms{B}_{m|n}(\infty)$ is connected if and only if $n=1$, and in general it decomposes as follows:
\begin{equation*}
 \ms{B}_{m|n}(\infty) \cong \left(\ms{B}_{m|1}(\infty)\times \ms{B}_{0|n}(\infty)\right)^{\oplus 2^{m(n-1)}},
\end{equation*}
as a $\U_{m|n}$-crystal up to shift of weights (Theorem \ref{thm:structure of Binfty}), where on the connected component $\ms{B}_{m|1}(\infty)\times \ms{B}_{0|n}(\infty)$, the crystal operators for $\gl(m|1)$ act on the first component and those for $\gl(0|n)$ act on the second component. 
%
We remark that the connected crystal $\ms{B}_{m|1}(\infty)$ for $\U_{m|1}^-$ or for $U_q(\gl(m|1))^-$ is isomorphic to the one constructed in \cite{C}. Also it would be interesting to find a categorical interpretation of the crystal structure on $\ms{B}_{m|1}(\infty)$ from a viewpoint of \cite{KS}.

The paper is organized as follows.
In Section \ref{sec:quantum superalgebra}, we recall the notion of quantum superalgebras and generalized quantum groups of finite type $A$. In Section \ref{sec: CB for HGQG}, we give necessary materials on the crystal bases for homogeneous generalized quantum groups. In Section \ref{sec:polynomial repn}, we review the crystal bases of irreducible polynomial representations and Kac modules. In Section \ref{sec:limit of Kac crystal}, we construct a crystal $\ms{B}_{m|n}(\infty)$ as a limit of the crystals of Kac modules. In Section \ref{sec:crystal base of U^-}, we show that $\U_{m|n}^-$ has a crystal base whose crystal is isomorphic to $\ms{B}_{m|n}(\infty)$ and it is compatible with \eqref{eq:projections 0}. In Appendix \ref{sec:app}, we review the crystal base theory for the algebra of $q$-bosons, which plays a crucial role in this paper.

\section{Quantum superalgebra}\label{sec:quantum superalgebra}

\subsection{Quantum superalgebra $U_q(\mf{gl}({\e}))$}\label{subsec:notations}
We assume the following notations: 
\begin{itemize}
\item[$\bullet$] $\Z_+$: the set of non-negative integers,

\item[$\bullet$] $\Bbbk=\Q(q)$ where $q$ is an indeterminate, $\Bbbk^{\times}=\Bbbk \setminus \left\{0\right\}$,

\item[$\bullet$] $\ell$: a fixed positive integer greater than $2$,

\item[$\bullet$] $\e=(\e_1,\cdots,\e_\ell)$ : a sequence with $\e_i\in \{0,1\}$ ($i=1,\dots, \ell$),

\item[$\bullet$] $m=|\{\,i\,|\,\e_i=0\,\}|$, $n=|\{\,i\,|\,\e_i=1\,\}|$,

\item[$\bullet$] $\e_{m|n}$ : a sequence with $\e_1=\dots=\e_m=0$, $\e_{m+1}=\dots=\e_{m+n}=1$ ($m+n=\ell$),
 
\item[$\bullet$] $\I= \{\,1<2<\cdots <\ell\,\}$ : a linearly ordered set with $\Z_2$-grading $\I=\I_{\ov 0}\cup\I_{\ov 1}$ such that $$\I_{\ov 0}=\{\,i\,|\,\e_i=0\,\},\quad \I_{\ov 1}=\{\,i\,|\,\e_i=1\,\},$$

\item[$\bullet$] $P$ : the free abelian group with a basis $\{\,\delta_i\,|\,i\in \I\,\}$,


\item[$\bullet$] $(\,\cdot\,|\,\cdot\,)$ : a symmetric bilinear form on $P$ such that $(\de_i|\de_j)=(-1)^{\e_i}\de_{ij}$ $(i,j\in \I)$,

\item[$\bullet$] $P^\vee={\rm Hom}_\Z(P,\Z)$ with a basis $\{\,\de^\vee_i\,|\,i\in \I\,\}$ such that $\langle \de_i, \de^\vee_j \rangle =\de_{ij}$ $(i,j\in \I)$,

\item[$\bullet$] $I=\{\,1,\ldots,\ell-1\,\}$,

\item[$\bullet$] $\alpha_i=\de_i-\de_{i+1}\in P$, $\alpha_i^\vee = \de^\vee_i-(-1)^{\e_i+\e_{i+1}}\de^\vee_{i+1}\in P^\vee$  ($i\in I$)

\item[$\bullet$] $Q=\bigoplus_{i\in I}\Z\alpha_i$, $Q^+=\sum_{i\in I}\Z_+\alpha_i$, $Q^-=-Q^+$,

\item[$\bullet$] $I_{\rm even}=\{\,i\in I\,|\,(\alpha_i|\alpha_i)=\pm 2\,\}$, 
$I_{\rm odd}=\{\,i\in I\,|\,(\alpha_i|\alpha_i)=0\,\}$,

\end{itemize}

The {\em quantum superalgebra $U_q(\gl(\e))$ associated to $\e$ or the Cartan matrix $A=(\langle \alpha_j, \alpha_i^\vee \rangle)_{i,j\in I}$} \cite{Ya94} is the associative $\Bbbk$-algebra with $1$ 
generated by $K_\mu, E_i, F_i$ for $\mu\in P$ and $i\in I$ 
satisfying
{\allowdisplaybreaks
\begin{gather*}
K_0=1, \quad K_{\mu +\mu'}=K_{\mu}K_{\mu'} \quad  (\mu, \mu' \in P), \label{eq:Weyl-rel-1-Ya}\\ 
 K_\mu E_i K_\mu^{-1}=q^{(\mu|\alpha_i)}E_i,\quad 
 K_\mu F_i K_\mu^{-1}=q^{-(\mu|\alpha_i)}F_i\quad (i\in I, \mu\in P),\label{eq:Weyl-rel-2-Ya} \\ 
 E_iF_j - (-1)^{p(i)p(j)} F_jE_i =
{(-1)^{\e_i}}\delta_{ij}\frac{K_{i} - K^{-1}_{i}}{q-q^{-1}} \quad (i,j\in I),\label{eq:Weyl-rel-3-Ya}\\
 E_i^2 = F_i^2=0 \quad (i\in I_{\rm odd}), \label{eq:Serre-rel-1-Ya} \\
 E_i E_j - (-1)^{p(i)p(j)} E_j E_i = F_i F_j - (-1)^{p(i)p(j)}  F_j F_i =0
 \quad (\text{$i,j\in I$ and $|i-j|\not= 0, 1$}), \label{eq:Serre-rel-2-Ya} \\ 
\!\!\!
\begin{array}{ll}
E_i^2 E_j- [2] E_i E_j E_i + E_j E_i^2= 0\\ F_i^2 F_j-[2] F_i F_j F_i+F_j F_i^2= 0
\end{array}
\quad (\text{$i\in I_{\rm even}$ and $|i-j|=1$}), \label{eq:Serre-rel-3-Ya}  \\
\!\!\!
\begin{array}{ll}
\left[E_i,\left[\left[E_{i-1},E_i\right]_{(-1)^{p(i-1)}q},E_{i+1}\right]_{(-1)^{(p(i-1)+p(i))p(i+1)}q^{-1}}\right]_{(-1)^{p(i-1)+p(i)+p(i+1)}}=0\\
\left[F_i,\,\left[\left[F_{i-1},\,F_i\right]_{(-1)^{p(i-1)}q},F_{i+1}\right]_{(-1)^{(p(i-1)+p(i))p(i+1)}q^{-1}}\right]_{(-1)^{p(i-1)+p(i)+p(i+1)}}=0 \end{array}\ (i\in I_{\rm odd}),\label{eq:Serre-rel-4-Ya} 
\end{gather*}
\noindent where $[a]=\frac{q^a-q^{-a}}{q-q^{-1}}$ for $a\in\Z_+$,
$p(i)=\e_i+\e_{i+1}$ $(i\in I)$, $[X,Y]_t = XY-tYX$ for $t\in \Bbbk$, and $K_i=K_{\alpha_i}$ for $i\in I$.}
We write $U_q(\gl({m|n}))=U_q(\gl(\e))$ when $\e=\e_{m|n}$.

\subsection{Generalized quantum group $\U(\gl({\e}))$}
Let us recall the notion of the {\em generalized quantum group of finite type $A$ associated to $\e$} \cite{KOS}.
Let 
\begin{itemize}
\item[$\bullet$] $q_i=\si q^{\si}$ $(i\in \I)$, that is,
\begin{equation*}
q_i=
\begin{cases}
q & \text{if $\e_i=0$},\\
-q^{-1} & \text{if $\e_i=1$},\\
\end{cases} \quad (i\in \I),
\end{equation*}

\item[$\bullet$] ${\bq}(\,\cdot\,,\,\cdot\,)$: a symmetric biadditive function from $P\times P$ to $\Bbbk^{\times}$ given by
\begin{equation*}
\bq(\mu,\nu) = \prod_{i\in \I}q_i^{\langle\mu,\delta^\vee_i\rangle \langle\nu,\delta^\vee_i\rangle}.
\end{equation*}
\end{itemize}

\begin{df}\label{def:U(e)}
{\rm
We define ${\U}(\gl(\e))$ to be the associative $\Bbbk$-algebra with $1$ 
generated by $k_\mu, e_i, f_i$ for $\mu\in P$ and $i\in I$ 
satisfying
{\allowdisplaybreaks
\begin{gather*}
k_0=1, \quad k_{\mu +\mu'}=k_{\mu}k_{\mu'} \quad (\mu, \mu' \in P),\label{eq:Weyl-rel-1} \\ 
k_\mu e_i k_{-\mu}=\bq(\mu,\alpha_i)e_i,\quad 
k_\mu f_i k_{-\mu}=\bq(\mu,\alpha_i)^{-1}f_i\quad (i\in I, \mu\in P), \label{eq:Weyl-rel-2} \\ 
e_if_j - f_je_i =\delta_{ij}\frac{k_{i} - k^{-1}_{i}}{q-q^{-1}}\quad (i,j\in I),\label{eq:Weyl-rel-3}\\
e_i^2= f_i^2 =0 \quad (i\in I_{\rm odd}),\label{eq:Weyl-rel-4}
\end{gather*}
and 
\begin{equation*}\label{eq:Serre-rel-1}
\begin{split}
&\ \, e_i e_j -  e_j e_i = f_i f_j -  f_j f_i =0
 \quad \text{($i,j \in I$ and $|i-j|\not = 0, 1$)},\\ 
&
\begin{array}{ll}
e_i^2 e_j- (-1)^{\e_i}[2] e_i e_j e_i + e_j e_i^2= 0\\ 
f_i^2 f_j- (-1)^{\e_i}[2] f_i f_j f_i+f_j f_i^2= 0
\end{array}
\quad \text{($i\in I_{\rm even}$ and $|i-j|= 1$)}, 
\end{split}
\end{equation*}
\begin{equation*}\label{eq:Serre-rel-2}
\begin{array}{ll}
  e_{i}e_{i-1}e_{i}e_{i+1}  
- e_{i}e_{i+1}e_{i}e_{i-1} 
+ e_{i+1}e_{i}e_{i-1}e_{i} \\  
\hskip 2cm - e_{i-1}e_{i}e_{i+1}e_{i} 
+ (-1)^{\e_i}[2]e_{i}e_{i-1}e_{i+1}e_{i} =0, \\ 
  f_{i}f_{i-1}f_{i}f_{i+1}  
- f_{i}f_{i+1}f_{i}f_{i-1} 
+ f_{i+1}f_{i}f_{i-1}f_{i}  \\  
\hskip 2cm - f_{i-1}f_{i}f_{i+1}f_{i} 
+ (-1)^{\e_i}[2]f_{i}f_{i-1}f_{i+1}f_{i} =0,
\end{array}\quad \text{($i\in I_{\rm odd}$)},
\end{equation*}}
where $k_i=k_{\alpha_i}$ for $i\in I$.}
\end{df}

We write $\U(\gl({m|n}))=\U(\gl(\e))$ when $\e=\e_{m|n}$.

\begin{rem}{\rm
The algebra $\U(\gl(\e))$ is the subalgebra of the generalized quantum group $\U(\e)$ of affine type $A$ in \cite{KL,KO}, which is denoted by $\ov{\U}(\e)$ or $\mathring{\U}(\e)$. 
We follow the convention in \cite[Definition 2.1]{KL} for its presentation, while $k_{\de_j}$ corresponds to $\omega_j$ in \cite[Definition 2.1]{KO}.
} 
\end{rem}

The algebra $\U(\gl(\e))$ is closely related to $U_q(\gl(\e))$ in the following sense.
Let $\Sigma$ be the bialgebra over $\Bbbk$ generated by $\sigma_j$ $(j\in \I)$  such that $\sigma_i\sigma_j=\sigma_j\sigma_i$ and $\sigma_j^2=1$ ($i,j\in \I$) where the comultiplication is given by $\Delta(\sigma_j)=\sigma_j\otimes \sigma_j$.
Let $\Sigma$ act on $U_q(\gl(\e))$ by 
{\allowdisplaybreaks 
\begin{equation*}
\begin{split}
&\sigma_j K_\mu =K_\mu,\quad
\sigma_jE_i=(-1)^{\e_j(\delta_j|\alpha_i)}E_i,\quad 
\sigma_jF_i=(-1)^{\e_j(\delta_j|\alpha_i)}F_i,
\end{split}
\end{equation*}
for $j\in \I$, $\mu\in P$ and $i\in I$ so that $U_q(\gl(\e))$ is a $\Sigma$-module algebra. Let $U_q(\gl(\e))[\sigma]$ be the semidirect product of $U_q(\gl(\e))$ and $\Sigma$. If we define $\U(\gl(\e))[\sigma]$ in the same way, then there exists an isomorphism of $\Bbbk$-algebras $\tau : U_q(\gl(\e))[\sigma] \longrightarrow \U(\gl(\e))[\sigma]$
such that $\tau(\sigma_j)=\sigma_j$ ($j\in \I$), $\tau(E_i)=e_i u_i$, $\tau(F_i)=f_i v_i$, and $\tau(K_i)=k_i w_i$ ($i\in I$) for some monomials $u_i, v_i, w_i$ in $\Sigma$ by \cite[Theorem 2.7]{KL}. 

For example, when $\e=\e_{m|n}$, we have  
\begin{equation}\label{eq:iso tau standard}
\xymatrixcolsep{2pc}\xymatrixrowsep{3pc}\xymatrix{
 U_q(\gl(m|n))[\sigma] \ \ar@{->}^{\tau}[r] &\ \U(\gl(m|n))[\sigma]},
\end{equation}
where
{\allowdisplaybreaks
\begin{align*}\label{eq:iso phi}
& \tau(K_{\de_j})= 
\begin{cases}
k_{\de_j} & \text{$(1\leq j\leq m)$},\\ 
k_{\de_j}\sigma_j & \text{$(m< j\leq \ell)$},
\end{cases}\\
& \tau(E_i)= 
\begin{cases}
e_i & \text{$(1\leq i\leq m)$},\\
-e_i(\sigma_i\sigma_{i+1})^{i+m} & \text{$(m< i\leq \ell-1)$},
\end{cases}\\
& \tau(F_i)= 
\begin{cases}
f_i & \text{$(1\leq i< m)$},\\
f_m\sigma_{m+1} & \text{$(i=m)$},\\
f_i(\sigma_i\sigma_{i+1})^{i+m+1} & \text{$(m< i\leq \ell-1)$},
\end{cases}
\end{align*}
}
(see also \cite[Proposition 4.4]{KO}).

Let $\U(\gl(\e))^+$ (resp. $\U(\gl(\e))^-$) be the subalgebra of $\U(\gl(\e))$ generated by $e_i$ (resp. $f_i$) for $i\in I$, and let $\U(\gl(\e))^0$ be the one generated by $k_\mu$ for $\mu\in P$.
We have $\U(\gl(\e))\cong \U(\gl(\e))^-\ot\U(\gl(\e))^0\ot\U(\gl(\e))^+$ as a $\Bbbk$-space, which follows from \cite[Theorem 10.5.1]{Ya94} and \eqref{eq:iso tau standard}. We call $\U(\gl(\e))^+$ (resp.~$\U(\gl(\e))^-$) the positive (resp.~negative) half of $\U(\gl(\e))$.
Note that $\U(\gl(\e))^\pm =\bigoplus_{\alpha\in Q^\pm}\U(\gl(\e))^{\pm}_\alpha$, where 
\begin{equation}\label{eq:Q grading}
 \U(\gl(\e))^{\pm}_\alpha=\{\,u\,|\,k_\mu u k_{-\mu}=\bq(\mu,\alpha)u\ (\mu\in P)\,\}.
\end{equation}
We put $|x|=\alpha$ for $x\in \U(\gl(\e))^\pm_\alpha$.
There is a Hopf algebra structure on $\U(\gl(\e))$, where the comultiplication $\Delta$, the antipode $S$ is given by 
\begin{equation}\label{eq:comult-1}
\begin{split}
& \Delta(k_\mu)=k_\mu\otimes k_\mu, \\ 
& \Delta(e_i)= 1\ot e_i + e_i\ot k_i^{-1}, \\
& \Delta(f_i)= f_i\ot 1 + k_i\ot f_i , \\  
S(k_i)= & k_i^{-1}, \quad S(e_i)=-e_i k_i, \quad S(f_i)=-k_i^{-1} f_i,
\end{split}
\end{equation}
for $\mu\in P$ and  $i\in I$.
Let $\eta$ be the anti-involution on $\U(\gl(\e))$ defined by
\begin{equation*}\label{eq:antiauto-1}
 \eta(e_i)=q_if_ik^{-1}_i,\quad \eta(f_i)=q^{-1}_ik_ie_i,\quad \eta(k_\mu)=k_\mu \quad (i\in I,\ \mu \in P).
\end{equation*}
Let $- : \U(\gl(\e))\longrightarrow \U(\gl(\e))$ be the involution of a $\Q$-algebra given by 
\begin{equation}\label{eq:bar involution}
\ov{q}=q^{-1}, \quad \ov{e_i}=e_i,\quad \ov{f_i}=f_i,\quad \ov{k_\mu}=k_{\mu}^{-1} \quad (i\in I,\ \mu \in P).
\end{equation}
For $i\in I$, let $e'_i$ and $e''_i$ denote the $\Bbbk$-linear maps on $\U(\gl(\e))^-$ defined by 
$e'_i(f_j)=\de_{ij}$ and $e''_i(f_j)=\de_{ij}$ for $j\in I$ and 
\begin{equation} \label{eq: derivation}
\begin{split}
e'_i(uv)&=e'_i(u)v + \bq{(\alpha_i,|u|)}ue'_i(v),\\
e''_i(uv)&=e''_i(u)v + \bq{(\alpha_i,|u|)}^{-1}ue''_i(v),
\end{split}
\end{equation}
for $u,v\in \U(\gl(\e))^-$ with $u$ homogeneous.

\subsection{Weight spaces and highest weight $\U(\gl(\e))$-modules}
For a $\U(\gl(\e))$-module $V$ and $\mu\in P$, let 
\begin{equation*}
V_\mu 
= \{\,u\in V\,|\,k_{\nu} u= \bq(\mu,\nu) u \ \ (\nu\in P) \,\}
\end{equation*}
be the $\mu$-weight space of $V$. For $u\in V_\mu\setminus\{0\}$, we call $u$ a weight vector with weight $\mu$ and put ${\rm wt}(u)=\mu$. 
Note that $ e_i V_\mu\subset V_{\mu+\alpha_i}$, and $f_i V_\mu\subset V_{\mu-\alpha_i}$ for $i\in I$.

Let $\U(\gl(\e))^{\ge 0}$ be the subalgebra generated by $k_\mu$ and $e_i$ for $\mu\in P$ and $i\in I$.
For $\la \in P$, let ${\bf 1}_\la=\Bbbk v_\la$ be the one-dimensional $\U(\gl(\e))^{\ge 0}$-module such that $e_i v_\la =0$ and $k_{\mu} v_\la = \bq(\la,\mu)v_\la$ for $i\in I$ and $\mu\in P$. 
Let 
\begin{equation*}
 M_\e(\la)=\U(\gl(\e))\ot_{\U(\gl(\e))^{\ge 0}}{\bf 1}_\la.
\end{equation*}
We have $M_\e(\la)=\bigoplus_{\mu}M_\e(\la)_\mu$, where the sum is over $\mu = \la-\sum_{i\in  I}c_i\alpha_i$ with $c_i\in\Z_+$ and ${\rm dim}M_\e(\la)_\mu<\infty$.
Let $V_\e(\la)$ be the maximal irreducible quotient of $M_\e(\la)$. 

\begin{rem}\label{rem:tau pullback}
{\rm 
Let $V$ be a $\U(\gl(\e))$-module with $V=\bigoplus_{\mu\in P}V_\mu$.
Then $V$ can be extended to a $\U(\gl(\e))[\sigma]$-module by defining
$\sigma_j u = (-1)^{\e_j\mu_j} u$
for $j\in \I$ and $u\in V_\mu$ with $\mu=\sum_i\mu_i\delta_i$.
Let $V^\tau=\{\,u^\tau\,|\,u\in V\,\}$ be a $U_q(\mf{gl}(\e))$-module, where $x u^\tau = (\tau(x)u)^\tau$ for $x\in U_q(\mf{gl}(\e))$ and $u\in V$. 
Then we have
$V^\tau=\bigoplus_{\mu\in P_{\ge 0}}V^\tau_\mu$, where
\begin{equation*}\label{eq:polynomial weight tau}
V^\tau_\mu 
= \{\,u\in V^\tau\,|\,K_{\nu} u= q^{(\mu|\nu)} u \ \ (\nu\in P) \,\}.
\end{equation*}}
\end{rem}

\subsection{The case of $\e_{m|n}$}\label{subsec:notations-standard}
From now on, we assume that $\e=\e_{m|n}$ so that 
\begin{gather*}
\I_0=\{\,1,\dots,m\,\},\ \I_1=\{\,m+1,\dots,m+n=\ell\,\},\\  
I_{\rm even}=I\setminus\{m\},\quad  I_{\rm odd}=\{\,m\,\}.
\end{gather*}
We let 
\begin{itemize}
\item[$\bullet$]  $\U=\U(\gl(\e_{m|n}))$, $\U^\pm = \U(\gl(\e_{m|n}))^\pm$,



\item[$\bullet$] $\U_{m,n}$ : the subalgebra  generated by $k_\mu, e_k, f_k$ ($\mu\in P, k\in I_{\rm even}$),

\item[$\bullet$] $\U^\pm_{m,n} = \U_{m,n}\cap \U^\pm$,

\item[$\bullet$]  $\U_{m|0}$ : the subalgebra generated by $k_{\de_i}, e_k, f_k$ ($i\in \I_0, 1\le k< m$),

\item[$\bullet$]  $\U_{0|n}$ : the subalgebra generated by $k_{\de_j}, e_k, f_k$ ($j\in \I_1,\, m< k\le \ell-1$),

\item[$\bullet$] $M(\la)=M_{\e_{m|n}}(\la)$, $V(\la)=V_{\e_{m|n}}(\la)$ for $\la\in P$.


\end{itemize}
Let 
\begin{equation*}\label{eq:dominant integral}
P^+=\left\{\,\la=\sum_{i\in \I}\la_i\de_i \in P\ \Bigg|\ \lambda_{1}\geq\ldots\geq\lambda_{m},\ \ \lambda_{m+1}\geq \ldots\geq \lambda_{m+n}\,\right\}.
\end{equation*}
For $\lambda\in P^+$, let 
\begin{equation*}
\begin{split}
\lambda_{+}=\sum_{i\in \I_0}\la_{i}\de_i,\ \ \ \ \
\lambda_{-}=\sum_{j\in \I_1}\la_{j}\de_j.
\end{split}
\end{equation*}

\subsection{PBW-type basis of $\U^-$} \label{subsec: PBW basis of U-}

Let 
\begin{equation*}
\begin{split}
\Phi^+&=\{\, \de_a-\de_b \,|\,   a<b  \,\}, \\
\Phi^+_{\ov 0}&=\{\,\de_a-\de_b\,|\,a<b, \ \e_a = \e_b\,\}=\{\,\alpha\in \Phi^+\,|\,(\alpha|\alpha)=\pm2\,\},\\
\Phi^+_{\ov 1}&=\{\,\de_a-\de_b\,|\,a<b, \ \e_a \neq \e_b \,\}=\{\,\alpha\in \Phi^+\,|\,(\alpha|\alpha)=0\,\}\\
\end{split}
\end{equation*} 
be the set of positive, even positive and odd positive roots of the general linear Lie superalgebra $\gl_{m|n}$, respectively (cf.~\cite{CW}). 
We have $\Phi^+=\Phi^+_{\ov 0}\cup \Phi^-_{\ov 1}$ and $\Phi^+_{\ov 0}=\Phi^+_{m|0}\cup \Phi^+_{0|n}$, where
\begin{equation*}
\begin{split}
\Phi^+_{m|0}&=\{\,\de_a-\de_b\,|\,1\le a<b\le m \,\}=\{\,\alpha\in \Phi^+\,|\,(\alpha|\alpha)=2\,\},\\
\Phi^+_{0|n}&=\{\,\de_a-\de_b\,|\,m< a<b\le \ell\,\}=\{\,\alpha\in \Phi^+\,|\,(\alpha|\alpha)=-2\,\}.
\end{split}
\end{equation*}

For homogeneous $x,y\in \U^-$, we let
\begin{equation*} \label{eq: braket}
 [x,y]_{\bq} = xy - \bq(|x|,|y|)^{-1}yx.
\end{equation*}
Let $\beta=\alpha_i+\cdots+\alpha_j \in \Phi^+$ be given with $i\le j$. 
We put
\begin{equation} \label{eq: root vectors}
\begin{split}
 \bff_\beta = 
 \begin{cases}
  {\rm ad}_q(f_j)\circ \dots \circ {\rm ad}_q(f_{m+1}) \circ {\rm ad}_q(f_{i})\circ \dots \circ {\rm ad}_q(f_{m-1})(f_m) & \text{if $\beta\in \Phi^+_{\ov 1}$},\\
   {\rm ad}_q(f_i)\circ \dots \circ   {\rm ad}_q(f_{j-1})(f_j) & \text{if $\beta\in \Phi^+_{m|0}$},\\
   {\rm ad}_q(f_j)\circ \dots \circ   {\rm ad}_q(f_{i+1})(f_i) & \text{if $\beta\in \Phi^+_{0|n}$},
 \end{cases} 
\end{split}
\end{equation}
where ${\rm ad}_q(f_k)(u)=[f_k,u]_\bq$.

Let $\prec$ be a linear order on $\Phi^+$ such that 
for $\alpha=\de_a-\de_b$ and $\beta=\de_c-\de_d \in \Phi^+$ ($\alpha\neq \beta$),
$\alpha\prec \beta$ if and only if the pair $(\alpha,\beta)$ satisfies one of the following conditions (see Example \ref{ex:B infty}):
\begin{itemize}
 \item[(1)] $\alpha\in \Phi^+_{\ov 1}$ and $\beta\in \Phi^+_{\ov 0}$,
 \item[(2)] {$\alpha, \beta \in \Phi^+_{\ov 1}$ with ($b<d$) or ($b=d$ and $a>c$),} 
 \item[(3)] $\alpha, \beta \in \Phi^+_{m|0}$ with ($b>d$) or ($b=d$ and $a>c$),
 \item[(4)] $\alpha, \beta \in \Phi^+_{0|n}$ with ($a<c$) or ($a=c$ and $b<d$),
 \item[(5)] {$\alpha \in \Phi^+_{m|0}$ and $\beta \in \Phi^+_{0|n}$}.
\end{itemize}


\begin{lem}\label{lem:commutation relation}
  Assume $\alpha, \beta \in \Phi^+$ such that $\alpha \prec \beta$. Then we have
  \begin{equation*}
    \begin{split}
     [\bff_\beta, \bff_\alpha]_{\bq} = 
     \begin{cases}
      (q^{-1} - q) \bff_\gamma \bff_\delta & \text{if $\alpha \prec \gamma \prec \delta \prec \beta$ and $\alpha+\beta=\gamma+\delta$},\\
      \bff_\gamma                          & \text{if $\gamma = \alpha + \beta \in \Phi^+$},\\
      0                                    & \text{otherwise},
     \end{cases} 
    \end{split}
  \end{equation*}
  and $\bff_\alpha^2=0$ if $\alpha \in \Phi^+_{\ov 1}$.

\end{lem}
\pf 
It can be proved directly or by the relations in $U_q(\gl(m|n))$ (cf.~\cite[(2.7)]{K14}) and \eqref{eq:iso tau standard}.
\qed\smallskip
  
\begin{prop}\label{prop:PBW basis}
Let
\begin{equation*}
{B}=\Bigg\{\,\overrightarrow{\prod_{\alpha\in\Phi^+}}\bff_\alpha^{m_\alpha} \,\Bigg|\, \text{$m_{\alpha}\in\Z_{\geq 0}$ $(\alpha\in\Phi^+_{\ov 0})$,\quad $m_{\alpha}=0,1$  $(\alpha\in \Phi^+_{\ov 1})$}\,\Bigg\},
\end{equation*} 
where the product is taken in increasing order with respect to $\prec$. Then
${B}$ is a $\Bbbk$-basis of $\U^-$.
\end{prop}
\pf 
	Let $\bar{B}$ be a PBW-type basis of $U_q(\gl(m|n))$ given in \cite[Sections 4 and 5]{CHW} (cf. \cite{Le})
	with respect to a total ordering $\{ m < m-1 < \cdots < 1 < m+1 < \cdots < \ell -1 \}$ on $I$.
	We can check that 
	\begin{equation*}
	\tau(\bar{B})=\Bigg\{\, {\bm \sigma}  \overrightarrow{\prod_{\alpha\in\Phi^+}}\bff_\alpha^{m_\alpha} \,\Bigg|\, \text{$m_{\alpha}\in\Z_{\geq 0}$ $(\alpha\in\Phi^+_{\ov 0})$,\quad $m_{\alpha}=0,1$  $(\alpha\in \Phi^+_{\ov 1})$}\,\Bigg\},
	\end{equation*}
	where ${\bm \sigma}$ is a monomial in $\pm \sigma_i$'s depending on $(m_\alpha)_{\alpha\in\Phi^+}$.
	Using the diamond lemma \cite{Be}, we can check that $B$ is linearly independent. 
	Therefore we conclude that $B$ is a $\Bbbk$-basis of $\U^-$ by dimension argument.
\qed

\section{Crystal bases for homogeneous generalized quantum groups} \label{sec: CB for HGQG}

\subsection{Crystal base for $U_q(\gl(n))$} \label{subsec: CB for Uqgln}
Let
\begin{align} \label{eq: def of U}
U=U_q(\gl_n):=\U(\gl(n|0))=U_q(\gl(n|0)).
\end{align}
Let us recall some necessary results on the crystal bases \cite{Kas91}.
We keep the notations in Section \ref{subsec:notations}. Let $V(\la)=V_{\e_{n|0}}(\la)$ ($\la\in P^+$) for simplicity.

\subsubsection{Crystal base of $V(\la)$ at $q=0$}\label{subsec:crystal base of hw gl q=0}

Let $V$ be a $U$-module with weight space decomposition $V=\bigoplus_{\mu\in P}V_\mu$ such that $e_i$ and $f_i$ act locally nilpotently.

Let $i\in I$ be given.
For a weight vector $u\in V$,
we have $u=\sum_{k \geq 0} f_i^{(k)}u_k$, where $f_i^{(k)}=f_i^k/[k]!$ with $[k]!=[1][2]\dots[k]$, and $e_iu_k=0$ for $k\geq 0$. Then we define the lower crystal operators
\begin{equation*}\label{eq:Kashiwara operator lower}
\te^{\rm\, low}_iu=\sum_{k\geq1}f_i^{(k-1)}u_k,\quad \tf^{\rm\, low}_iu=\sum_{k\geq0}f_i^{(k+1)}u_k.
\end{equation*}
We also define the upper crystal operators 
\begin{equation*}
\te^{\rm\, up}_iu=\sum_{k\geq1}q^{-l_k+2k-1}f_i^{(k-1)}u_k,\quad 
\tf^{\rm\, up}_iu=\sum_{k\geq 0}q^{l_k-2k-1}f_i^{(k+1)}u_k,
\end{equation*}
where $l_k=({\rm wt}(u_k)|\alpha_i)$.

Let $A_0$ be the subring of $\Bbbk$ consisting of $f(q)/g(q)$ with $f(q), g(q)\in \Q[q]$ and $g(0)\neq 0$.
For $\la\in P^+$, a {\it lower crystal base} of $V(\la)$ at $q=0$ is a pair $(L^{\rm\, low}(\la),B^{\rm\, low}(\la))$ given by
\begin{equation}\label{eq:lower crystal lattice for la}
\begin{split}
L^{\rm\, low}(\la)&=\sum_{r\geq 0,\, i_1,\ldots,i_r\in I}A_0 \tf^{\rm\, low}_{i_1}\cdots\tf^{\rm\, low}_{i_r}v_\la, \\
B^{\rm\, low}(\la)&=\{\,\tf^{\rm\, low}_{i_1}\cdots\tf^{\rm\, low}_{i_r}v_\la\!\!\! \pmod{q L^{\rm\, low}(\la)}\,|\,r\geq 0, i_1,\ldots,i_r\in I\,\}\setminus\{0\},
\end{split}
\end{equation}
where $v_\la$ is a highest weight vector in $V(\la)$.  
Then the pair $(L,B)=(L^{\rm\, low}(\la),B^{\rm\, low}(\la))$ is a {\em crystal base} of $V=V(\la)$ with respect to $\te_i=\te^{\rm\, low}_i$ and $\tf_i=\tf^{\rm\, low}_i$ ($i\in I$) in the following sense:
\begin{itemize}
\item[(C1)] $L$ is an $A_0$-lattice of $V$ and $L=\bigoplus_{\mu \in P}L_\mu$, where $L_\mu=L \cap V_\mu$,

\item[(C2)] $B$ is a $\Q$-basis of $L/qL$,

\item[(C3)] $B= \bigsqcup_{\mu \in P} B_\mu$ where $B_\mu \subset (L/qL)_\mu$,

\item[(C4)] $\tilde{e}_i L \subset L, \tilde{f}_i L \subset L$ and $\tilde{e}_iB \subset B \cup \{0\}, \tilde{f}_iB \subset B \cup \{0\}$ for $i \in I$,

\item[(C5)] $\tilde{f}_ib=b'$ if and only if $\tilde{e}_ib'=b$ for $i\in I$ and $b, b' \in B$.
\end{itemize}

We define an {\em upper crystal base} $(L^{\rm\, up}(\la),B^{\rm\, up}(\la))$ of $V(\la)$ at $q=0$ in the same way as in \eqref{eq:lower crystal lattice for la} using $\tilde{f}^{\rm\, up}_i$ for $i\in I$, which also satisfies the above properties (C1)--(C5) with respect to $\te_i=\te^{\rm\, up}_i$ and $\tf_i=\tf^{\rm\, up}_i$ ($i\in I$).

Let $(\ ,\ )_\la$ be a unique non-degenerate symmetric bilinear form on $V(\la)$ satisfying
\begin{equation*}
 (v_\la,v_\la)_\la=1,\quad 
 (e_iu,v)_\la=(u,f_iv)_\la,\quad 
 (k_\mu u, v)_\la = (u,k_\mu v)_\la,
\end{equation*}
for $u,v\in V(\la)$, $i\in I$ and $\mu\in P$. Then we have the following:
\begin{itemize}
\item[(1)] $L^{\rm\, up}(\la)=\{\,u\,|\,(u,L^{\rm\,low}(\la))_\la\subset A_0\,\}\subset V(\la)$, 

\item[(2)] $(\te^{\rm\, low}_ib,b')_\la=(b,\tf^{\rm\, up}_ib')_\la$ for $i \in I$, $b\in B^{\rm\, low}(\la)$ and $b'\in B^{\rm\, up}(\la)$,

\item[(3)] $(\tf^{\rm\, low}_ib,b')_\la=(b,\te^{\rm\, up}_ib')_\la$ for $i \in I$, $b\in B^{\rm\, low}(\la)$ and $b'\in B^{\rm\, up}(\la)$,
\end{itemize}
where we still denote the bilinear form $(\ ,\ )_\la{|_{q=0}}$ by $(\ ,\ )_\la$.

\subsubsection{Crystal base of $V(\la)$ at $q=\infty$}\label{subsec:crystal base of hw gl q=inf}
Let $\ov{\,\cdot\,}$ also denote a $\Q$-linear involution on $V(\la)$ such that $\ov{v}=\ov{x}v_\la$ for $v=x v_\la$ with $x\in U^-$, where $\ov{\,\cdot\,}$ is given in \eqref{eq:bar involution}. Let $A_\infty=\ov{A_0}$.
Then $(\ov{L^{\rm\, low}(\la)},\ov{B^{\rm\, low}(\la)})$ can be viewed as a lower crystal base of $V(\la)$ at $q=\infty$, where the associated crystal operators are given by 
\begin{equation*}
 \te^{\rm\, \ov{low}}_i:= - \circ \te^{\rm\, low}_i\circ -,\quad \tf^{\rm\, \ov{low}}_i:= - \circ \tf^{\rm\, low}_i\circ -\quad (i\in I).
\end{equation*}

Let $\sigma_\la$ be a $\Q$-linear involution of $V(\la)$ given by
$(\sigma_\la(u),v)_\la = \ov{(u,\ov{v})_\la}$ for $u,v\in V(\la)$.
Indeed, we have $\sigma_\la = \ov{\,\cdot\,}$.
Then $(\sigma_\la(L^{\rm\, up}(\la)),\sigma_\la(B^{\rm\, up}(\la)))$ can be viewed as an upper crystal base of $V(\la)$ at $q=\infty$, where the associated crystal operators are given by 
\begin{equation*}
 \te^{\rm\, \ov{up}}_i:= \sigma_\la \circ \te^{\rm\, up}_i\circ \sigma_\la,\quad \tf^{\rm\, \ov{up}}_i:= \sigma_\la \circ \tf^{\rm\, up}_i\circ \sigma_\la\quad (i\in I).
\end{equation*}
Indeed, for $u=\sum_{k \geq 0} f_i^{(k)}u_k$ with $e_iu_k=0$ ($k\geq 0$), we have
\begin{equation*}\label{eq:Kashiwara operators on hw up infty}
\tilde{e}^{\rm\, \ov{up}}_iu=\sum_{k\geq1}q^{l_k-2k+1}f_i^{(k-1)}u_k,\quad 
\tilde{f}^{\rm\, \ov{up}}_iu=\sum_{k\geq 0}q^{-l_k+2k+1}f_i^{(k+1)}u_k,
\end{equation*}
where $l_k=({\rm wt}(u_k)|\alpha_i)$.
Then it follows from the facts in case of $q=0$ that
\begin{itemize}
\item[(1)] $\sigma_\la(L^{\rm\, up}(\la))=\{\,u\,|\,(u,\ov{L^{\rm\,low}(\la)})_\la\subset A_\infty\,\}\subset V(\la)$, 

\item[(2)] $(\te^{\rm\, \ov{low}}_ib,b')_\la=(b,\tf^{\rm\, \ov{up}}_ib')_\la$ for $i \in I$, $b\in \ov{B^{\rm\, low}(\la)}$ and $b'\in \sigma_\la(B^{\rm\, up}(\la))$,

\item[(3)] $(\tf^{\rm\, \ov{low}}_ib,b')_\la=(b,\te^{\rm\, \ov{up}}_ib')_\la$ for $i \in I$, $b\in \ov{B^{\rm\, low}(\la)}$ and $b'\in \sigma_\la(B^{\rm\, up}(\la))$. 
\end{itemize}

\subsubsection{Crystal base of $U^-$ at $q=0$ and $\infty$}\label{subsec:crystal base of gl verma}
Recall that $U^-$ be the subalgebra of $U$ generated by $f_i$ for all $i \in I$.
Let $i\in I$ be given.
For homogeneous $u\in U^-$, we have $u=\sum_{k\ge 0}f_i^{(k)}u_k$ with $e'_i(u_k)=0$ ($k\ge 0$) (cf.~\eqref{eq: derivation}), and define
\begin{equation}\label{eq:Kashiwara operators for U^-}
\te_i u = \sum_{k\ge 1}f_i^{(k-1)}u_k,\quad \tf_i u = \sum_{k\ge 0}f_i^{(k+1)}u_k.
\end{equation}
A crystal base of $U^-$ at $q=0$ is a pair $(L(\infty),B(\infty))$ given by
\begin{equation} \label{eq: CB for U-}
\begin{split}
L(\infty)&=\sum_{r\geq 0,\, i_1,\ldots,i_r\in I}A_0 \tf_{i_1}\cdots\tf_{i_r}1, \\
B(\infty)&=\{\, \,\tf_{i_1}\cdots\tf_{i_r}1\!\!\! \pmod{q L(\infty)}\,|\,r\geq 0, i_1,\ldots,i_r\in I\,\}\setminus\{0\},
\end{split}
\end{equation}
which satisfies the conditions (C1)--(C5) where $(L,B)=(L(\infty),B(\infty))$ and $V=U^-$ with weight space \eqref{eq:Q grading} and with respect to \eqref{eq:Kashiwara operators for U^-}.

Let $(\ ,\ )$ be a unique non-degenerate symmetric bilinear form on $U^-$ satisfying $(1,1)=1$ and $(e'_iu,v)=(u,f_iv)$ for $u,v\in U^-$ and $i\in I$.
Then we have the following:
\begin{itemize}
\item[(1)] $L(\infty)=\{\,u\,|\,(u,L(\infty))\subset A_0\,\}\subset U^-$, 

\item[(2)] $(\te_ib,b')=(b,\tf_ib')$ for $i \in I$ and $b, b'\in B(\infty)$,
\end{itemize}
{where we still denote by $(\ , \ )$ the bilinear form $(\ , \ ){|_{q=0}}$ on $B(\infty)$.}
Then $(\ov{L(\infty)},\ov{B(\infty)})$ can be viewed as a crystal base of $U^-$ at $q=\infty$, where the associated crystal operators are given by 
\begin{equation*}
 {\te}^{\, -}_i:= - \circ \te_i\circ -,\quad {\tf}^{\, -}_i:= - \circ \tf_i\circ -\quad (i\in I).
\end{equation*}

Let $\sigma$ be the $\Q$-linear involution on $U^-$ such that $(\sigma(u),v)= \ov{(u,\ov{v})}$ for $u,v\in U^-$.
Another crystal base of $U^-$ at $q=\infty$ is the pair $(\sigma(L(\infty)),\sigma(B(\infty)))$ \cite{Kim, LNT}, where the associated crystal operators are given by 
\begin{equation*}
 {\te}^{\, \sigma}_i:= \sigma \circ \te_i\circ \sigma,\quad  {\tf}^{\, \sigma}_i:= \sigma \circ \tf_i\circ \sigma\quad (i\in I).
\end{equation*}

\begin{lem} \label{eq: sigma sl2}
Let $u\in U^-_\alpha$ and $i\in I$ be given. 
Then $u$ can be written uniquely as 
$$u=\sum_{k\ge 0}u_kf_i^{(k)},$$ where $e'_i(u_k)=0$ for $k\ge 0$, and 
\begin{equation*}
\te^{\,\sigma}_i u = \sum_{k\ge 1}q^{l_k-2k+2}u_k f_i^{(k-1)},\quad
\tf^{\,\sigma}_i u = \sum_{k\ge 0}q^{-l_k+2k}u_k f_i^{(k+1)},
\end{equation*}
where $l_k=({\rm wt}(u_k)|\alpha_i)$.
\end{lem}
\pf Let $\ast$ be the $\Bbbk$-algebra anti-automorphism of $U$ such that $e_i^\ast=e_i$, $f_i^\ast=f_i$, and $k_\mu^\ast=k_{-\mu}$ for $i\in I$ and $\mu\in P$. 
Then $\sigma(x) = q^{N(\alpha)} (\ov{x})^\ast$ for $x\in U^-_\alpha$, where $N(\alpha)=\hf(\alpha|\alpha)+(\alpha|\rho)$ with $\rho$ the half sum of positive roots in $\Phi^+$. Also we have $\sigma(xy)=q^{(|x|||y|)}\sigma(y)\sigma(x)$ for homogeneous $x,y\in U^-$ \cite[Proposition 3.2]{Kim}.

On the other hand, we have 
$\left(\ast\circ e'_i\circ\ast\right) (x) =q^{(\alpha|\alpha_i)+2}e_i''(x)$ and 
$\left(-\circ e''_i\circ - \right)(x)= e_i'(x)$ for $i \in I$ and $x\in U^-_\alpha$.
Then we have $(e'_i)^\sigma:=\sigma\circ e'_i\circ \sigma =e'_i$. In particular, we have ${\rm Ker}\,e'_i={\rm Ker}\,(e'_i)^\sigma$.

Let $\sigma(u)=\sum_{k\ge 0}f_i^{(k)}v_k$ with $e'_i(v_k)=0$ ($k\ge 0$). 
Then 
\begin{equation*}
\tf_i \sigma(u)=\sum_{k\ge 0}f_i^{(k+1)}v_k. 
\end{equation*}
Applying $\sigma$ to $\sigma(u)$ and $\tf_i \sigma(u)$, we get
\begin{equation*}
\begin{split}
 u&=\sum_{k\ge 0}q^{k(k-1)-(k\alpha_i||v_k|)} \sigma(v_k) f_i^{(k)},\\
 \tf^{\,\sigma}_i u &= \sum_{k\ge 0}q^{k(k+1)-((k+1)\alpha_i||v_k|)} \sigma(v_k) f_i^{(k+1)}.
\end{split}
\end{equation*}
Putting $u_k=q^{k(k-1)-(k\alpha_i||v_k|)} \sigma(v_k)\in {\rm Ker}(e'_i)^\sigma={\rm Ker} \, e'_i$, we get
\begin{equation*}
\begin{split}
 u=\sum_{k\ge 0}u_k f_i^{(k)}, \quad
  \tf^{\,\sigma}_i u = \sum_{k\ge 0}q^{2k-(\alpha_i||v_k|)} u_k f_i^{(k+1)}.
\end{split}
\end{equation*}
The proof for $\te^{\,\sigma}_i u$ is similar.
\qed\smallskip

Finally, let us review the relation between the crystal bases of $V(\la)$ and $U^-$ at $q=0$ and $\infty$.
Let $\pi_\la : U^- \longrightarrow V(\la)$ be the canonical projection as a $U^-$-module. Then it induces a surjective $A_0$-linear map
\begin{equation*}\label{eq:}
\xymatrixcolsep{2pc}\xymatrixrowsep{3pc}\xymatrix{
 \pi_\la  :  L(\infty)   \ \ar@{->}[r] &\ L^{\rm low}(\la)},
\end{equation*}
such that $\ov{\pi}_\la(B(\infty))=B(\la)\cup \{0\}$ and  $\ov{\pi}_\la(\tf_i b)= \tf^{\rm \,low}_i \ov{\pi}_\la(b)$ for $i\in I$ and $b\in B(\infty)$, where $\ov{\pi}_\la$ is the induced map from $L(\infty)/qL(\infty)$ to $L^{\rm low}(\la)/qL^{\rm low}(\la)$. 
On the other hand, we have an embedding
$\pi^\vee_\la : V(\la)^\vee \longrightarrow (U^-)^\vee$ where $V(\la)^\vee$ and $(U^-)^\vee$ are the restricted duals of $V(\la)$ and $U^-$. If we identify $V(\la)$ and $U^-$ with $V(\la)^\vee$ and $(U^-)^\vee$ by $(\ ,\ )_\la$ and $(\ ,\ )$, respectively, then we have an $A_0$-linear map
\begin{equation*}\label{eq:}
\xymatrixcolsep{2pc}\xymatrixrowsep{3pc}\xymatrix{
 \pi_\la^\vee :  L^{\rm\, up}(\la)   \ \ar@{->}[r] &\ L(\infty)},
\end{equation*}
such that $\ov{\pi}^\vee_\la(B^{\rm\, up}(\la))\subset B(\infty)$ and $\ov{\pi}^\vee_\la(\te^{\rm \,up}_i b)= \te_i \ov{\pi}^\vee_\la(b)$ for $i\in I$ and $b\in B^{\rm\, up}(\la)$, where $\ov{\pi}^\vee_\la$ is the induced map from $L^{\rm\, up}(\la)/qL^{\rm\, up}(\la)$ to $L(\infty)/qL(\infty)$.

The map $\pi^\vee_\la$ is also compatible with the crystal bases at $q=\infty$ as follows.

\begin{lem}\label{lem:embedding}
Under the above hypothesis, we have
\begin{equation*}\label{eq:}
\xymatrixcolsep{2pc}\xymatrixrowsep{3pc}\xymatrix{
 \pi_\la^\vee :  \sigma_\la(L^{\rm\, up}(\la))   \ \ar@{->}[r] &\ \sigma(L(\infty))},
\end{equation*}
where $\ov{\pi}^\vee_\la(\sigma_\la(B^{\rm\, up}(\la)))\subset \sigma(B(\infty))$ and $\ov{\pi}^\vee_\la(\te^{\rm \,\ov{up}}_i b)= \te^{\,\sigma}_i \ov{\pi}^\vee_\la (b)$ for $i\in I$ and $b\in \sigma_\la(B^{\rm\, up}(\la))$. 
\end{lem}
\pf Recall that we have for $u\in V(\la)$ and $v\in U^-$
\begin{equation*}
 (\pi^\vee_\la(u),v) = (u,\pi_\la(v))_\la.
\end{equation*}
Let $u\in \sigma_\la(L^{\rm\, up}(\la))$ be given.
Since $\sigma_\la(L^{\rm\, up}(\la))=\{\,u\,|\,(u,\ov{L^{\rm\,low}(\la)})_\la\subset A_\infty\,\}\subset V(\la)$, we have
\begin{equation*}
 (\pi^\vee_\la(u),\ov{L(\infty)}) = (u,\pi_\la(\ov{L(\infty)}))_\la = (u,\ov{\pi_\la(L(\infty))})_\la=(u,\ov{L^{\rm low}(\la)})_\la\subset A_\infty.
\end{equation*}
Hence $\pi^\vee_\la(u)\in \sigma(L(\infty))$.
Now we have $\ov{\pi}^\vee_\la(\te^{\rm \,\ov{up}}_i b)= \te^{\,\sigma}_i \ov{\pi}^\vee_\la (b)$ for $i\in I$ and $b\in \sigma_\la(B^{\rm\, up}(\la))$ since the following diagram is commutative
\begin{equation*}
 \xymatrixcolsep{2pc}\xymatrixrowsep{3.5pc}\xymatrix{
  L^{\rm\, up}(\la)   \ \ar@{->}^{\pi_\la^\vee}[r] \ar@{->}_{\sigma_\la}[d] &\  L(\infty) \ar@{->}^{\sigma}[d]\\
  \sigma_\la(L^{\rm\, up}(\la))   \ \ar@{->}_{\pi_\la^\vee}[r] &\ \sigma(L(\infty))}.
\end{equation*}
The proof completes.
\qed

\subsection{Crystal bases for $\U_{m|0}$ and $\U_{0|n}$}\label{eq:crystal base for homogeneous case}
In this subsection, we define the crystal operators for $\U_{m|0}$ and $\U_{0|n}$, which will be used in the later sections. For notational convenience, we assume in this subsection that $\U_{0|n} = \U(\gl(\epsilon_{0|n}))$, which is generated by $k_{\de_j}, e_k, f_k$ ($1\le j\le n,\, 1\le k\le n-1$) by shifting the indices $j$ and $k$ (cf.~ Section \ref{subsec:notations-standard}).
}
\subsubsection{$\U_{m|0}$}
First, consider the case of $\U_{m|0}=U_q(\gl_m)$.
We let  
\begin{equation}\label{eq:crystal for m|0}
 (\ms{L}_{m|0}(\infty),\ms{B}_{m|0}(\infty))=(L(\infty),B(\infty)), 
\end{equation}
where $(L(\infty),B(\infty))$ is the crystal base of $U_q(\gl_m)^-$ at $q=0$ in Section \ref{subsec:crystal base of gl verma}.

For $\la\in P^+$, let $V_{m|0}(\la)$ and $V(\la)$ denote the irreducible representations of $\U_{m|0}$ and $U_q(\gl_m)$ with highest weight $\la$, respectively, which coincide in this case.
Let
\begin{equation}\label{eq:crystal for m|0 hw}
 (\ms{L}_{m|0}(\la),\ms{B}_{m|0}(\la))=(L^{\rm\, low}(\la),B^{\rm\, low}(\la)), 
\end{equation}
where $(L^{\rm\, low}(\la),B^{\rm\, low}(\la))$ is the lower crystal base of $V(\la)$ at $q=0$ in Section \ref{subsec:crystal base of hw gl q=0}.
We simply denote by $\te_i$ and $\tf_i$ for $i=1,\dots,m-1$ the crystal operators in \eqref{eq:crystal for m|0} and \eqref{eq:crystal for m|0 hw}.

\subsubsection{$\U_{0|n}$} \label{subsec: crystal for 0|n}
Next, consider the case of $\U_{0|n}$.
%
Let $\ttq$ be an indeterminate.
There exists an isomorphism of $\Q$-algebras 
\begin{equation}\label{eq:psi}
 \xymatrixcolsep{2pc}\xymatrixrowsep{3pc}\xymatrix{
 \psi : U_\ttq(\gl_n) \ \ar@{->}[r] &\ \U_{0|n}},
\end{equation}
given by $\psi(\ttq)=-q^{-1}$, $\psi(e_i)=e_i$, $\psi(f_i)=f_i$, and $\psi(k_\mu)=k_\mu$ for $i=1,\dots,n-1$ and $\mu\in P$.
We let
{
\begin{equation}\label{eq:crystal for 0|n}
\begin{split}
\ms{L}_{0|n}(\infty) &= \psi\circ\sigma(L(\infty)),\\
\ms{B}_{0|n}(\infty) &= \psi\circ\sigma(B(\infty)) \,\cup\, \left(-\psi\circ\sigma(B(\infty)) \right),
\end{split}
\end{equation}
where $(\sigma(L(\infty)),\sigma(B(\infty)))$ is the crystal base of $U_\ttq(\gl_n)^-$ at $\ttq=\infty$ in Section \ref{subsec:crystal base of gl verma}.
}

For $\la\in P^+$, let $V_{0|n}(\la)$ and $V(\la)$ denote the irreducible representations of $\U_{0|n}$ and $U_\ttq(\gl_n)$ with highest weight $\la$, respectively. There exists a unique  $\Q$-linear isomorphism $\psi_\la: V(\la) \longrightarrow V_{0|n}(\la)$ such that $\psi_\la(v_\la)=v_\la$ and $\psi_\la(x v)=\psi(x)\psi_\la(v)$ for $x\in U_\ttq(\gl_n)$ and $v\in V(\la)$. 
Let
{
\begin{equation}\label{eq:crystal for 0|n hw}
\begin{split}
	\ms{L}_{0|n}(\la) &= \psi_\la\circ\sigma_\la(L^{\rm\, up}(\la)), \\
	\ms{B}_{0|n}(\la) &= \psi_\la\circ\sigma_\la(B^{\rm\, up}(\la)) \,\cup\, \left( -\psi_\la\circ\sigma_\la(B^{\rm\, up}(\la)) \right),
\end{split}
\end{equation}
where $(\sigma_\la(L^{\rm\, up}(\la)),\sigma_\la(B^{\rm\, up}(\la)))$ is the upper crystal base of $V(\la)$ at $\ttq=\infty$ in Section \ref{subsec:crystal base of hw gl q=inf}.
}

Let us denote by $\te^{\,\infty}_i$ and $\tf^{\,\infty}_i$ for $i=1,\dots,n-1$ the crystal operators on \,$\U_{0|n}^-$ and $V_{0|n}(\la)$ 
induced from \eqref{eq:crystal for 0|n} and \eqref{eq:crystal for 0|n hw}, respectively.

By Lemma \ref{lem:embedding}, we have
\begin{equation}\label{eq:dual embedding}
\xymatrixcolsep{2pc}\xymatrixrowsep{3pc}\xymatrix{
 \pi_\la^\vee :  \ms{L}_{0|n}(\la)  \ \ar@{->}[r] &\ \ms{L}_{0|n}(\infty)},
\end{equation}
where $\ov{\pi}^\vee_\la(\ms{B}_{0|n}(\la))\subset \ms{B}_{0|n}(\infty)$ and $\ov{\pi}^\vee_\la(\te^{\,\infty}_i b)= \te^{\,\infty}_i \ov{\pi}^\vee_\la (b)$ for $i=1,\dots,n-1$ and $b\in \ms{B}_{0|n}(\la)$.

\subsubsection{Crystal operators for $\U_{0|n}$}
Let $u\in V_{0|n}(\la)$ be a weight vector. 
For $i=1,\dots,n-1$, if $u=\sum_{k \geq 0} f_i^{(k)}u_k$ with $e_iu_k=0$ ($k\geq 0$), then since $\sigma_\la(u) = \ov{u}$ we have
\begin{equation}\label{eq:induced Kashiwara operators}
\te^{\,\infty}_iu=\sum_{k\geq1}\pm q^{-l_k+2k-1}f_i^{(k-1)}u_k,\quad 
\tf^{\,\infty}_iu=\sum_{k\geq 0}\pm q^{l_k-2k-1}f_i^{(k+1)}u_k,
\end{equation}
where $\pm$ depends on the weight of $u_k$ and $k$, and {$l_k= - ({\rm wt}(u_k)|\alpha_i)$} 
(recall the difference of symmetric bilinear forms on the weight lattices for $U_q(\gl_n)=\U_{n|0}$ and $\U_{0|n}$). 
We have
\begin{equation}\label{eq:induced crystal base}
\begin{split}
\ms{L}_{0|n}(\la)&=\sum_{r\geq 0,\, i_1,\ldots,i_r\in I}A_0 \tf^{\,\infty}_{i_1}\cdots\tf^{\,\infty}_{i_r}v_\la, \\
\ms{B}_{0|n}(\la)&=\{\,\pm \,\tf^{\,\infty}_{i_1}\cdots\tf^{\,\infty}_{i_r}v_\la\!\!\! \pmod{q \ms{L}_{0|n}(\la)}\,|\,r\geq 0, i_1,\ldots,i_r\in I\,\}\setminus\{0\}.
\end{split}
\end{equation}
The pair $(L,B)=(\ms{L}_{0|n}(\la),\ms{B}_{0|n}(\la))$ satisfies the conditions (C1)--(C5) in Section \ref{subsec:crystal base of hw gl q=0} with respect to \eqref{eq:induced Kashiwara operators} except that (C2) and (C5) are replaced by 
\begin{itemize}
 \item[(C$2'$)] $B$ is a signed basis of $L/qL$, that is $B=\mathbf{B}\cup -\mathbf{B}$ where $\mathbf{B}$ is a $\Q$-basis of $L/qL$, 
 
 \item[(C$5'$)] $\tf_ib=b'$ if and only if $\te_ib'=\pm b$ for $i\in I$ and $b, b' \in B$, where $\te_i=\te^{\,\infty}_i$ and $\tf_i=\tf^{\,\infty}_i$.
\end{itemize}

We may replace $\te^{\,\infty}_i$ and $\tf^{\,\infty}_i$ ($i=1,\dots,n-1$) for $(\ms{L}_{0|n}(\la),\ms{B}_{0|n}(\la))$ with
\begin{equation}\label{eq:modified induced Kashiwara operators}
\tilde{e}_iu=\sum_{k\geq1}q^{-l_k+2k-1}f_i^{(k-1)}u_k,\quad 
\tilde{f}_iu=\sum_{k\geq 0}q^{l_k-2k-1}f_i^{(k+1)}u_k.
\end{equation}
More precisely, if we let $(\ms{L}_{0|n}'(\la),\ms{B}_{0|n}'(\la))$ be the pair defined as in \eqref{eq:induced crystal base} with respect to \eqref{eq:modified induced Kashiwara operators}, then we have the following.

\begin{lem}\label{lem:comparison of crystal bases}
Under the above hypothesis, we have
\begin{itemize}
 \item[(1)] $(\ms{L}_{0|n}(\la),\ms{B}_{0|n}(\la))=(\ms{L}_{0|n}'(\la),\ms{B}_{0|n}'(\la))$,

 \item[(2)] $\tilde{x}^{\,\infty}_ib \equiv \pm \tilde{x}_i b \pmod{q\ms{L}_{0|n}(\la)}$ for $b\in \ms{B}_{0|n}(\la)$, $x=e, f$ and $i=1,\dots,n-1$.
\end{itemize}
\end{lem}
\pf Let $i=1,\dots,n-1$ be given. For $u\in \ms{L}_{0|n}(\la)$, let $u=\sum_{k \geq 0} f_i^{(k)}u_k$ with $e_iu_k=0$ ($k\geq 0$).
Then
\begin{equation*}
\tf_i^{\infty}u=\sum_{k\geq 0}c_k q^{l_k-2k-1}f_i^{(k+1)}u_k,
\end{equation*}
where $c_k\in \{\pm1\}$ for each $k$ depending on $k$ and $u_k$.
Since $\ms{L}_{0|n}(\la)$ is invariant under $\te^{\,\infty}_i$ and $\tf^{\,\infty}_i$, we have 
\begin{equation*}
 u_k\in \ms{L}_{0|n}(\la),\quad 
 q^{l_k-2k-1}f_i^{(k+1)}u_k\in \ms{L}_{0|n}(\la)\quad (k\ge 0).
\end{equation*}
This implies that $\tf_iu=\sum_{k\geq 0}q^{l_k-2k-1}f_i^{(k+1)}u_k\in \ms{L}_{0|n}(\la)$. Similarly, $\te_i u\in \ms{L}_{0|n}(\la)$.
Hence $\ms{L}_{0|n}(\la)$ is invariant under $\te_i$ and $\tf_i$ for $i=1,\dots,n-1$ and $\ms{L}'_{0|n}(\la)\subset \ms{L}_{0|n}(\la)$.  
  
Let $b\in \ms{B}_{0|n}(\la)$ be given. 
We have
$b\equiv (\tf^{\,\infty}_i)^k u \pmod{q\ms{L}_{0|n}(\la)}$ for some $k\ge 0$ and $u\in \ms{B}_{0|n}(\la)$ with $\te^{\,\infty}_iu=0$.  
By the previous argument, we have $b\equiv \pm \tf_i^k u \pmod{q\ms{L}_{0|n}(\la)}$ and hence 
\begin{equation*}
 \tf_i b \equiv \pm (\tf^{\,\infty}_i)^{k+1} u\equiv \pm \tf^{\,\infty}_ib \pmod{q\ms{L}_{0|n}(\la)}.
\end{equation*}
This shows that $\ms{B}_{0|n}(\la) = \ms{B}'_{0|n}(\la)$. Since $\ms{B}_{0|n}(\la)$ is a signed $\Q$-basis of $\ms{L}_{0|n}(\la)/q\ms{L}_{0|n}(\la)$, we have $\ms{L}_{0|n}(\la) = \ms{L}'_{0|n}(\la)$ by Nakayama lemma.
\qed\smallskip

Similarly, we can replace $\te^{\,\infty}_i$ and $\tf^{\,\infty}_i$ on $\ms{L}_{0|n}(\infty)$ by
\begin{equation}\label{eq:Kashiwara operator L(infty) for 0|n}
\te_i u = \sum_{k\ge 1}q^{-l_k+2k-2}u_k f_i^{(k-1)},\quad
\tf_i u = \sum_{k\ge 0}q^{l_k-2k}u_k f_i^{(k+1)},
\end{equation}
for $u=\sum_{k\ge 0}u_kf_i^{(k)} \in \U_{\alpha}^-$ with $e'_i(u_k)=0$ ($k\ge 0$) and 
$l_k=-({\rm wt}(u_k)|\alpha_i)$ (cf.~Lemma \ref{eq: sigma sl2}).

\subsubsection{PBW-type bases for $\U_{m|0}$ and $\U_{0|n}$} 
Let $U$ be as in \eqref{eq: def of U}. 
Recall the $\Bbbk$-algebra automorphism $T_i : U \longrightarrow U$ for $i = 1, \dots, n-1$, which is $T_{i,1}''$ in \cite[37.1.3]{Lu}, given by
\begin{align*}
T_i(e_j) &= 
\begin{cases}
	-f_i k_i & \text{if $i=j$,} \\
	e_j & \text{if $|i-j|>1$,} \\
	e_i e_j - q^{-1} e_j e_i & \text{if $|i-j|=1$,}
\end{cases} 
\quad
T_i(k_j) =
\begin{cases}
	k_j^{-1} & \text{if $i=j$,} \\
	k_j & \text{if $|i-j| > 1$,} \\
	k_ik_j & \text{if $|i-j|=1$,}
\end{cases}
\\
T_i(f_j) &= 
\begin{cases}
	-k_i^{-1}e_i & \text{if $i=j$,} \\
	f_j & \text{if $|i-j|>1$,} \\
	f_j f_i - q f_i f_j & \text{if $|i-j|=1$.}
\end{cases} 
\end{align*}
for $j = 1, \dots, n-1$.
%
Let $\bi = (i_1, \dots, i_N)$ be a reduced expression of the longest element $w_0$ of Weyl group for $\gl_n$. 
For $1\le k\le N$, let $\beta_k = s_{i_1} \dots s_{i_{k-1}}(\alpha_k) \in \Phi^+$ and $f_{\beta_k}(\bi)= T_{i_1} \dots T_{i_{k-1}}(f_k)$, where $s_i$ denotes the simple reflection for $i$.

First consider the case of $\U_{m|0} = U_q(\gl_m)$,
and let $\bi=(i_1, \dots, i_M)$ be the reduced expression of $w_0$ for $\gl_m$ given by 
$\bi = \bi_{1} \cdot \bi_{2} \cdot \,\dots\, \cdot \bi_{m-1}$,
where $\bi_k = (m-1, m-2, \dots, k)$ for $1 \le k \le m-1$, and $\bi_k \cdot \bi_{k+1}$ denotes the concatenation of $\bi_k$ and $\bi_{k+1}$.
Then the linear order $\beta_1<\dots<\beta_M$ is equal to the one on $\Phi^+_{m|0}$ in Section \ref{subsec: PBW basis of U-} and
\begin{equation*}
	\bff_\beta = f_\beta(\bi)\quad (\beta \in \Phi^+_{m|0}),
\end{equation*}
where $\bff_\beta$ is given in \eqref{eq: root vectors}.
Hence 
\begin{equation} \label{eq: PBW m0}
B_{m|0} := \left\{\, \bff_{\beta_1}^{(c_1)} \dots \bff_{\beta_M}^{(c_M)} \, \big| \, {\bf c}=(c_k) \in \Z_+^M \,\right\} 
\end{equation}
is an $A_0$-basis of $\ms{L}_{m|0}(\infty)$,
where  $\bff_{\beta_k}^{(c_k)} = \frac{1}{[c_k]!}\bff_{\beta_k}^{c_k}$ for $1 \le k \le M$ \cite{Lu90a}. 

Next, we consider the case of $\U_{0|n}$ and let $\bj=(j_1,\dots,j_N)$ be the reduced expression of the longest element for $\gl_n$ given by 
$\bj = \bj_{n-1} \cdot \bj_{n-2} \cdot \, \dots \, \cdot \bj_1$, where $\bj_k = (1, 2, \dots, k)$ for $1 \le k \le n-1$. 
Then the induced linear order $\gamma_1<\dots<\gamma_N$ on $\Phi^+_{0|n}$ is equal to the one in Section \ref{subsec: PBW basis of U-}, where $\gamma_k= s_{j_1}\dots s_{j_{k-1}}(\alpha_k)$ for $1\le k\le N$, and
\begin{equation*}\label{eq:img of root vectors for 0|n}
	\bff_\gamma = \psi(f_\gamma(\bj))\quad (\gamma\in \Phi^+_{0|n}),
\end{equation*}
where $\bff_\gamma$ is given in \eqref{eq: root vectors} and $\psi$ is the isomorphism in \eqref{eq:psi}. 
We also denote by $\sigma$ the $\Q$-linear involution on $\U_{0|n}^-$ defined by $\psi \circ \sigma \circ \psi^{-1}$. Then it follows from \eqref{eq:crystal for 0|n} that 
\begin{equation} \label{eq: PBW 0n}
 \sigma(B_{0|n}) := \left\{\, \sigma\left(\bff_{\gamma_1}^{(c_1)} \dots \bff_{\gamma_N}^{(c_N)}\right) \, \big| \, {\bf c}=(c_k) \in \Z_+^N \,\right\}
\end{equation} 
is an $A_0$-basis of $\ms{L}_{0|n}(\infty)$.

\section{Crystal bases of polynomial modules and Kac modules}\label{sec:polynomial repn}

\subsection{Polynomial representations $V(\la)$}
Let $P_{\geq0}=\sum_{i\in \I}\Z_+\delta _i$. 
Let $\cO_{\geq0}$ be the category of $\U$-modules with objects $V$ such that
$V=\bigoplus_{\mu\in P_{\ge 0}}V_\mu$ with $\dim V_\mu < \infty$.
It is closed under taking submodules and tensor products. 

\begin{rem}{\rm 
Let $V\in \cO_{\geq0}$ be given and consider $V^\tau$ (cf.~Remark \ref{rem:tau pullback}). Then we have $V^\tau\in \cO_{int}$ by \cite[Proposition 4.6]{KO}, where $\cO_{int}$ is the category of $U_q(\gl(\e))$-modules introduced in \cite[Definition 2.2]{BKK}.
 }
\end{rem}

Let $\cP$ be the set of all partitions $\la=(\la_i)_{i\ge 1}$ with $\la_1\ge\la_2\ge\dots$. 
A partition $\la=(\la_i)_{i\ge 1}\in \cP$ is called an $(m|n)$-hook partition if $\la_{m+1}\leq n$ (cf.~\cite{BR}). Let $\cP_{m|n}$ be the set of all $(m|n)$-hook partitions.
For $\la\in \cP_{m|n}$, let
\begin{equation}\label{eq:highest weight correspondence}
\La_\la = \la_1\de_1 +\cdots +\la_m\de_m + \mu_1\de_{m+1}+\cdots+\mu_n\de_{m+n},
\end{equation}
where $(\mu_1,\ldots,\mu_n)$ is the conjugate of the partition $(\la_{m+1},\la_{m+2},\ldots)$. The map $\la \mapsto \La_\la$ is injective, and hence we may identify $\la\in \cP_{m|n}$ with $\La_\la\in P_{\geq 0}$.

\begin{prop}[{\cite[Proposition 4.8]{KO})}]\label{prop:irr poly rep}
Let $V$ be an irreducible $\U$-module in $\cO_{\geq 0}$. Then $V\cong V(\la)$ for some $\la\in \cP_{m|n}$. 
\end{prop}

\subsection{Crystal base of $V(\la)$}

Let us recall the notion of crystal base of a $\U$-module $V$, which is introduced in \cite{BKK} for a $U_q(\gl({m|n}))$-module $V^\tau$. Here we follow the convention in \cite{KO} for $\U$-modules. 

Let $V$ be a $\U$-module which has a weight space decomposition.
The crystal operators $\tilde{e}_i$ and $\tilde{f}_i$ on $V$ for $i\in I$ are given by the ones in Section \ref{eq:crystal base for homogeneous case} for $i\neq m$ and 
\begin{equation*}
\tilde{e}_m u =\eta(f_m) u =q^{-1}k_me_m u,\quad \tilde{f}_m u=f_m u\quad (u\in V).
\end{equation*}

\begin{df}\label{def:crystal for poly}
{\rm A pair $(L,B)$ is a {\em crystal base of $V$} if it satisfies the following conditions:
\begin{itemize}
\item[(1)] $L$ is an $A_0$-lattice of $V$ and $L=\bigoplus_{\mu \in P}L_\mu$, where $L_\mu=L \cap V_\mu$,

\item[(2)] $B$ is a signed basis of $L/qL$, that is $B=\mathbf{B}\cup -\mathbf{B}$ where $\mathbf{B}$ is a $\Q$-basis of $L/qL$,

\item[(3)] $B= \bigsqcup_{\mu \in P} B_\mu$ where $B_\mu \subset (L/qL)_\mu$,

\item[(4)] $\tilde{e}_i L \subset L, \tilde{f}_i L \subset L$ and $\tilde{e}_iB \subset B \cup \{0\}, \tilde{f}_iB \subset B \cup \{0\}$ for $i \in I$,

\item[(5)] $\tilde{f}_ib=b'$ if and only if $\tilde{e}_ib'=\pm b$ for $i\in I$ and $b, b' \in B$.
\end{itemize}
}
\end{df}

Note that a crystal base of a $U_q(\gl(m|n))$-module is defined in \cite{BKK}, where $\tilde{E}_i$ and $\tilde{F}_i$ denote the crystal operators.
The following lemma explains how it is related to a crystal base of a $\U$-module.

\begin{lem}\label{lem:comparison with crystal for super case}
Let $V$ be a $\U$-module in $\mc{O}_{\ge 0}$.
If $(L,B)$ is a crystal base of $V^\tau$ with the crystal operators $\tilde{E}_i$ and $\tilde{F}_i$ $(i \in I)$ as a $U_q(\gl(m|n))$-module, then $(L,B)$ is also a crystal base of $V$. Furthermore, for $b\in B$ and $i\in I$, we have $\tf_ib\equiv \pm \tilde{F}_i b \pmod{qL}$.
\end{lem}
\pf The proof is similar to that of Lemma \ref{lem:comparison of crystal bases}.
Let $(L,B)$ be a crystal base of $V^\tau$ as a $U_q(\gl(m|n))$-module in the sense of \cite[Definition 2.4]{BKK}.

It is clear that $(L,B)$ satisfies the conditions (1)--(3) in Definition \ref{def:crystal for poly}. It suffices to check the conditions (4) and (5).

Suppose that $u\in L$ is given. It is also clear that $L$ satisfies the condition (4) when $1\le i<m$. Suppose that $i>m$. Let $u\in L$ be given with $u=\sum_{k \geq 0} f_i^{(k)}u_k$. 
By \eqref{eq:iso tau standard} and \cite[(2.12)]{BKK}, we have
\begin{equation*}
 u=\sum_{k \geq 0}\pm F_i^{(k)}u_k,\quad 
 \tilde{F}_i u = \sum_{k\geq 0}\pm q^{l_k-2k-1}F_i^{(k+1)}u_k\in L,
\end{equation*}
where $\pm$ depends on the weight of $u_k$ and $k$ (cf.~\eqref{eq:modified induced Kashiwara operators}). 
Since $L$ is a crystal lattice for $V^\tau$, we have 
$u_k\in L$ and $q^{l_k-2k-1}F_i^{(k+1)}u_k\in L$  $(k\ge 0)$.
This implies that $q^{l_k-2k-1}f_i^{(k+1)}u_k\in L$ for $k\ge 0$, and 
\begin{equation*}
\tf_i u = \sum_{k\geq 0} q^{l_k-2k-1}f_i^{(k+1)}u_k\in L.
\end{equation*}
Hence $L$ is invariant under $\tf_i$. The proof for $i=m$ is the same. 
  
Let us show that $B$ satisfies the condition (5).
It is clear when $1 \le i < m$.
Suppose that $m<i\le n-1$.  Let $b\in B$ be given. By \cite[Lemma 2.5]{BKK}  
$b\equiv \tilde{F}_i^k u \pmod{qL}$ for some $k\ge 0$ and $u$ with $\tilde{E}_iu=0$.  
By the previous argument, we have $b\equiv \pm \tf_i^k u \pmod{qL}$ and hence 
$\tf_i b \equiv \pm \tilde{F}_ib \pmod{qL}$.
This shows that $B \cup \{0\}$ is invariant under $\tf_i$. The proof for $i=m$ is the same.
\qed\smallskip

\begin{thm}\label{thm:crystal base of poly repn}
For $\la\in\cP_{m|n}$, let
\begin{equation*}
\begin{split}
\ms{L}(\la)&=\sum_{r\geq 0,\, i_1,\ldots,i_r\in I}A_0 \tilde{x}_{i_1}\cdots\tilde{x}_{i_r}v_\la, \\
\ms{B}(\la)&=\{\,\pm \,{\tilde{x}_{i_1}\cdots\tilde{x}_{i_r}v_\la}\!\!\! \pmod{q \ms{L}(\la)}\,|\,r\geq 0, i_1,\ldots,i_r\in I\,\}\setminus\{0\},
\end{split}
\end{equation*}
where $v_\la$ is a highest weight vector in $V(\la)$ and $x=e, f$ for each $i_k$. 
Then $(\ms{L}(\la),\ms{B}(\la))$ is a crystal base of $V(\la)$.
\end{thm}
\pf It follows from Lemma \ref{lem:comparison with crystal for super case} and \cite[Theorem 5.1]{BKK}. 
\qed\smallskip

An explicit combinatorial description of $\ms{B}(\la)$ can be found in \cite{BKK}.
As in the case of crystal bases of $U_q(\gl(m|n))$-modules \cite[Proposition 2.8]{BKK}, we have a tensor product theorem for crystal bases of $\U$-modules as follows (see \cite[Proposition 3.4]{KO} and \cite[Section 2.2]{KY} for $\U(\gl(\e))$-modules)).
We remark that the map $\tau$ in \eqref{eq:iso tau standard} does not preserve the comultiplications, so we may not obtain this directly from \cite[Proposition 2.8]{BKK}.

\begin{prop}\label{prop:tensor product rule}
Let $V_1, V_2\in \cO_{\geq 0}$ be given. Suppose that $(L_k,B_k)$ is a crystal base of $V_i$ for $k=1,2$. 
Then $(L_1\otimes L_2, B_1\otimes B_2)$ is a crystal base of $V_1\otimes V_2$, where $B_1\otimes B_2\subset (L_1/qL_1)\otimes (L_2/qL_2)=(L_1\otimes L_2)/(qL_1\otimes L_2)$. 
Moreover, for $i\in I$, $\te_i$ and $\tf_i$ act on $B_1\otimes B_2$ as follows:
\begin{itemize}
\item[(1)] 
if $i=m$, then
{\allowdisplaybreaks
\begin{equation*}\label{eq:tensor product rule for odd +}
\begin{split}
\te_m(b_1\otimes b_2)=&
\begin{cases}
\te_m  b_1\otimes b_2, & \text{if }\langle {\rm wt}(b_1),\alpha^\vee_m\rangle>0, \\ 
 b_1\otimes  \te_m b_2, & \text{if }\langle {\rm wt}(b_1),\alpha^\vee_m\rangle=0,
\end{cases}
\\
\tf_m(b_1\otimes b_2)=&
\begin{cases}
 \tf_m b_1\otimes b_2, & \text{if }\langle {\rm wt}(b_1),\alpha^\vee_m\rangle >0, \\ 
b_1\otimes  \tf_m b_2, & \text{if }\langle {\rm wt}(b_1),\alpha^\vee_m\rangle=0,
\end{cases}
\end{split}
\end{equation*}}
\item[(2)] 
if $i< m$, then
{\allowdisplaybreaks
\begin{equation*}\label{eq:tensor product rule for even -}
\begin{split}
&\te_i(b_1\otimes b_2)= \begin{cases}
\te_i b_1 \otimes b_2, & \text{if $\varphi_i(b_1)\geq\varepsilon_i(b_2)$}, \\ 
 b_1 \otimes \te_ib_2, & \text{if $\varphi_i(b_1)<\varepsilon_i(b_2)$},\\
\end{cases}
\\
&\tf_i(b_1\otimes b_2)=
\begin{cases}
\tf_ib_1 \otimes  b_2, & \text{if $\varphi_i(b_1)>\varepsilon_i(b_2)$}, \\
b_1 \otimes \tf_i  b_2, & \text{if $\varphi_i(b_1)\leq\varepsilon_i(b_2)$}, 
\end{cases}
\end{split}
\end{equation*}}

\item[(3)] 
if $i>m$, then
{\allowdisplaybreaks
\begin{equation*}\label{eq:tensor product rule for even +}
\begin{split}
&\te_i(b_1\otimes b_2)= \begin{cases}
 b_1 \otimes \te_ib_2, & \text{if $\varphi_i(b_2)\geq\varepsilon_i(b_1)$}, \\ 
s_i \te_ib_1 \otimes  b_2, & \text{if $\varphi_i(b_2)<\varepsilon_i(b_1)$},\\
\end{cases}
\\
&\tf_i(b_1\otimes b_2)=
\begin{cases}
 b_1 \otimes \tf_ib_2, & \text{if $\varphi_i(b_2)>\varepsilon_i(b_1)$}, \\
 s_i \tf_ib_1 \otimes b_2, & \text{if $\varphi_i(b_2)\leq\varepsilon_i(b_1)$}, 
\end{cases}
\end{split}
\end{equation*}}
where $s_i=(-1)^{({\rm wt}(b_1)|\alpha_i)}$.
\end{itemize}
Here we put $\varepsilon_i(b)=\max\{k\geq 0 \,|\ \te_i^k b  \neq 0 \}$ and $\varphi_i(b)=\max\{k\geq 0 \,|\ \tf_i^k b \neq 0 \}$ for $b\in B_1, B_2$.
\end{prop}

\subsection{$q$-deformed Kac module $K(\la)$}

Let $\K=\K_{m|n}$ be the subalgebra of $\U^-$ generated by $\bff_\alpha$ $(\alpha\in \Phi_1^+)$.
For $S\subset \Phi^+_{\ov 1}$ with $S=\{\,\beta_1\prec \cdots \prec \beta_r\,\}$, we put
\begin{equation}\label{eq: monomial F_S}
{\bf f}_S = {\bf f}_{\beta_1}\cdots {\bf f}_{\beta_r},
\end{equation}  
where we assume that ${\bf f}_S=1$ when $S=\emptyset$.
Then $B_{K}=\{\,{\bf f}_S\,|\,S\subset \Phi^+_{\ov 1}\,\}$ is a $\Bbbk$-basis of $\mc K$, and 
\begin{equation}\label{eq:decomp of U^-}
 \U^- \cong \K\otimes \U^-_{m,n}\cong  \K \ot \U^-_{m|0}\ot \U^-_{0|n},
\end{equation}
as a $\Bbbk$-space by Proposition \ref{prop:PBW basis}, where the isomorphism is given by multiplication.

Let $\lambda\in P^+$ be given. 
Let $V_{m,n}(\lambda)$ be the irreducible $\U_{m,n}$-module with highest weight $\lambda$, and let $V_{m|0}(\la_+)$ (resp. $V_{0|n}(\la_-)$) the irreducible highest weight module over $\U_{m|0}$ (resp. $\U_{0|n}$) with highest weight $\la_+$ (resp. $\la_-$). 
Note that $V_{m,n}(\la)\cong V_{m|0}(\la_+)\otimes V_{0|n}(\la_-)$ as a $\U_{m|0}\ot\U_{0|n}$-module since $\U_{m,n}\cong \U_{m|0}\ot\U_{0|n}$. 

Let $\mc{P}$ be the subalgebra of $\U$ generated by $\U_{m,n}$ and $e_m$, and extend $V_{m,n}(\lambda)$ to a $\mc{P}$-module in an obvious way. 
Then we define 
\begin{equation*}
K(\lambda) = \U\otimes_{\mc{P}}V_{m,n}(\lambda).
\end{equation*}
We call $K(\la)$ the ($q$-deformed) {\it Kac module with highest weight $\la$}. Let $1_\la$ denote a highest weight vector of $K(\la)$ with weight $\la$. 
Since $\mc{P}\cong \U^-_{m,n}\ot \U^0 \ot \U^+$ as a $\Bbbk$-space,  
we have as a $\Bbbk$-space
\begin{equation} \label{eq: decomposition of Kac module}
 K(\la) \cong \K\ot V_{m,n}(\la)\cong \K\ot V_{m|0}(\la_+)\otimes V_{0|n}(\la_-).
\end{equation} 

\subsection{Crystal base of $K(\la)$}

Let $\la\in P^+$ be given. The $\Bbbk$-linear map $e'_m$ on $\U^-$ in \eqref{eq: derivation} induces a $\Bbbk$-linear map on $K(\la)$ in \cite[Section 4.2]{K14}.

We define a crystal base of $K(\la)$ in the same way as in Definition \ref{def:crystal for poly} where $\tilde{e}_i$ and $\tilde{f}_i$ on $V$ for $i\in I$ are given by the ones in Section \ref{eq:crystal base for homogeneous case} for $i\neq m$ and 
\begin{equation}\label{eq:eM fM for K}
\te_m u = e'_m(u), \ \ \ \tf_m u = f_m u.
\end{equation}
Then we have the following by \cite[Theorem 4.7 and Corollary 4.9]{K14}.
\begin{thm}\label{thm: crystal base of Kac module}
Let
\begin{equation*}
\begin{split}
\ms{L}({K(\lambda)})&=\sum_{r\geq 0,\, k_1,\ldots,k_r\in I}A_0 \tilde{x}_{k_1}\cdots\tilde{x}_{k_r}1_\lambda, \\
\ms{B}(K(\lambda))&=\{\,\pm \,{\tilde{x}_{k_1}\cdots\tilde{x}_{k_r}1_\lambda}\!\!\! \pmod{q\ms{L}({K(\lambda)})}\,|\,r\geq 0, k_1,\ldots,k_r\in I\,\}\setminus\{0\},
\end{split}
\end{equation*}
where $x=e, f$ for each $k_i$.
Then $\left(\ms{L}({K(\lambda)}),\ms{B}(K(\lambda))\right)$ is a crystal base of $K(\lambda)$. 
\end{thm}
\pf We can check that $K(\la)^\tau$ is isomorphic to a Kac-module over $U_q(\gl(m|n))$ \cite[Section 2.5]{K14}. It is shown \cite[Theorem 4.7 and Corollary 4.9]{K14} that the pair $(L,B)$ is a crystal base of $K(\la)^\tau$ in the sense of \cite[Definition 4.6]{K14}, where
\begin{equation*}\label{eq:crystal base of Kac module}
\begin{split}
L&=\sum_{r\geq 0,\, k_1,\ldots,k_r\in I}A_0 \tilde{X}_{k_1}\cdots\tilde{X}_{k_r}1_\lambda, \\
B&=\{\,\pm \,{\tilde{X}_{k_1}\cdots\tilde{X}_{k_r}1_\lambda}\!\!\! \pmod{qL}\,|\,r\geq 0, k_1,\ldots,k_r\in I\,\}\setminus\{0\},
\end{split}
\end{equation*}
with $X=E, F$ for each $k_i$. 
    By the same argument as in the proof of Lemma \ref{lem:comparison with crystal for super case}, $(L,B)$ is a crystal base of $K(\la)$, where we can replace $\tilde{E}_i$ and $\tilde{F}_i$ with $\te_i$ and $\tf_i$, respectively.
\qed\smallskip

\begin{cor}\label{cor:crystal base of K(0)}
Under the above hypothesis, we have as an $A_0$-module
\begin{equation*}
\ms{L}(K(\la)) \cong \ms{L}(K(0))\ot \ms{L}_{m|0}(\la_+)\ot \ms{L}_{0|n}(\la_-), 
\end{equation*}
where $\ms{L}(K(0))$ is the $A_0$-span of $\bff_S  1_0$ for $S\subset \Phi^+_{\ov 1}$.
\end{cor}
\pf It follows from \cite[Theorem 4.7]{K14} and the uniqueness of the crystal base of $K(\la)$ \cite[Theorem 4.10]{K14}.
\qed\smallskip

By Corollary \ref{cor:crystal base of K(0)}, we may identify $\ms{L}(K(0))$ and $\ms{B}(K(0))$ with 
\begin{equation*}\label{eq:LK BK}
 \ms{L}(\K):=\bigoplus_{S\subset \Phi^+_{\ov 1}}A_0\bff_S,\quad \ms{B}(\K):=\{\,\pm {\bf f}_S \pmod{q\ms{L}(\K)}\,|\,S\subset \Phi^+_{\ov 1}\,\},
\end{equation*}
as an $A_0$-module and its signed $\Q$-basis, respectively. 
We can also deduce the following from \cite[Theorem 4.11]{K14}.
\begin{thm}\label{main result - compatibility}
 Let $\lambda\in \cP_{m|n}$ be given. Let $\pi_\lambda : K(\la) \longrightarrow V(\la)$ be the $\U$-module homomorphism such that $\pi_\la(1_\la)=v_\la$.
Then  
\begin{enumerate}
\item[\em (1)] $\pi_\la(\ms{L}(K(\la)))=\ms{L}(\la)$,

\item[\em (2)] $\ov{\pi}_\la(\ms{B}(K(\la)))= \ms{B}(\la) \cup\{0\}$, where $\ov{\pi}_\lambda : \ms{L}(K(\la))/q\ms{L}(K(\la)) \rightarrow \ms{L}(\la)/q\ms{L}(\la)$ is the induced $\Q$-linear map,

\item[\em (3)] $\ov{\pi}_\la$ restricts to a weight preserving bijection 
$$\ov{\pi}_\la : \{\,b\in \ms{B}(K(\la))\,|\,\ov{\pi}_\la(b)\neq 0 \,\} \longrightarrow \ms{B}(\la),$$ 
which commutes with $\te_i$ and $\tf_i$ for $i\in I$.
\end{enumerate}
\end{thm}

\section{Limit of crystals of Kac-modules}\label{sec:limit of Kac crystal}

\subsection{$\U$-crystals} \label{subsec: U-crystals}
Let us give some additional terminologies and conventions.
A {\em $\U$-crystal} is a set 
$B$ together with the maps ${\rm wt} : B \rightarrow P$,
$\varepsilon_i, \varphi_i: B \rightarrow \mathbb{Z}\cup\{-\infty\}$ and
$\te_i, \tf_i: B \rightarrow B\cup\{ {\bf 0} \}$ for $i\in I$ such that for $b\in B$,
\begin{itemize}
\item[(1)]  
$\varphi_i(b) =\langle {\rm wt}(b),\alpha^\vee_i \rangle +
\varepsilon_i(b)$  ($i\neq m$) and  
$\varphi_m(b) + \varepsilon_m(b)\in\{\,0,1\,\}$,

\item[(2)] $\varepsilon_i(\te_i b) = \varepsilon_i(b) - 1,\ \varphi_i(\te_i b) =
\varphi_i(b) + 1,\ {\rm wt}(\te_ib)={\rm wt}(b)+\alpha_i$ if $\te_i b \in B$,

\item[(3)] $\varepsilon_i(\tf_i b) = \varepsilon_i(b) + 1,\ \varphi_i(\tf_i b) =
\varphi_i(b) - 1,\ {\rm wt}({\tf_i}b)={\rm wt}(b)-\alpha_i$ if $\tf_i b \in B$,

\item[(4)] $\tf_i b = b'$ if and only if $b = \te_i b'$ for $b' \in B$,

\item[(5)] $\te_ib=\tf_ib={\bf 0} $ when $\varphi_i(b)=-\infty$,
\end{itemize}
where ${\bf 0}$ is a formal symbol and $-\infty$ is the smallest
element in $\Z\cup\{-\infty\}$ such that $-\infty+n=-\infty$
for all $n\in\Z$. As usual, a $\U$-crystal becomes an $I$-colored oriented graph, where $b\stackrel{i}{\rightarrow}b'$ if and only if $b'=\tf_{i}b$ for $b, b'\in B$ and $i\in I$.

For example, let $(L,B)$ be a crystal base of a $\U$-module $V=V(\la)$ or $K(\la)$ in Sections \ref{sec:polynomial repn}. Then $B/\{\pm 1\}$ is a $\U$-crystal with $\varepsilon_i(b)=\max\{\,k\,|\,\te_i^kb\neq {\bf 0}\,\}$ and $\varphi_i(b)=\max\{\,k\,|\,\tf_i^kb\neq {\bf 0}\,\}$ for $i\in I$ and  $b\in B/\{\pm 1\}$, where ${\bf 0}$ is the zero vector.
For simplicity, we often denote by $B$ (instead of $B/\{\pm 1\}$) the crystal associated to a crystal base $(L,B)$, where $B$ is a signed basis of $L/qL$.
For $\la \in P$, we denote by $T_{\la}=\{ t_\la \}$ be a $\U$-crystal such that ${\rm wt}(t_\la) = \la $, $\te_i t_\la = \tf_i t_\la = {\bf 0}$ and $\varepsilon_i(t_\la)=\varphi_i(t_\la)=-\infty$ for $i \in I$.

Let $B_1$ and $B_2$ be $\U$-crystals.
For $b_i \in B_i$ ($i=1,2$), we say that $b_1$ is equivalent to $b_2$, and write $b_1 \equiv b_2$ if there is an isomorphism of $\U$-crystals $\phi : C(b_1) \rightarrow C(b_2)$ such that $\phi(b_1) = b_2$, where $C(b_i)$ is the connected component of $b_i$ in $B_i$ for $i = 1, 2$.

\subsection{Crystal $\ms{B}(K(\la))$} \label{subsec: BKla}
Let us recall the crystal structure of $\ms{B}(K(\la))$ for $\la\in P^+$ \cite{K14}.
By Corollary \ref{cor:crystal base of K(0)}, we have a bijection
\begin{equation} \label{eq: iso bkla}
 \xymatrixcolsep{2pc}\xymatrixrowsep{0pc}\xymatrix{
 \ms{B}({K(\lambda)})   \ \ar@{->}[r] &\ \cP(\Phi^-_{\ov 1})\times \ms{B}_{m|0}(\la_+) \times \ms{B}_{0|n}(\la_-) \\
 \bff_S \ot b_+\ot b_- \ \ar@{|->}[r] &\ (-S,b_+,b_-)
 },
\end{equation}
where $\cP(\Phi^-_{\ov 1})$ is the power set of $\Phi^-_{\ov 1}=-\Phi^+_{\ov 1}$, and $-S=\{-\beta\,|\,\beta\in S\}$ for $S\subset \Phi^+_{\ov 1}$. 
We identify $\cP(\Phi^-_{\ov 1})$ with $\ms{B}(\K)$.

Let $\prec$ denote the linear order in Section \ref{subsec: PBW basis of U-} restricted on $\Phi^+_{\ov 1}$, and 
let $\prec'$ be a linear order on $\Phi^+_{\ov 1}$ such that $\alpha \prec' \beta$ if and only if $(a>c)$ or $(a=c, \ b>d)$.
for $\alpha, \beta\in \Phi^+_{\ov 1}$ with $\alpha=\de_a-\de_b$ and $\beta=\de_c-\de_d$. For $\alpha,\beta\in \Phi^-_{\ov 1}$, we define $\alpha\prec \beta$ (resp. $\alpha\prec' \beta$) if and only if $-\alpha\prec-\beta$ (resp. $-\alpha\prec'-\beta$). 
 
Let $S\in \cP(\Phi^-_{\ov 1})$ be given with $S=\{\,\beta_1\prec\ldots\prec\beta_r\,\}=\{\,\beta'_1\prec'\ldots\prec'\beta'_r\,\}$. 
Then $\te_i S$ and $\tf_iS$ for $i\in I$ is given as follows:

{\em Case 1}. For $i=m$, we have
\begin{equation*}\label{tilde e0  f0 action on K(0)}
\begin{split}
\te_m S&=
\begin{cases}
{S\setminus\{-\alpha_m\}} & \text{if $-\alpha_m\in S$}, \\
{\bf 0} & \text{if $-\alpha_m\not\in S$},
\end{cases} \ \  \ \
\tf_m S=
\begin{cases}
{S\cup\{-\alpha_m\}}  & \text{if $-\alpha_m\not\in S$}, \\
{\bf 0} & \text{if $-\alpha_m\in S$}.
\end{cases}
\end{split}
\end{equation*}
Here we understand $0$ as the zero vector in $\ms{L}(\K)/q\ms{L}(\K)$.

{\em Case 2}.
Suppose that $i\neq m$. First, we have for $k=1,\ldots,r$ 
\begin{equation*}
\begin{split}
\te_i \beta_k &=
\begin{cases}
\beta_k+\alpha_i & \text{if $\beta_k+\alpha_i\in \Phi^-_{\ov 1}$}, \\
{\bf 0} & \text{otherwise},
\end{cases} \ \ \ 
\tf_i \beta_k =
\begin{cases}
\beta_k-\alpha_i  & \text{if $\beta_k-\alpha_i\in \Phi^-_{\ov 1}$}, \\
{\bf 0} & \text{otherwise}.
\end{cases}
\end{split}
\end{equation*}
Next we identify $S$ with $\beta_1\ot\dots\ot\beta_r$ when $i<m$ and with $\beta'_1\ot\dots\ot\beta'_r$ when $i>m$.
Then we define $\te_i S$ and $\tf_iS$ by the tensor product rule given in Proposition \ref{prop:tensor product rule}.
Indeed, we have
\begin{equation} \label{eq: BK}
	\cP(\Phi^-_{\ov 1}) \cong \bigsqcup_{\substack{\ell(\la) \leq m \\ \ell(\la^t) \leq n}} \ms{B}_{m|0}(\la) \times \ms{B}_{0|n}(\la^t),
\end{equation}
as a $(\U_{m|0},\U_{0|n})$-bicrystal, where $\ell(\la)$ is the length and $\la^t$ is the transpose of the partition $\la$ (see \cite{DK,K07}).

Now, the crystal structure on $\ms{B}({K(\lambda)})$ can be described as follows.
\begin{prop}[{\cite[Proposition 5.1]{K14}}] \label{prop: description of BKla}
Let $(S,b_+,b_-)\in \ms{B}({K(\lambda)})$ be given. For $i\in I$ and $x=e,f$, we have
\begin{equation*}
\widetilde{x}_i(S,b_+,b_-)=
\begin{cases}
(S',b'_+,b_-) & \text{if $i<m$ and $\widetilde{x}_i(S
\otimes b_+ )=S'\otimes b'_+$,} \\
(S'',b_+,b_-'') & \text{if $i>m$ and $\widetilde{x}_i(S
\otimes b_- )=S''\otimes b''_-$},\\
(\widetilde{x}_mS, b_+,b_-) & \text{if $i=m$},
\end{cases}
\end{equation*}
where we assume that $\widetilde{x}_i(S,b_+,b_-)={\bf 0}$ if any
of its components on the right-hand side is ${\bf 0}$.
\end{prop}

Recall from \cite{Kas93} that $\ms{B}_{m|0}(\la_+)$ and $\ms{B}_{0|n}(\la_-)$ can be viewed as subscrystals of $\ms{B}_{m|0}(\infty) \ot T_{\la_+}$ and $\ms{B}_{0|n}(\infty)\ot T_{\la_-}$, respectively as follows:
	\begin{equation} \label{eq: iso 0nla}
	\begin{split}
		\ms{B}_{m|0}(\la_+) &\simeq \{ \, b_+ \ot t_{\la_+} \in \ms{B}_{m|0}(\infty) \ot T_{\la_+} \, | \, \varepsilon_i^*(b_+) \leq \langle \la_+ , \alpha_i^\vee \rangle \text{  for } 1 \leq i < m  \, \}, \\
	\ms{B}_{0|n}(\la_-) &\simeq \{ \, b_- \ot t_{\la_-} \in \ms{B}_{0|n}(\infty) \ot T_{\la_-} \, | \, \varepsilon_i^*(b_-) \leq \langle \la_- , \alpha_i^\vee \rangle \text{  for } m < i \leq \ell-1 \, \},
	\end{split}
	\end{equation}
where $*$ denotes the involution on $\ms{B}_{m|0}(\infty)$ and $\ms{B}_{0|n}(\infty)$ \cite[Theorem 2.1.1]{Kas93} and $\varepsilon_i^*(b)=\varepsilon_i(b^*)$.

\subsection{Crystal $\ms{B}(K(\infty))$}\label{subsec:limit of B(K(la))}

For $\la,\mu\in P^+$, we define $\la<\mu$ if and only if $\mu-\la=\nu\in P^+$.

Let $b=(S,b_+,b_-)\in \ms{B}({K(\lambda)})$ be given. 
We observe that
\begin{equation}\label{eq:theta embedding}
 b_+={X}v_{\la_+},\quad S\ot b_-={Y}(S_0\ot v_{\la_-}),
\end{equation}
where ${X}$ is a product of $\tf_i$'s for $i<m$, ${Y}$ is a product of $\tf_i$'s for $i>m$, and $S_0\subset \Phi^-_{\ov 1}$ such that $\te_i(S_0\ot v_{\la_-})={\bf 0}$ for all $i>m$. 
Here we regard $S\in \ms{B}_{0|n}(\eta) \subset \ms{B}(\K)$ for some $\eta\in P^+$ as an element of a $\U_{0|n}$-crystal, and hence $S\ot b_-\in \ms{B}_{0|n}(\xi)\subset \ms{B}_{0|n}(\eta)\ot \ms{B}_{0|n}(\la_-)$ for some $\xi\in P^+$.
By Proposition \ref{prop: description of BKla}, we may write
\begin{equation}\label{eq:theta embedding 2}
 b=(S,b_+,b_-)=Y(S_0,Xv_{\la_+},v_{\la_-}).
\end{equation}

For $\la < \mu$, we define a map
\begin{equation*} \label{eq: Theta map}
 \xymatrixcolsep{2pc}\xymatrixrowsep{0pc}\xymatrix{
 \Theta_{\la,\mu} : \ms{B}({K(\la)})  \ \ar@{->}[r] &\ \ms{B}({K(\mu)}) \\
 \quad (S,b_+,b_-) \ \ar@{|->}[r] &\ (S',b'_+,b'_-)
 },
\end{equation*}
by $b'_+={X}v_{\mu_+}$ and $S'\ot b'_-={Y}(S_0\ot v_{\mu_-})$ where ${X}$, $S_0$ and ${Y}$ are given in \eqref{eq:theta embedding}. 
Note that $\Theta_{\la,\mu}$ is a well-defined injective map, and 
for $\la<\mu<\nu$
\begin{equation}\label{eq: transitive of Theta}
 \Theta_{\mu,\nu}\circ\Theta_{\la,\mu}=\Theta_{\la,\nu}.
\end{equation}

\begin{lem} \label{lem: crystal equiv on K0}
Let $i,j\in I$ such that $i < m < j$.
Then we have
\begin{enumerate}[\em (1)]
	\item $\tf_i \tf_j S = \tf_j \tf_i S$,
	
	\item if $\tf_i S \neq 0$, then $\tf_i S \equiv S$ as elements of $\U_{0|n}$-crystals,
	\vskip 1mm
	
	\item if $\tf_j S \neq 0$, then $\tf_j S \equiv S$ as elements of $\U_{m|0}$-crystals.
\end{enumerate}
\end{lem}
\pf We may regard $\ms{B}(\K)=\cP(\Phi^-_{\ov 1})$ as the set of $m\times n$ binary matrices, where $S\in \cP(\Phi^-_{\ov 1})$ is identified with the matrix $M=(m_{ab})$ ($1\le a\le m<b\le \ell$) with $m_{ab}=1$ if and only if $-\de_a+\de_b\in S$. Then we may apply the $(\gl_m,\gl_n)$-bicrystal structure on $\ms{B}(\K)$ (see \cite{DK,K07}).
\qed 

\begin{lem}\label{lem:embedding Theta}
Under the above hypothesis, we have 
\begin{enumerate}[\em (1)]
\item if $\tf_ib\neq {\bf 0}$ for some $i\neq m$, then 
$\Theta_{\la,\mu}(\tf_ib)=\tf_i\Theta_{\la,\mu}(b)$,

\item if $\tf_mb\neq {\bf 0}$ and $\tf_m S_0 \neq 0$, then $\Theta_{\la,\mu}(\tf_mb)=\tf_m \Theta_{\la,\mu}(b)\neq {\bf 0}$ for all $\la<\mu$,

\item if $\tf_mb\neq {\bf 0}$ and $\tf_m S_0 = 0$, then there exists $M_b\in \Z_+$ such that $\tf_m \Theta_{\la,\mu}(b) = {\bf 0}$ for $\mu$ with $\langle \mu,\alpha^\vee_{m+1}\rangle>M_b$.

\end{enumerate}
\end{lem}
\pf 
(1) Let $b=(S,b_+,b_-)$ be given such that $\tf_ib\neq {\bf 0}$. 
We have $b=Y(S_0,Xv_{\la_+},v_{\la_-})$ by \eqref{eq:theta embedding}.
If $i>m$, then it follows immediately from Proposition \ref{prop: description of BKla} and the definition of $\Theta_{\la,\mu}$ that $\Theta_{\la,\mu}(\tf_ib)=\tf_i\Theta_{\la,\mu}(b)$.
So we assume that $i < m$.

Suppose that $\tf_i \left( S \ot b_+ \right) = \left( \tf_i S \right) \ot b_+$, which is equivalent to 
\begin{equation} \label{eq:fi S0 ot la+}
	\tf_i(S_0\ot X v_{\la_+})=\left(\tf_iS_0\right) \ot  X v_{\la_+}
\end{equation}
since $S\equiv S_0$ by Lemma \ref{lem: crystal equiv on K0}(3).
Since embedding of $\ms{B}_{m|0}(\la_+)$ into $\ms{B}_{m|0}(\mu_+) \ot T_{\la_+ - \mu_+}$ is $e$-strict
(see \cite[Lemma 7.1.2]{Kas00}), we have $\varepsilon_i(b_+) = \varepsilon_i(b_+')$, and hence
\begin{equation} \label{eq:fi S0 ot mu}
	\tf_i \left( S_0 \ot Xv_{\mu_+} \right) = \left( \tf_i S_0 \right) \ot Xv_{\mu_+}.
\end{equation}
Since
$\tf_ib
  = Y(\tf_iS_0,Xv_{\la_+},v_{\la_-})$ by Lemma \ref{lem: crystal equiv on K0} and \eqref{eq:fi S0 ot la+},
it follows from Lemma \ref{lem: crystal equiv on K0} and \eqref{eq:fi S0 ot mu} that
\begin{align*}
 \Theta_{\la,\mu}(\tf_ib)&
  = Y(\tf_i(S_0,Xv_{\mu_+},v_{\mu_-}))
  = \tf_i(Y(S_0,Xv_{\mu_+},v_{\mu_-}))
  =\tf_i\Theta_{\la,\mu}(b).
\end{align*}
The proof for the case when $\tf_i \left( S \ot b_+ \right) = S \ot \left( \tf_ib_+ \right)$ is similar.

(2) Suppose that $b=Y(S_0,Xv_{\la_+},v_{\la_-})=(Y_1S_0,Xv_{\la_+},Y_2v_{\la_-})$ for some $Y_1$ and $Y_2$, which are products of $\tf_i$'s for $i > m$.

Since $\tf_mb\neq {\bf 0}$ and $\tf_m S_0\neq {\bf 0}$, that is, $-\alpha_m \notin S$ and $-\alpha_m \notin S_0$, we see from Propositions \ref{prop:tensor product rule}(3), \ref{prop: description of BKla} and the crystal structure on $\cP(\Phi^-_{\ov 1})$ that $Y_1\tf_mS_0=\tf_mY_1S_0$. Hence
\begin{equation}\label{eq:f_m commuting}
\begin{split}
 \tf_m (Y_1S_0,Xv_{\la_+},Y_2v_{\la_-})
 &=(\tf_mY_1S_0,Xv_{\la_+},Y_2v_{\la_-}) \\
 &=(Y_1\tf_mS_0,Xv_{\la_+},Y_2v_{\la_-})\\
 &=Y(\tf_mS_0,Xv_{\la_+},v_{\la_-}),
\end{split}
\end{equation}
and the equation \eqref{eq:f_m commuting} is also true when $\la_{\pm}$ is replaced by $\mu_\pm$.
Since $\te_i((\tf_mS_0)\ot v_{\la_-})={\bf 0}$ for all $i>m$, we have by \eqref{eq:f_m commuting}
\begin{align*}
 \Theta_{\la,\mu}(\tf_m b)
 =\Theta_{\la,\mu}(Y(\tf_mS_0,Xv_{\la_+},v_{\la_-}))
 =Y(\tf_mS_0,Xv_{\mu_+},v_{\mu_-})
 =\tf_m\Theta_{\la,\mu}(b).
\end{align*}

(3) We may take $\varphi_{m+1}(v_{\mu_-}) = \langle \mu_- ,\alpha^\vee_{m+1}\rangle$ to be large enough so that ${Y}(S_0 \otimes v_{\mu_-}) = ({Y}_1 S_0) \otimes ({Y}_2 v_{\mu_-})$ for some $Y_1$ and $Y_2$, which are products of $\tf_i$'s for $i > m$, and $Y_1$ has no factor $\tf_{m+1}$ (recall Proposition \ref{prop:tensor product rule}).
Since $\tf_mS_0={\bf 0}$ and $Y_1$ has no factor $\tf_{m+1}$, we have $\tf_m Y_1S_0=Y_1\tf_mS_0={\bf 0}$ and hence
\begin{equation*}
 \tf_m \Theta_{\la, \mu}(b)
 = \tf_m (Y_1S_0,Xv_{\mu_+},Y_2v_{\mu_-})
 =  (\tf_m Y_1S_0,Xv_{\mu_+},Y_2v_{\mu_-})
 =  {\bf 0}.
\end{equation*}
\qed

\begin{rem}
{\rm Let $b\in \ms{B}(K(\la))$ be such that $\tf_m b\neq {\bf 0}$. Due to Lemma \ref{lem:embedding Theta}(3), it may happen $\tf_m \Theta_{\la,\mu}(b)={\bf 0}$ for some $\mu > \la$, and hence we have $\Theta_{\la,\mu}(\tf_m b)\neq \tf_m \Theta_{\la,\mu}(b)$ in general.
}
\end{rem}

Let $b^\la \in \ms{B}(K(\la))$ and $b^\mu \in \ms{B}(K(\mu))$ ($\la,\mu \in P^+$) be given. We write $b^\la \sim b^\mu$ if there exists $\nu \in P^+$ such that
$\la,\mu<\nu$ and $\Theta_{\la,\nu}(b^\la) = \Theta_{\mu,\nu}(b^\mu)$. It is not difficult to see that $\sim$ defines an equivalence relation on 
$\bigsqcup_{\la\in P^+}\ms{B}(K(\la))$.

\begin{lem} \label{lem: equivalence condition}
For $b^\la \in \ms{B}(K(\la))$ and $b^\mu \in \ms{B}(K(\mu))$, let $b^\la=Y^\la (S_0^\la,X^\la v_{\la_+},v_{\la_-})$, $b^\mu=Y^\mu (S_0^\mu,X^\mu v_{\mu_+},v_{\mu_-})$ as in \eqref{eq:theta embedding 2}.
Then $b^\la \sim b^\mu$ is equivalent to
\begin{equation*} \label{eq: equivalence condition}
	 S^\la_0 = S^\mu_0 , \quad X^\la u_+ = X^\mu u_+, \quad Y^\la u_- = Y^\mu u_-,
\end{equation*}
where $u_\pm$ denotes the highest weight elements $1$ in $\ms{B}_{m|0}(\infty)$ and $\ms{B}_{0|n}(\infty)$, respectively.
\end{lem}
\pf
Suppose $b^\la \sim b^\mu$. Then there exists $ \nu' > \la, \mu $ such that $\Theta_{\la,\nu'}(b^\la) = \Theta_{\mu,\nu'}(b^\mu)$.
Choose $ \nu \gg \nu' $ (that is, $\nu-\nu'$ is sufficiently large) such that 
$$
\Theta_{\gamma,\nu}(b^\gamma) = Y^\gamma (S^\gamma_0, X^\gamma v_{\nu_+}, v_{\nu_-})= (S^\gamma_0, X^\gamma v_{\nu_+}, Y^\gamma v_{\nu_-}) \quad (\gamma = \la, \mu).
$$ 
By \eqref{eq: transitive of Theta}, we have $\Theta_{\la,\nu}(b^\la) = \Theta_{\mu,\nu}(b^\mu)$, that is,   
$(S^\la_0, X^\la v_{\nu_+}, Y^\la v_{\nu_-}) = (S^\mu_0, X^\mu v_{\nu_+}, Y^\mu v_{\nu_-})$. 
Hence we have $X^\la u_+ = X^\mu u_+$ and $Y^\la u_- = Y^\mu u_-$ by \eqref{eq: iso 0nla}, and $S^\la_0=S^\mu_0$.
The proof of the converse is similar. \qed\smallskip
 
Let $\ms{B}(K(\infty))$ be the set of equivalence classes with respect to $\sim$.
For $b \in \ms{B}(K(\infty))$, let us write $b = (b^\la)_{\la \in P^+}$, where 
\begin{enumerate}
	\item $b^\la\in \ms{B}(K(\la))\cup\{{\bf 0}\}$ for all $\la \in P^+$,
	\item $b^\mu\sim b^\nu$ for non-zero $b^\mu$ and $b^\nu$.
\end{enumerate}

Let $b=(b^\la)_{\la\in P^+} \in \ms{B}(K(\infty))$ be given.
Note that $b^\la\neq {\bf 0}$ for some $\la\in P^+$.  
We define ${\rm wt}(b)= {\rm wt}(b^\la)-\la$ for some $\la$ with $b^\la\neq 0$.
For $i\in I$, we define $\te_i b$ and $\tf_i b$ as follows:
\begin{itemize}
 \item[(1)] Suppose that $i\neq m$. We define $\tilde{x}_i b = \left(\tilde{x}_i b^\la\right)_{\la\in P^+}$ $(x=e,f)$, where we assume that $\tilde{x}_i b={\bf 0}$ if $\tilde{x}_i b^\la={\bf 0}$ for all $\la\in P^+$. 

 \item[(2)] Suppose that $i=m$. We define $\tf_m b$ to be the equivalence class of $\tf_m b^\mu$ if there exists $\mu\in P^+$ such that $\tf_m b^\nu\neq {\bf 0}$ for all $\nu\ge\mu$, and define $\tf_m b ={\bf 0}$ otherwise.
We define $\te_m b=b'$ if $\tf_m b'=b$ for some $b\in \ms{B}(K(\infty))$, and $\te_m b={\bf 0}$, otherwise.
 
\end{itemize}
Put
\begin{equation}\label{eq:ep phi}
\varepsilon_i(b)=\max\{\,k\,|\,\te_i^kb\neq {\bf 0}\,\},\quad
 \varphi_i(b) =
 \begin{cases}
 \langle {\rm wt}(b),\alpha^\vee_i \rangle + \varepsilon_i(b) & \text{for $i\neq m$},\\
 \max\{\,k\,|\,\tf_i^kb\neq {\bf 0}\,\} & \text{for $i= m$}.
 \end{cases}
\end{equation}

\begin{lem}\label{lem:B(K(infty)) crystal}
Under the above hypothesis, $\ms{B}(K(\infty))$ is a well-defined $\U$-crystal. 
\end{lem}
\pf 
 The well-definedness of $\te_i, \tf_i$ for $i\neq m$ follows from Lemma \ref{lem:embedding Theta}(1), while $\te_m, \tf_m$ are well-defined by Lemma \ref{lem:embedding Theta}(2) and (3). The other conditions are easy to check.
\qed
\smallskip

Now, we give another description of $\ms{B}(K(\infty))$.
Let $b=(b^\la)_{\la\in P^+}\in \ms{B}(K(\infty))$ be given and choose $\la\in P^+$ such that $b^\la =Y(S_0^\la,Xv_{\la_+},v_{\la_-}) \neq {\bf 0}$.
Then we define
\begin{equation*}\label{eq:kappa}
 \xymatrixcolsep{2pc}\xymatrixrowsep{0pc}\xymatrix{
 \kappa : \ms{B}({K(\infty)})  \ \ar@{->}[r] &\ \ms{B}(\K)\times \ms{B}_{m|0}(\infty) \times \ms{B}_{0|n}(\infty) \\
 \quad b \ \ar@{|->}[r] &\ (S,Xu_+,Yu_-)
 },
\end{equation*}
where $u_\pm$ denotes the highest weight elements $1$ in $\ms{B}_{m|0}(\infty)$ and $\ms{B}_{0|n}(\infty)$, respectively. 
It is well-defined by Lemma \ref{lem: equivalence condition}. Indeed $\kappa$ is a bijection since it is clearly surjective and injective by Lemma \ref{lem: equivalence condition}.

On the other hand, we define a $\U$-crystal structure on $\ms{B}(\K)\times \ms{B}_{m|0}(\infty) \times \ms{B}_{0|n}(\infty)$ as follows: 
Let $b=(S,b_+,b_-)$ be given. 
We let
${\rm wt}(b) = {\rm wt}(S) + {\rm wt}(b_+) + {\rm wt}(b_-)$,
and let $\varepsilon_i(b)$ and $\varphi_i(b)$ be as in \eqref{eq:ep phi}. 
For $i\in I$ and $x=e,f$, we define
\begin{equation}\label{eq:abstract crystal B(infty)}
\begin{split}
 \tilde{x}_i b &= 
 \begin{cases}
 (S',b'_+,b_-) & \text{if $i<m$ and $\tilde{x}_i (S\ot b_+)= S'\ot b'_+$},\\
 (S,b_+,\tilde{x}_ib_-) & \text{if $i>m$},\\
 (\tilde{x}_mS,b_+,b_-) & \text{if $i=m$},
 \end{cases}
\end{split}
\end{equation}
where $\tilde{x}_i (S\ot b_+)$ is defined by \eqref{eq:tensor product rule for Boson}.
Here we assume that $\tilde{x}_i b={\bf 0}$ if any of its component on the right-hand side is ${\bf 0}$.

\begin{thm}  \label{thm: isomorphism kappa}
The map $\kappa$ is an isomorphism of $\U$-crystals. 
Hence the crystal $\ms{B}(K(\infty))$ is isomorphic to $\ms{B}(\K)\times \ms{B}_{m|0}(\infty) \times \ms{B}_{0|n}(\infty)$.
\end{thm}
\pf It is enough to show that $\kappa$ is a homomorphism of $\U$-crystals.
Let $b \in \ms{B}(K(\infty))$ be given with $\kappa(b)=(S,Xu_+,Yu_-)$ for some $X$ and $Y$.
Choose a component $b^\la=(S,Xv_{\la_+},Yv_{\la_-})$ in $b$ such that $\la \gg 0$.\smallskip

If $i < m$, then we have $\kappa(\tilde{x}_i b) = \tilde{x}_i \kappa(b)$ ($x=e,f$) by \eqref{eq: iso 0nla}. 
If $i > m$, then by our choice of $\la$, we may assume that $\tilde{x}_ib^\la=(S,Xv_{\la_+},Yv_{\la_-})=(S,Xv_{\la_+},\tilde{x}_iYv_{\la_-})$, which implies $\kappa(\tilde{x}_i b) = \tilde{x}_i \kappa(b)$ ($x=e,f$). 
Finally, it is clear from the definition of $\kappa$ that $\kappa(\tilde{x}_m b)=\tilde{x}_m \kappa(b)$ ($x=e,f$) for $b\in \ms{B}(K(\infty))$, where we assume that $\kappa({\bf 0})={\bf 0}$.
Therefore, $\kappa$ is an isomorphism of $\U$-crystals. \qed

\section{Crystal base of $\U^-$}\label{sec:crystal base of U^-}

\subsection{Crystal base of $\U^-$}

For $i\in I$, we define $\te_i$ and $\tf_i$ on $\U^-$ as follows:
Let $u\in \U^-$ be given which is homogeneous.
\begin{enumerate}
 \item Suppose that $i<m$. We have $u=\sum_{k\ge 0}f_i^{(k)}u_k$ with $e'_i(u_k)=0$ ($k\ge 0$), and then define
\begin{equation}\label{eq:Kashiwara operators for U^-:<M}
\te_i u = \sum_{k\ge 1}f_i^{(k-1)}u_k,\quad \tf_i u = \sum_{k\ge 0}f_i^{(k+1)}u_k,
\end{equation}
in the same way as in \eqref{eq:Kashiwara operators for U^-}.
 
 \item If $i=m$, then we define
\begin{equation}\label{eq:Kashiwara operators for U^-:M}
\te_m u = e'_m(u),\quad \tf_m u = f_m u,
\end{equation} 
in the same way as in \eqref{eq:eM fM for K}.
 
 \item Suppose that $i>m$. We have $u=u_1 u_2 u_3$ for unique $u_1\in \mc{K}$, $u_2\in \U_{m|0}^-$ and $u_3\in \U_{0|n}^-$ (up to scalar multiplication) by \eqref{eq:decomp of U^-}, and then define
\begin{equation}\label{eq:Kashiwara operators for U^-:>M}
\te_i u =  u_1u_2(\te_iu_3),\quad \tf_i u =  u_1u_2(\tf_iu_3),
\end{equation} 
where $\te_iu_3$ and $\tf_iu_3$ are given in \eqref{eq:Kashiwara operator L(infty) for 0|n}.
\end{enumerate}
We define a crystal base of $\U^-$ as in Definition \ref{def:crystal for poly} with respect to 
 \eqref{eq:Kashiwara operators for U^-:<M}, 
 \eqref{eq:Kashiwara operators for U^-:M}, 
 \eqref{eq:Kashiwara operators for U^-:>M}, and
 with weight space \eqref{eq:Q grading}.  
Then we have the following, which is the main result in this paper. 

\begin{thm}\label{thm:crystal base of U^-}
Suppose that $m, n > 0$.
Let 
\begin{equation*}
\begin{split}
\ms{L}(\infty) & = \ms{L}(\K)\cdot \ms{L}_{m|0}(\infty)\cdot \ms{L}_{0|n}(\infty), \\
\ms{B}(\infty)&=\ms{B}(\K) \cdot \ms{B}_{m|0}(\infty) \cdot \ms{B}_{0|n}(\infty),
\end{split}
\end{equation*}
where $\cdot$ denotes the multiplication in $\U^-$ and the induced one in $\ms{L}(\infty)/q\ms{L}(\infty)$ respectively.
Then $(\ms{L}(\infty),\ms{B}(\infty))$ is a crystal base of $\U^-$, and as $\U$-crystals
\begin{equation*}
\begin{split}
 \ms{B}(\infty) & \cong \ms{B}(K(\infty)). 
\end{split}
\end{equation*}
\end{thm}
\pf
The pair $(\ms{L}(\infty),\ms{B}(\infty))$ clearly satisfies the conditions (1)--(3) in Definition \ref{def:crystal for poly}.
Let us check the conditions (4) and (5) in Definition \ref{def:crystal for poly}. It is easy to check when $i=m$ by Corollary \ref{cor:crystal base of K(0)} and \eqref{eq: iso bkla}.

Suppose that $i<m$.
It is enough to consider $( \ms{L}(\K)\cdot \ms{L}_{m|0}(\infty),  \ms{B}(\K)\cdot \ms{B}_{m|0}(\infty))$ 
since $e'_i(u)=0$ and $f_i$ commutes with $u$ for any $u \in \U^{-}_{0|n}$.
Let $B_q=B_q(\U_{m|0})$ (see Appendix \ref{app:B_q}).
Recall the $\Bbbk$-linear isomorphism \eqref{eq:decomp of U^-}
given by multiplication.
We can check that there is a well-defined action of $B_q$ on $\U^-$, 
where the actions of $e'_i$, $f_i$, and $k_\mu$ for $i<m$ and $\mu\in P$ are given by \eqref{eq: derivation}, left multiplication, and conjugation, respectively. 
Moreover, $\K \cdot \U_{m|0}^-$ is a $B_q$-submodule of $\U^-$ by Lemma \ref{lem:commutation relation} and
\begin{align*} \label{eq: e_m' and f_m on K}
	e_i'(\bff_{\beta})
	=
	\begin{cases}
		\bff_{\beta-\alpha_{i}} 	&	 \text{if $\beta-\alpha_{i} \in \Phi_{1}^{+}$}, \\
		0 							&	 \text{otherwise},
	\end{cases}
	\quad	(i<m, \beta \in \Phi_{1}^{+}).
\end{align*}
On the other hand, we may regard $K(0) \ot \U_{m|0}^-$ as a $B_q$-module via \eqref{eq: comul on B} since $K(0)$ is a finite-dimensional $\U_{m|0}$-module and $\U_{m|0}^-$ is a $B_q$-module.
Note that $K(0) =\U\ot_{\mc{P}}V_{m,n}(0) = \K \ot V_{m,n}(0)$ by \eqref{eq: decomposition of Kac module}, where $V_{m,n}(0)$ is a trivial $\U_{m,n}$-module spanned by $v_0$.
 
We first claim that the $\Bbbk$-linear isomorphism
\begin{equation}\label{eq: K dot U- isom K ot U- for m|0}
 \xymatrixcolsep{3pc}\xymatrixrowsep{0pc}\xymatrix{
\phi : K(0) \ot \U_{m|0}^- \ \ar@{->}[r] &\ \K \cdot \U_{m|0}^- \\
 (u_1\ot v_0) \ot u_2 \ar@{|->}[r] &\ u_1 u_2
 }
\end{equation}
is an isomorphism of $B_q$-modules. It suffices to show that 
the map \eqref{eq: K dot U- isom K ot U- for m|0} preserves the action of $B_q$. 
Let ${\rm ad}_q: \U \longrightarrow {\rm End}_{\Bbbk}(\U)$ be the adjoint representation on $\U$ with respect to \eqref{eq:comult-1},
where
\begin{equation*} \label{eq: quantum adjoint}
\begin{split}
	& {\rm ad}_q(k_\mu)(u) = k_\mu u k_\mu^{-1}, \\
	& {\rm ad}_q(e_i)(u) = \left( e_i u - u e_i \right) k_i, \\
	& {\rm ad}_q(f_i)(u) = f_i u - k_i u k_i^{-1} f_i,
\end{split}
\end{equation*}
for $i \in I$, $\mu \in P$ and $u \in \U$.
We write ${\rm ad}_q(u)(x)=u\cdot x$ for simplicity.

Let $u_1 \in \K$ and $u_2 \in \U_{m|0}^-$ be homogeneous elements.
First we have
\begin{equation*} 
\begin{split}
e'_i(u_1 u_2) & = e'_i(u_1)u_2 + \bq (\alpha_i, |u_1|) u_1 e'_i(u_2)
= e'_i(u_1)u_2 + (k_i \cdot u_1) e'_i(u_2),\\
e'_i ((u_1\ot v_0) \ot u_2) & = (q^{-1}-q) ( k_i e_i (u_1\ot v_0) ) \ot u_2 + ( k_i (u_1\ot v_0) ) \ot e'_i (u_2),
\end{split}
\end{equation*}
by \eqref{eq: derivation} and \eqref{eq: comul on B}, respectively.
On the other hand, we have
\begin{equation*}
\begin{split}
k_i(u_1 \ot v_0)
& = k_i u_1 k^{-1}_i \ot v_0 \\
& = ( k_i \cdot u_1 ) \ot v_0, \\
(q^{-1}-q) k_i e_i (u_1\ot v_0) 
& = (q^{-1}-q) (k_i e_i u_1)\ot v_0 \\
& = (q^{-1}-q) (k_i (e_i u_1 - u_1 e_i) k_i k_i^{-1})\ot v_0\\
& = (q^{-1}-q) k_i \left( \frac{k_i e''_i(u_1) - k_i^{-1}e'_i(u_1)}{q-q^{-1}} \right) \ot v_0 \\
& = ( e'_i(u_1) - k^2_ie''_i(u_1) ) \ot v_0\\
& = e'_i (u_1) \ot v_0,
\end{split}
\end{equation*}
where the last equation follows from the fact that
$e''_i ( \bff_{\beta} ) = 0$ for $i<m$ and $\beta \in \Phi^+_1$, which can be checked directly.
Therefore $\phi (e'_i ((u_1\ot v_0) \ot u_2))=e'_i(u_1 u_2)=e'_i \phi  ((u_1\ot v_0) \ot u_2)$.
Also we have
\begin{equation*} 
\begin{split}
	f_i u_1 u_2 & = (f_i \cdot u_1) u_2 + k_i u_1 k_i^{-1} f_i u_2
	= (f_i \cdot u_1) u_2 + (k_i \cdot u_1) f_i u_2,\\
	f_i ((u_1\ot v_0) \ot u_2 ) & = ( f_i (u_1\ot v_0) ) \ot u_2 + ( k_i (u_1\ot v_0) ) \ot f_i u_2,
\end{split}
\end{equation*}
by \eqref{eq: comul on B}, where
\begin{equation*}
	f_i (u_1\ot v_0) 
	= (f_i u_1) \ot v_0 
	=  (f_i u_1 - k_i u_1 k_i^{-1} f_i) \ot v_0
	= ( f_i \cdot u_1 ) \ot v_0.
\end{equation*}
This implies that $\phi (f_i ((u_1\ot v_0) \ot u_2 )) = f_i \phi ((u_1\ot v_0) \ot u_2 )$.
It is clear that $\phi (k_i ((u_1\ot v_0) \ot u_2 )) = k_i \phi ((u_1\ot v_0) \ot u_2 )$.
Therefore $\phi$ is $\Boson$-linear, which proves the claim.

Since $(\ms{L}(K(0)) \ot \ms{L}_{m|0}(\infty),\ms{B}(K(0)) \ot \ms{B}_{m|0}(\infty))$ 
is a crystal base of $K(0) \ot \U_{m|0}^-$ as a $\Boson$-module by Theorem \ref{thm: tensor product rule for Boson},
$(\ms{L}(\K) \cdot \ms{L}_{m|0}(\infty),\ms{B}(\K) \cdot \ms{B}_{m|0}(\infty))$ is a crystal base 
of $\K \cdot \U_{m|0}^-$ as a $\Boson$-module by \eqref{eq: K dot U- isom K ot U- for m|0},
and therefore satisfies the conditions Definition \ref{def:crystal for poly}(4) and (5).

Next suppose that $i>m$. 
It is enough to consider $(\ms{L}_{0|n}(\infty), \ms{B}_{0|n}(\infty))$ by \eqref{eq:Kashiwara operators for U^-:>M},
and $(\ms{L}_{0|n}(\infty), \ms{B}_{0|n}(\infty))$ satisfies the conditions Definition \ref{def:crystal for poly}(4) and (5) by Section \ref{subsec: crystal for 0|n}.

Therefore we conclude that $(\ms{L}(\infty),\ms{B}(\infty))$ is a crystal base of $\U^-$, where 
$\ms{B}(\infty) \cong \ms{B}(\K) \times \ms{B}_{m|0}(\infty) \times \ms{B}_{0|n}(\infty)$ as a $\U$-crystal by \eqref{eq:abstract crystal B(infty)}.
Moreover, we have $\ms{B}(\infty) \cong \ms{B}(K(\infty))$ by Theorem \ref{thm: isomorphism kappa}.
This completes the proof.
\qed

\subsection{Parabolic Verma module}

Let $\mc{Q}$ be the subalgebra of $\U$ generated by $\U_{0|n}$, $\U^0$ and $\U^+$. 
Let $\la\in P^+$ be given.
We extend $V_{0|n}(\la_-)$ to a $\mc{Q}$-module by $e_i v_{\la_-}=0$ and $k_\mu v_{\la_-}=\bq(\la,\mu)v_{\la_-}$ for $i\le m$ and $\mu\in P$. 
Define a $\U$-module
\begin{equation*}
X(\lambda) = \U\otimes_{\mc{Q}}V_{0|n}(\la_-).
\end{equation*}
Let $1_\la$ denote a highest weight vector of $X(\la)$. 
Since $\mc{Q}\cong \U^-_{0|n}\ot \U^0 \ot \U^+$ as a $\Bbbk$-space,  
we have as a $\Bbbk$-space
\begin{equation}\label{eq:decomp of M(la)}
 X(\la) \cong \K\ot \U_{m|0}^-\ot V_{0|n}(\la_-).
\end{equation} 
We define $\te_i$, $\tf_i$ ($i\in I$) on $X(\la)$ 
as in \eqref{eq:Kashiwara operators for U^-:<M}, \eqref{eq:Kashiwara operators for U^-:M} for $i\le m$, and 
as in \eqref{eq:modified induced Kashiwara operators} for $i>m$. We may also define a crystal base of $X(\la)$ as in Definition \ref{def:crystal for poly}.

Let us assume that $\ms{B}(\K)\times \ms{B}_{m|0}(\infty) \times \ms{B}_{0|n}(\la_-)$ is a $\U$-crystal, where for $b=(S,b_+,b_-)$, $i\in I$ and $x=e,f$,
\begin{equation}\label{eq:abstract crystal parabolic Verma}
\begin{split}
 {\rm wt}(b) &= {\rm wt}(S) + {\rm wt}(b_+) + {\rm wt}(b_-),\\
 \tilde{x}_i b &= 
 \begin{cases}
 (S',b'_+,b_-) & \text{if $i<m$ and $\tilde{x}_i (S\ot b_+)=S'\ot b'_+$},\\
 (S'',b_+,b'_-) & \text{if $i>m$ and $\tilde{x}_i (S\ot b_-)=S''\ot b'_-$},\\
 (\tilde{x}_mS,b_+,b_-) & \text{if $i=m$},
 \end{cases}
\end{split}
\end{equation}
and $\varepsilon_i(b)$, $\varphi_i(b)$ are defined as in \eqref{eq:ep phi}. 
Here we assume that $\tilde{x}_i b={\bf 0}$ if any of its component on the right-hand side is ${\bf 0}$.

\begin{thm}\label{thm:crystal base of parabolic Verma}
Let 
\begin{equation*}
\begin{split}
\ms{L}(X(\la))&=\sum_{r\geq 0,\, k_1,\ldots, k_r\in I}A_0 \tilde{x}_{k_1}\cdots\tilde{x}_{k_r}1_\la, \\
\ms{B}(X(\la))&=\{\, \,\pm\,\tilde{x}_{k_1}\cdots\tilde{x}_{k_r}1_\la\!\!\! \pmod{q \ms{L}(X(\la))}\,|\,r\geq 0, k_1,\ldots, k_r\in I\,\}\setminus\{0\},
\end{split}
\end{equation*}
where $x=e, f$ for each $k_i$.
Then we have the following:
\begin{itemize}
	\item[(1)] $(\ms{L}(X(\la)),\ms{B}(X(\la)))$ is a crystal base of $X(\la)$,
	
	\item[(2)] $\ms{L}(X(\la)) \,= \left(\ms{L}(\K)\cdot \ms{L}_{m|0}(\infty)\right)\ot \ms{L}_{0|n}(\la_-)$,
	
	\item[(3)] $\ms{B}(X(\la)) \cong \left( \ms{B}(\K) \times \ms{B}_{m|0}(\infty)  \times \ms{B}_{0|n}(\la_-) \right)  \ot T_{\la_+} $ as a $\U$-crystal, and it is connected.
\end{itemize}
\end{thm}
\pf 
Let
\begin{equation*}
	\begin{split}
	\ms{L}' & = \left( \ms{L}(\K)\cdot \ms{L}_{m|0}(\infty) \right) \ot \ms{L}_{0|n}(\la_-),\\
	\ms{B}' & = \left( \ms{B}(\K) \cdot \ms{B}_{m|0}(\infty) \right) \ot \ms{B}_{0|n}(\la_-) \ot T_{\la_+}.
	\end{split}
\end{equation*}
We first claim that $(\ms{L}', \ms{B}')$ is a crystal base of $X(\la)$.
It clearly satisfies the conditions Definition \ref{def:crystal for poly}(1)--(3).
So it is enough to check the conditions Definition \ref{def:crystal for poly}(4)--(5), which is clear when $i=m$ by Corollary \ref{cor:crystal base of K(0)} and \eqref{eq: iso bkla}.

Suppose that $i<m$. We have shown in the proof of Theorem \ref{thm:crystal base of U^-} that 
$$(\ms{L}(\K) \cdot \ms{L}_{m|0}(\infty),\ms{B}(\K) \cdot \ms{B}_{m|0}(\infty))$$ is a crystal base of $\K \cdot \U_{m|0}^-$ as a $\Boson$-module.
This implies the conditions Definition \ref{def:crystal for poly}(4) and (5).
Suppose that $i>m$.
It is enough to consider 
\begin{equation} \label{eq: parabolic Verma i>m}
( (\ms{L}(\K) \cdot 1) \ot \ms{L}_{0|n}(\la_-),  (\ms{B}(\K) \cdot 1) \ot \ms{B}_{0|n}(\la_-)),
\end{equation}
since $e_i$ and $f_i$ commute with $\U^{-}_{m|0}$.
By Theorem \ref{thm: crystal base of Kac module} and Corollary \ref{cor:crystal base of K(0)}, the pair \eqref{eq: parabolic Verma i>m} satisfies the conditions Definition \ref{def:crystal for poly}(4)--(5).
Therefore $(\ms{L}', \ms{B}')$ is a crystal base of $X(\la)$, and as a $\U$-crystal 
$$\ms{B}' \cong \left( \ms{B}(\K) \times \ms{B}_{m|0}(\infty)  \times \ms{B}_{0|n}(\la_-) \right)  \ot T_{\la_+}.$$

Next, we claim that the $\U$-crystal $\ms{B}'$ is connected. 
Let
\begin{equation*}
	\xymatrixcolsep{2pc}\xymatrixrowsep{0pc}\xymatrix{
	\iota_\la : \ms{B}({K(\la)})  \ \ar@{->}[r] &\ \ms{B}' \\
	\quad (S, b_+, b_-) \ \ar@{|->}[r] &\ (S, b_+, b_-) \ot t_{\la_+}
	},
\end{equation*}
where we regard $\ms{B}_{m|0}(\la_+) \subset \ms{B}_{m|0}(\infty) \ot T_{\la_+}$ by \eqref{eq: iso 0nla}.
Let $b_1, b_2 \in \ms{B}(K(\la))$ such that $\tf_i b_1 = b_2$ for some $i \in I$.
Then it is straightforward to check that $\tf_i \iota_\la(b_1) = \iota_\la(b_2)$
by comparing Proposition \ref{prop: description of BKla} and \eqref{eq:abstract crystal parabolic Verma}.

Now let $b \in \ms{B}'$ be given. 
Note that $\ms{B'}$ depends not on $\la_+$, but only on $\la_-$ as an $I$-colored oriented graph.
So we may assume that $\la_+ \gg 0$ so that $b=\iota_\la(b_0)$ for a unique $b_0 \in \ms{B}({K(\la)})$.
Since $\ms{B}({K(\la)})$ is connected by Theorem \ref{thm: crystal base of Kac module}, we have $b_0 = \tilde{x}_{k_1}\cdots\tilde{x}_{k_r}1_\la$ for some $r\geq 0, k_1,\ldots, k_r\in I$, where $x=e,f$ for each $k_i$.
By the previous arguments, we have 
\begin{equation*}
b = \iota_\la (b_0) = \iota_\la (\tilde{x}_{k_1}\cdots\tilde{x}_{k_r}1_\la) = \tilde{x}_{k_1}\cdots\tilde{x}_{k_r}1_\la.
\end{equation*}
Therefore $\ms{B}'$ is connected.

Finally, we conclude by standard arguments that $(\ms{L}(X(\la)),\ms{B}(X(\la)))=(\ms{L}',\ms{B}')$,
since $\ms{L}'_\la = \ms{L}(X(\la))_\la = A_0 1_\la$ and $\ms{B}'$ is connected.
\qed

\begin{cor}\label{cor:compatibility with U}
 A crystal base of $X(\la)$ is unique up to scalar multiplication.
\end{cor}

Let 
\begin{equation*} 
 \xymatrixcolsep{3pc}\xymatrixrowsep{0pc}\xymatrix{
  \U^-  \ar@{->}^{\pi_-}[r] &  X(\la) \ar@{->}^{\pi_+}[r] & K(\la).
  }
\end{equation*}
be the canonical projections of $\U^-$-modules.
By \eqref{eq:decomp of U^-}, \eqref{eq:decomp of M(la)} and \eqref{eq:dual embedding}, we have an embedding $\pi^\vee_- : X(\la) \longrightarrow \U^-$ of a $\Bbbk$-space so that we have
\begin{equation}\label{eq:projections}
 \xymatrixcolsep{3pc}\xymatrixrowsep{0pc}\xymatrix{
  \U^-  \ar@{<-}^{\pi^\vee_-}[r] &  X(\la) \ar@{->}^{\pi_+}[r] & K(\la).
  }
\end{equation}
The crystal bases for $\U^-$, $X(\la)$, and $K(\la)$ are compatible with \eqref{eq:projections} in the following sense.

\begin{cor}\label{cor:compatibility with K}
Under the above hypothesis, we have
\begin{itemize}
 \item[(1)] $\pi_+(\ms{L}(X(\la)))=\ms{L}(K(\la))$, 
 \item[(2)] $\ov{\pi}_+(\ms{B}(X(\la)))=\ms{B}(K(\la))\cup \{0\}$, where $\ov{\pi}_+$ is the $\Q$-linear map induced from $\pi_+$,
 \item[(3)] $\ov{\pi}_+$ restricts to a weight preserving bijection 
$$\ov{\pi}_+ : \{\,b\in \ms{B}(X(\la))\,|\,\ov{\pi}_+(b)\neq 0 \,\} \longrightarrow \ms{B}(K(\la)),$$ 
which commutes with $\te_i$, $\tf_i$ for $i\in I$.
\end{itemize}
\end{cor}

\begin{cor}
Under the above hypothesis, we have  
\begin{itemize}
 \item[(1)] $\pi^\vee_-(\ms{L}(X(\la)))\subset \ms{L}(\infty)$,
 
 \item[(2)] $\ov{\pi}^\vee_-(\ms{B}(X(\la)))\subset \ms{B}(\infty)$, where $\ov{\pi}^\vee_-$ is the $\Q$-linear map induced from $\pi_-^\vee$,
 
 \item[(3)] $\ov{\pi}^\vee_-$ restricts to a bijection 
\begin{equation*}
 \ov{\pi}^\vee_- : \ms{B}(X(\la))  \longrightarrow \{\,b=b_1b_2b_3\in \ms{B}(\infty)\,|\ b_3 \ot t_{\la_-} \in \ms{B}_{0|n}(\la_-) \,\},
\end{equation*} 
which commutes with $\te_i$, $\tf_i$ for $i<m$, where each $b \in \ms{B}(\infty)$ is uniquely written as $b=b_1b_2b_3 \in \ms{B}(\infty)$ with $b_1\in \ms{B}(\K), b_2\in \ms{B}_{m|0}(\infty), b_3\in \ms{B}_{0|n}(\infty)$.

\end{itemize}
\end{cor}

\begin{rem} \label{rem: compatibility}
{\rm 
Let $b_0 \in \ms{B}(\infty)$ be given.
Suppose $b_0=\ov{\pi}^\vee_-(b)$ for some $b \in \ms{B}(X(\la))$.
Then we have
$\ov{\pi}_+(\tilde{x}_i b)= \tilde{x}_i \ov{\pi}_+(b)$ 
for $i\in I$, where $x=f$ for $i<m$ and $x=e,f$ for $i\ge m$.
Also, we have
$\ov{\pi}^\vee_-(\tilde{x}_i b)= \tilde{x}_i \ov{\pi}^\vee_-(b)$ 
for $x=e,f$ and $i<m$.
On the other hand, 
we have 
{$\ov{\pi}^\vee_-(\tilde{x}_i b)= \tilde{x}_i \ov{\pi}^\vee_-(b)$, 
for $x=e, f$ if $i>m$ and $\la_- \gg 0$}, or {if $i=m$ and $\langle {\rm wt}(b),\alpha_{m+1}^\vee \rangle\gg 0$} (see Examples \ref{ex:B infty} and \ref{ex: bxla bkla}).
}
\end{rem}

\subsection{Crystal $\ms{B}(\infty)$}

Let us write $\ms{B}_{m|n}(\infty) = \ms{B} (\infty)$.

\begin{prop} \label{prop:Bm1 cnn}
	The crystal $\ms{B}_{m|1}(\infty) \cong \ms{B}(\K_{m|1}) \times \ms{B}_{m|0}(\infty)$ is connected.
	Moreover, $\ms{B}_{m|1}(\infty) = \{ \, \tilde{f}_{k_1}\cdots\tilde{f}_{k_r} (\emptyset ,1) \,|\, r\geq 0, k_1,\ldots, k_r\in I\,\}\setminus\{{\bf 0}\}$.
\end{prop}
\pf
We recall that $\ms{B}(\K_{m|1}) = \cP(\Phi_{\ov 1}^-) $ 
where $\Phi_{\ov 1}^- = \{ \, - \de_a + \de_{m+1} \, | \, 1 \leq a \leq m \, \}$.
Let $b=(S,b_{+}) \in \ms{B}(\K_{m|1}) \times \ms{B}_{m|0}(\infty)$ be given.
We claim that $b= \tilde{f}_{k_1}\cdots\tilde{f}_{k_r} (\emptyset,1)$ 
for some $r\geq 0$ and $ k_1,\ldots, k_r\in I$.
Let $|b| = -\sum_{i\in I} c_i \alpha_i$ for some $c_i \in \Z_+$.
We use induction on $\sum_{i\in I} c_i$.

Suppose that $S=\emptyset$. If $b_{+}\neq 1$, then $\te_i b_+ \neq {\bf 0}$ for some $i\in I \setminus \{ m \}$,
and hence $\te_i b = \te_i (\emptyset, b_+) = (\emptyset, \te_i b_+) \neq {\bf 0}$.
By induction hypothesis, $\te_i b$ is connected to $(\emptyset, 1)$ by $\te_j$'s for $j \in I$, and so is $b$.

Suppose that $S \neq \emptyset$. 
If $-\alpha_m \in S$, then $\te_m S = S\setminus \{ -\alpha_m \} \neq {\bf 0}$, and hence $\te_m b = \te_m (S, b_+) = (\te_m S, b_+) \neq {\bf 0}$.
If $-\alpha_m \not\in S$, then we may take $\beta=- \de_i + \de_{m+1} \in S$ $(1 \leq i < m)$ such that $i$ is the largest one.
Then $\te_i S = (S \setminus \{ \beta \}) \cup \{ \beta + \alpha_i \} \neq {\bf 0}$. 
Then we have
\begin{equation*}
	\te_i b =
	\begin{cases}
		(\te_i S, b_+) & \text{if $\varphi_i(S)\geq\varepsilon_i(b_+)$}, \\
		(S, \te_i b_+) & \text{if $\varphi_i(S)<\varepsilon_i(b_+)$}.
	\end{cases}
\end{equation*}
In both cases, $\te_i b \neq {\bf 0}$ due to our choice of $i$.
Therefore $b$ is connected to $(\emptyset, 1)$ by $\te_j$'s for $j \in I$ by induction hypothesis.
\qed

Let $\ms{B}_{m|1}(\infty) \times \ms{B}_{0|n}(\infty)$ be a $\U$-crystal, where for $b=(b_1,b_2)$ and $i\in I$, 
\begin{equation}\label{eq:abstract crystal B(infty) connected component}
\begin{split}
 {\rm wt}(b) &={\rm wt}(b_1)+{\rm wt}(b_2), \\
 \tilde{x}_i b &= 
 \begin{cases}
 (\tilde{x}_i b_1,b_2) & \text{if $i\leq m$},\\
 (b_1,\tilde{x}_i b_2) & \text{if $i>m$},
 \end{cases}
 \quad (x=e,f),
\end{split}
\end{equation}
and $\varepsilon_i(b), \varphi_i(b)$ are given as in \eqref{eq:ep phi}.
We assume that $\tilde{x}_i b={\bf 0}$ if any of its component on the right-hand side is $0$.
Note that it follows from Proposition \ref{prop:Bm1 cnn} and \eqref{eq:abstract crystal B(infty) connected component} that the $\U$-crystal $\ms{B}_{m|1}(\infty) \times \ms{B}_{0|n}(\infty)$ is connected.

\begin{thm}\label{thm:structure of Binfty}
	We have the following:
	\begin{itemize}
		\item[(1)] Any connected component of $\ms{B}_{m|n}(\infty)$ is isomorphic to $\ms{B}_{m|1}(\infty) \times \ms{B}_{0|n}(\infty)$ 
		as a $\U$-crystal up to shift of weight.
		\item[(2)] The number of connected components of $\ms{B}_{m|n}(\infty)$ is $2^{m(n-1)}$.
	\end{itemize}
\end{thm}
\pf
	Let $X = \{ \, - \de_a + \de_{m+1} \, | \, 1 \leq a \leq m \, \} \subset \Phi_{\ov 1}^-$ and $Y=\Phi_{\ov 1}^- \setminus X$. 
	Let $\ms{C}=\{\, S \,|\, S \subset Y \,\} \subset \ms{B}(\K_{m|n})$.
	As a $\U_{m|0}$-crystal, $\ms{C}$ is isomorphic to a finite disjoint union of $ \ms{B}_{m|0}(\la) $ for $\la \in P^+$ 
	 (see the proof of Lemma \ref{lem: crystal equiv on K0}).
	Then we have an isomorphism of $\U_{m|1}$-crystals,
	\begin{equation}\label{eq: thm cnn - 1}
		\xymatrixcolsep{2pc}\xymatrixrowsep{0pc}\xymatrix{
			 \ms{B}(\K_{m|n})  \ \ar@{->}[r] &\ \ms{B}(\K_{m|1}) \times \ms{C} \\
		\quad S \ \ar@{|->}[r] &\ (S \cap X, S \cap Y)
		},
	\end{equation}
	where the $\U_{m|1}$-crystal structure on $\ms{B}(\K_{m|1}) \times \ms{C}$ 
	is defined in the same manner as in Proposition \ref{prop: description of BKla} for $1 \leq i\leq m$.

	Let us regard $\cP_{m|0} \subset P^+$ (cf. \eqref{eq:highest weight correspondence}).
	By \eqref{eq: BK}, we get
	\begin{equation} \label{eq: thm cnn - 3}
		\ms{C} \cong \bigsqcup_{\substack{\ell(\la) \leq m \\ \ell(\la^t) < n}} \ms{B}_{m|0}(\la) ^{\oplus m_\la}
	\end{equation}
	as a $\U_{m|0}$-crystal, where $m_\la = |\ms{B}_{0|n-1}(\la^t)|$ and $B^{\oplus m}$ denotes the crystal $ B \sqcup \dots \sqcup B$ ($m$ times) for a crystal $B$ and $m \geq 1$.
	This implies
	\begin{equation}\label{eq: thm cnn - 4}
		\ms{C} \ot \ms{B}_{m|0}(\infty) \cong \bigsqcup_{\substack{\ell(\la) \leq m \\ \ell(\la^t) < n}} \bigsqcup_{b \in \ms{B}_{m|0}(\la)} \left( \ms{B}_{m|0}(\infty) \ot T_{{\rm wt}(b)} \right) ^{\oplus m_\la},
	\end{equation}
	by Corollary \ref{cor:Bla ot Binf}.
	Therefore we have isomorphisms of $\U$-crystals
	\begin{equation*}
		\begin{split}
			\ms{B}_{m|n}(\infty) & \cong \ms{B}(\K_{m|n}) \times \ms{B}_{m|0}(\infty) \times \ms{B}_{0|n}(\infty) \\
			& \cong \left( \ms{B}(\K_{m|1}) \times \ms{C} \right) \times \ms{B}_{m|0}(\infty) \times \ms{B}_{0|n}(\infty) \quad \text{by \eqref{eq: thm cnn - 1}}\\
			& \cong \bigsqcup_{\substack{\ell(\la) \leq m \\ \ell(\la^t) < n}} \bigsqcup_{b \in \ms{B}_{m|0}(\la)} 
			\left( \ms{B}(\K_{m|1}) \times  ( \ms{B}_{m|0}(\infty) \ot T_{{\rm wt}(b)} )  \times \ms{B}_{0|n}(\infty) \right) ^{\oplus m_\la} \quad \text{by {\eqref{eq: thm cnn - 4}}}\\
			& \cong \bigsqcup_{\substack{\ell(\la) \leq m \\ \ell(\la^t) < n}} \bigsqcup_{b \in \ms{B}_{m|0}(\la)} 
			\left( \left( \ms{B}_{m|1}(\infty)  \times \ms{B}_{0|n}(\infty) \right) \ot T_{{\rm wt}(b)}  \right) ^{\oplus m_\la}.
		\end{split}
	\end{equation*}
	Note that $\ms{B}_{m|1}(\infty) \times \ms{B}_{0|n}(\infty)$ is connected by Proposition \ref{prop:Bm1 cnn}.
	Also we see from the above decomposition that the number of connected components of $\ms{B}_{m|n}(\infty)$ is $|\ms{C}|=2^{m(n-1)}$ by \eqref{eq: thm cnn - 3}.
\qed

\subsection{Examples}

Let us give examples to describe the crystal operators $\tf_i$ on $\ms{B}(\infty)\cong \ms{B}(\K)\times \ms{B}_{m|0}(\infty) \times \ms{B}_{0|n}(\infty)$.

Note that an explicit description of the crystal $\ms{B}(\K)$ is given in Section \ref{subsec: BKla}, where we identify $\ms{B}(\K)$ with $\ms{P}(\Phi^-_{\ov 1})$ or equivalently with the set of PBW-type monomials $\bff_S$ in \eqref{eq: monomial F_S}.
Hence by using a well-known realization of $\ms{B}_{m|0}(\infty)$ and $\ms{B}_{0|n}(\infty)$ and applying Theorem \ref{thm: isomorphism kappa}, one may compute $\tf_i b$ ($b\in\ms{B}(\infty)$). Here, let us use the crystals of PBW bases (or the corresponding Lusztig data) for $\ms{B}_{m|0}(\infty)$ and $\ms{B}_{0|n}(\infty)$ with respect to \eqref{eq: PBW m0} and \eqref{eq: PBW 0n}, respectively, and a combinatorial description of $\tf_i$ on them (see \cite{Re, SST}).

Under the above identification,	we may write $b = (S, b_+, b_-) \in \ms{B}(\infty)$ as 
	\begin{equation} \label{eq: Lusztig data in 3|4}
		S = (c_{st})_{\delta_s - \delta_t \in \Phi_{\ov 1}^+}, \quad
		b_+ = (c_{st})_{\delta_s - \delta_t \in \Phi_{m|0}^+}, \quad
		b_- = (c_{st})_{\delta_s - \delta_t \in \Phi_{0|n}^+},
	\end{equation}
	where $c_{st}$ is the multiplicity of $\de_s - \de_t$ or its root vector $\bff_{\de_s - \de_t } $ in $b$ with
	$ c_{st} \in \{ 0, 1 \}$ for $\delta_s - \delta_t \in \Phi_{\ov 1}^+$ and $c_{st} \in \Z_+$ otherwise.

\begin{ex} \label{ex:B infty}
	{\em 
	Let $m = 3$ and $n = 4$.
 	Let $b = (S, b_+, b_-) \in \ms{B}(\infty)$ be given as in \eqref{eq: Lusztig data in 3|4}.
	We identify $b$ with the following array of $(c_{st})$:
	
	\begin{equation*}
	\begin{tikzpicture}[baseline=(current  bounding  box.center), every node/.style={scale=1}]
		\node (root_34) at (0,0) {$\overset{\delta_3 - \delta_4}{\ot}$};
		\node (root_35) at (0.7,0.7) {$\overset{\delta_3 - \delta_5}{\ot}$};
		\node (root_36) at (1.4,1.4) {$\overset{\delta_3 - \delta_6}{\ot}$};
		\node (root_37) at (2.1,2.1) {$\overset{\delta_3 - \delta_7}{\ot}$};
		\node (root_24) at (0.7,-0.7) {$\overset{\delta_2 - \delta_4}{\ot}$};
		\node (root_25) at (1.4,0) {$\overset{\delta_2 - \delta_5}{\ot}$};
		\node (root_26) at (2.1,0.7) {$\overset{\delta_2 - \delta_6}{\ot}$};
		\node (root_27) at (2.8,1.4) {$\overset{\delta_2 - \delta_7}{\ot}$};
		\node (root_23) at (3.5,2.1) {$\overset{\delta_2 - \delta_3}{\oplus}$};
		\node (root_13) at (4.2,1.4) {$\overset{\delta_1 - \delta_3}{\oplus}$};
		\node (root_12) at (4.9,2.1) {$\overset{\delta_1 - \delta_2}{\oplus}$};	
		\node (root_14) at (1.4,-1.4) {$\overset{\delta_1 - \delta_4}{\ot}$};
		\node (root_15) at (2.1,-0.7) {$\overset{\delta_1 - \delta_5}{\ot}$};
		\node (root_16) at (2.8,0) {$\overset{\delta_1 - \delta_6}{\ot}$};
		\node (root_17) at (3.5,0.7) {$\overset{\delta_1 - \delta_7}{\ot}$};
		\node (root_45) at (2.8,-1.4) {$\overset{\delta_4 - \delta_5}{\ominus}$};
		\node (root_46) at (3.5,-0.7) {$\overset{\delta_4 - \delta_6}{\ominus}$};
		\node (root_47) at (4.2,0) {$\overset{\delta_4 - \delta_7}{\ominus}$};
		\node (root_56) at (4.2,-1.4) {$\overset{\delta_5 - \delta_6}{\ominus}$};
		\node (root_57) at (4.9,-0.7) {$\overset{\delta_5 - \delta_7}{\ominus}$};
		\node (root_67) at (5.6,-1.4) {$\overset{\delta_6 - \delta_7}{\ominus}$};
	\end{tikzpicture}
	\,\,\,
	\begin{tikzpicture}[baseline=(current  bounding  box.center), every node/.style={scale=0.9}]
		\node (root_34) at (0,0) {$c_{34}$};
		\node (root_35) at (0.7,0.7) {$c_{35}$};
		\node (root_36) at (1.4,1.4) {$c_{36}$};
		\node (root_37) at (2.1,2.1) {$c_{37}$};
		\node (root_24) at (0.7,-0.7) {$c_{24}$};
		\node (root_25) at (1.4,0) {$c_{25}$};
		\node (root_26) at (2.1,0.7) {$c_{26}$};
		\node (root_27) at (2.8,1.4) {$c_{27}$};
		\node (root_23) at (3.5,2.1) {$c_{23}$};
		\node (root_13) at (4.2,1.4) {$c_{13}$};
		\node (root_12) at (4.9,2.1) {$c_{12}$};
		\node (root_14) at (1.4,-1.4) {$c_{14}$};
		\node (root_15) at (2.1,-0.7) {$c_{15}$};
		\node (root_16) at (2.8,0) {$c_{16}$};
		\node (root_17) at (3.5,0.7) {$c_{17}$};
		\node (root_45) at (2.8,-1.4) {$c_{45}$};
		\node (root_46) at (3.5,-0.7) {$c_{46}$};
		\node (root_47) at (4.2,0) {$c_{47}$};
		\node (root_56) at (4.2,-1.4) {$c_{56}$};
		\node (root_57) at (4.9,-0.7) {$c_{57}$};
		\node (root_67) at (5.6,-1.4) {$c_{67}$};
	\end{tikzpicture}
	\end{equation*}\smallskip
	
	\noindent where the positive roots in $\Phi_{\ov 1}^+$ (resp. $\Phi_{3|0}^+$ and $\Phi_{0|4}^+$) are marked as $\ot$ (resp. $\oplus$ and $\ominus$). (The array of the roots on the left is equal to the one representing a convex order on the set of positive roots for $\gl(m+n)$ associated to the longest element in the Weyl group of $\gl(m|n)$ adapted to the oriented Dynkin diagram with a unique sink at $\de_3-\de_4$.)
	
		Let us compute $\tf_i b$ for 
	\begin{equation} \label{eq: b}
	b = \quad
	\scalebox{0.8}{$
	\begin{tikzpicture}[baseline=(current  bounding  box.center), every node/.style={scale=0.9}]
		\node (root_34) at (0,0) {$0$};
		\node (root_35) at (0.7,0.7) {$1$};
		\node (root_36) at (1.4,1.4) {$0$};
		\node (root_37) at (2.1,2.1) {$1$};
		\node (root_24) at (0.7,-0.7) {$1$};
		\node (root_25) at (1.4,0) {$1$};
		\node (root_26) at (2.1,0.7) {$1$};
		\node (root_27) at (2.8,1.4) {$0$};
		\node (root_14) at (1.4,-1.4) {$0$};
		\node (root_15) at (2.1,-0.7) {$0$};
		\node (root_16) at (2.8,0) {$1$};
		\node (root_17) at (3.5,0.7) {$0$};
		\node (root_23) at (3.5,2.1) {$2$};
		\node (root_13) at (4.2,1.4) {$1$};
		\node (root_12) at (4.9,2.1) {$2$};	
		\node (root_45) at (2.8,-1.4) {$2$};
		\node (root_46) at (3.5,-0.7) {$1$};
		\node (root_47) at (4.2,0) {$1$};
		\node (root_56) at (4.2,-1.4) {$1$};
		\node (root_57) at (4.9,-0.7) {$2$};
		\node (root_67) at (5.6,-1.4) {$1$};
	\end{tikzpicture}
	$}\quad .
	\end{equation}
	\smallskip

	{\it Case 1}. Suppose that $i = 3$.
	We have $\te_3 b = {\bf 0}$ since $c_{34} = 0$, and
	\begin{equation*}
	\tf_3 b= \quad
	\scalebox{0.8}{$
	\begin{tikzpicture}[baseline=(current  bounding  box.center), every node/.style={scale=0.9}]
		\node (root_34) at (0,0) {$\red{\mb 1}$};
		\node (root_35) at (0.7,0.7) {$\color{gray}{1}$};
		\node (root_36) at (1.4,1.4) {$\color{gray}{0}$};
		\node (root_37) at (2.1,2.1) {$\color{gray}{1}$};
		\node (root_24) at (0.7,-0.7) {$\color{gray}{1}$};
		\node (root_25) at (1.4,0) {$\color{gray}{1}$};
		\node (root_26) at (2.1,0.7) {$\color{gray}{1}$};
		\node (root_27) at (2.8,1.4) {$\color{gray}{0}$};
		\node (root_14) at (1.4,-1.4) {$\color{gray}{0}$};
		\node (root_15) at (2.1,-0.7) {$\color{gray}{0}$};
		\node (root_16) at (2.8,0) {$\color{gray}{1}$};
		\node (root_17) at (3.5,0.7) {$\color{gray}{0}$};
		\node (root_23) at (3.5,2.1) {$\color{gray}{2}$};
		\node (root_13) at (4.2,1.4) {$\color{gray}{1}$};
		\node (root_12) at (4.9,2.1) {$\color{gray}{2}$};	
		\node (root_45) at (2.8,-1.4) {$\color{gray}{2}$};
		\node (root_46) at (3.5,-0.7) {$\color{gray}{1}$};
		\node (root_47) at (4.2,0) {$\color{gray}{1}$};
		\node (root_56) at (4.2,-1.4) {$\color{gray}{1}$};
		\node (root_57) at (4.9,-0.7) {$\color{gray}{2}$};
		\node (root_67) at (5.6,-1.4) {$\color{gray}{1}$};
	\end{tikzpicture}
	$}\quad .
	\end{equation*}
	\smallskip

	{\it Case 2}. Suppose that $i = 1$.
	Let us first compute $\varphi_1(S)$.
	Let
	\begin{equation}\label{eq: sigma(S)}
		\sigma(S) = ( +^{c_{24}} , -^{c_{14}} , +^{c_{25}} , -^{c_{15}} , \dots , +^{c_{27}} , -^{c_{17}})=(\sigma_1, \sigma_2, \dots),
	\end{equation}
	be a finite sequence of $\pm$,
	where $\pm^{a} = (\pm , \dots , \pm)$ ($a$ times) and $\pm^0=\cdot \,$.
	We replace a pair $(\sigma_i, \sigma_j) = (+ , -)$ by $(\cdot , \cdot)$, where $i < j$ and $\sigma_k = \cdot$ for $i < k < j$,
	and repeat this process until we get a sequence $\sigma^{\rm red}(S)$ with no $-$ placed to the right of $+$.
	Then $\varphi_1(S)$ (resp. $\varepsilon_1(S)$) is the number of $+$'s (resp. $-$'s) in $\sigma^{\rm red}(S)$ (cf.~Section \ref{subsec: BKla}).
	In this case, we have $\varphi_1(S)=2$ and $\varepsilon_1(S)=0$ since $\sigma(S)=(\, + , \, \cdot \, , + , \, \cdot \, , + , - , \, \cdot \, , \, \cdot \,)$ 
	and $\sigma^{\rm red}(S)=(\, + , \, \cdot \, , + , \, \cdot \, , \, \cdot \, , \, \cdot \, , \, \cdot \, , \, \cdot \,)$.
	
	Next, we compute $\varepsilon_1(b_+)$. Let
	\begin{equation}\label{eq: sigma(b_+)}
		\sigma(b_+) = (-^{c_{13}} , +^{c_{23}} , -^{c_{12}}),
	\end{equation}
	and let $\sigma^{\rm red}(b_+)$ be given in the same way as above.
	Then $\varepsilon_1(b_+)$ is the number of $-$'s in $\sigma^{\rm red}(b_+)$.
	In this case, we have $\varepsilon_1(b_+)=1$ since $\sigma(b_+) = (-,+,+,-,-)$ and $\sigma^{\rm red}(b_+)=(\, - , \, \cdot \, , \, \cdot \, , \, \cdot \, , \, \cdot \,)$.

	Therefore by \eqref{eq:abstract crystal B(infty)}, we have
	\begin{equation} \label{eq:f1b}
	\tf_1 b= (\tf_1 S, b_+, b_-) = \quad
	\scalebox{0.8}{$
	\begin{tikzpicture}[baseline=(current  bounding  box.center), every node/.style={scale=0.9}]
		\node (root_34) at (0,0) {\color{gray}{$0$}};
		\node (root_35) at (0.7,0.7) {\color{gray}{$1$}};
		\node (root_36) at (1.4,1.4) {\color{gray}{$0$}};
		\node (root_37) at (2.1,2.1) {\color{gray}{$1$}};
		\node (root_24) at (0.7,-0.7) {$\red{\mb 0}$};
		\node (root_25) at (1.4,0) {${\mb 1}$};
		\node (root_26) at (2.1,0.7) {${\mb 1}$};
		\node (root_27) at (2.8,1.4) {${\mb 0}$};
		\node (root_14) at (1.4,-1.4) {$\red{\mb 1}$};
		\node (root_15) at (2.1,-0.7) {${\mb 0}$};
		\node (root_16) at (2.8,0) {${\mb 1}$};
		\node (root_17) at (3.5,0.7) {${\mb 0}$};
		\node (root_23) at (3.5,2.1) {${\mb 2}$};
		\node (root_13) at (4.2,1.4) {${\mb 1}$};
		\node (root_12) at (4.9,2.1) {${\mb 2}$};	
		\node (root_45) at (2.8,-1.4) {\color{gray}{$2$}};
		\node (root_46) at (3.5,-0.7) {\color{gray}{$1$}};
		\node (root_47) at (4.2,0) {\color{gray}{$1$}};
		\node (root_56) at (4.2,-1.4) {\color{gray}{$1$}};
		\node (root_57) at (4.9,-0.7) {\color{gray}{$2$}};
		\node (root_67) at (5.6,-1.4) {\color{gray}{$1$}};
	\end{tikzpicture}
	$}\quad 	,
	\end{equation}
	where the entry $1$ at $\de_2 - \de_4$ corresponding to the leftmost $+$ in $\sigma^{\rm red}(S)$ is moved to the place at $\de_1 - \de_4$.
Here the bold-faced multiplicities denote the ones appearing in the sequences \eqref{eq: sigma(S)} and \eqref{eq: sigma(b_+)}.
	Similarly, applying $\tf_1$ once again, we get
	\begin{equation} \label{eq:f1b2}
		\tf^2_1 b = (\tf_1 S, \tf_1 b_+, b_-) = \quad
		\scalebox{0.8}{$
		\begin{tikzpicture}[baseline=(current  bounding  box.center), every node/.style={scale=0.9}]
			\node (root_34) at (0,0) {\color{gray}{$0$}};
			\node (root_35) at (0.7,0.7) {\color{gray}{$1$}};
			\node (root_36) at (1.4,1.4) {\color{gray}{$0$}};
			\node (root_37) at (2.1,2.1) {\color{gray}{$1$}};
			\node (root_24) at (0.7,-0.7) {${\bm 0}$};
			\node (root_25) at (1.4,0) {${\bm 1}$};
			\node (root_26) at (2.1,0.7) {${\bm 1}$};
			\node (root_27) at (2.8,1.4) {${\bm 0}$};
			\node (root_14) at (1.4,-1.4) {${\bm 1}$};
			\node (root_15) at (2.1,-0.7) {${\bm 0}$};
			\node (root_16) at (2.8,0) {${\bm 1}$};
			\node (root_17) at (3.5,0.7) {${\bm 0}$};
			\node (root_23) at (3.5,2.1) {${\bm 2}$};
			\node (root_13) at (4.2,1.4) {${\bm 1}$};
			\node (root_12) at (4.9,2.1) {$\red{\bm 3}$};	
			\node (root_45) at (2.8,-1.4) {\color{gray}{$2$}};
			\node (root_46) at (3.5,-0.7) {\color{gray}{$1$}};
			\node (root_47) at (4.2,0) {\color{gray}{$1$}};
			\node (root_56) at (4.2,-1.4) {\color{gray}{$1$}};
			\node (root_57) at (4.9,-0.7) {\color{gray}{$2$}};
			\node (root_67) at (5.6,-1.4) {\color{gray}{$1$}};
		\end{tikzpicture}
		$}\quad
		,
	\end{equation}

	\noindent
	where $\tf_1 b_+$ is given by adding $1$ at $\de_1 -\de_2$ since there is no $+$ in $\sigma^{\rm red}(b_+)$.
	\smallskip

	{\it Case 3}. Suppose that $i = 5$.
	By \eqref{eq:abstract crystal B(infty)}, we have 
	\begin{equation*}
	\tf_5 b = (S, b_+, \tf_5 b_-) = \quad
	\scalebox{0.8}{$
	\begin{tikzpicture}[baseline=(current  bounding  box.center), every node/.style={scale=0.9}]
		\node (root_34) at (0,0) {$\color{gray}{0}$};
		\node (root_35) at (0.7,0.7) {$\color{gray}{1}$};
		\node (root_36) at (1.4,1.4) {$\color{gray}{0}$};
		\node (root_37) at (2.1,2.1) {$\color{gray}{1}$};
		\node (root_24) at (0.7,-0.7) {$\color{gray}{1}$};
		\node (root_25) at (1.4,0) {$\color{gray}{1}$};
		\node (root_26) at (2.1,0.7) {$\color{gray}{1}$};
		\node (root_27) at (2.8,1.4) {$\color{gray}{0}$};
		\node (root_14) at (1.4,-1.4) {$\color{gray}{0}$};
		\node (root_15) at (2.1,-0.7) {$\color{gray}{0}$};
		\node (root_16) at (2.8,0) {$\color{gray}{1}$};
		\node (root_17) at (3.5,0.7) {$\color{gray}{0}$};
		\node (root_23) at (3.5,2.1) {$\color{gray}{2}$};
		\node (root_13) at (4.2,1.4) {$\color{gray}{1}$};
		\node (root_12) at (4.9,2.1) {$\color{gray}{2}$};	
		\node (root_45) at (2.8,-1.4) {$\red{\bm 1}$};
		\node (root_46) at (3.5,-0.7) {$\red{\bm 2}$};
		\node (root_47) at (4.2,0) {$\color{gray}{1}$};
		\node (root_56) at (4.2,-1.4) {${\bm 1}$};
		\node (root_57) at (4.9,-0.7) {$\color{gray}{2}$};
		\node (root_67) at (5.6,-1.4) {$\color{gray}{1}$};
	\end{tikzpicture}
	$}\quad
	,
	\end{equation*}\smallskip
	
	\noindent
	where $\tf_5 b_-$ is computed in the same way as $\tf_1b_+$ in {\it Case 2}.
	}
	\end{ex}
	
	\begin{ex} \label{ex: bxla bkla}
	{\em
	Let us compare the crystal operators on $ \ms{B}(\infty)$ with  the those on $\ms{B}(X(\la))$ and $\ms{B}(K(\la))$ for $\la\in P^+$ (\eqref{eq:abstract crystal parabolic Verma} and Proposition \ref{prop: description of BKla}).
 Let $b=(S,b_+,b_-) \in \ms{B}(\infty) $ given.
%
	Let $m=3$, $n=4$ and $\la = 6 \de_1 + 4 \de_2 + \de_3 + 3 \de_4 + 2 \de_5 - 4 \de_7 \in P^+ $. Let $b=(S,b_+,b_-)$ as in Example \ref{ex:B infty}.

		First consider $\ms{B}(X(\la))$.
	We claim that $b_- \ot t_{\la_-} \in \ms{B}_{0|n}(\la_-)$.
	For example, to compute $\varepsilon_5^*(b_-)$, we consider
	\begin{equation*}
		\sigma_*(b_-) = ( -^{c_{57}} , +^{c_{67}} , -^{c_{56}}) = (-,-,+,-) , \qquad \sigma_* ^{\rm red}(b_-) = (-,-, \, \cdot \, , \, \cdot \,),
	\end{equation*}
	where $\sigma_* ^{\rm red}(b_-)$ is given in the same way as in Example \ref{ex:B infty}.
	Then $ \varepsilon^*_5(b_-) = 2 $, which is the number of $-$'s in $\sigma_* ^{\rm red}(b_-)$ (cf.~\cite{CT,Re}).
	Similarly we have $\varepsilon^*_4(b_-) =1 $ and $\varepsilon^*_6(b_-) = 1$.
	By \eqref{eq: iso 0nla}, we have $b_-^{\la_-}:=b_- \ot t_{\la_-} \in \ms{B}_{0|n}(\la_-)$, 
	since $\varepsilon^*_4(b_-)=1 \leq 1$, $\varepsilon^*_5(b_-)=2 \leq 2$ and $\varepsilon^*_6(b_-)=1 \leq 4$.
	Hence we may regard $b':=(S,b_+,b_-^{\la_-}) \, \ot \, t_{\la_+}$ as an element in $\ms{B}(X(\la))$,
	where we still identify $b'$ with the array given in \eqref{eq: b}.
	
	Let us compute $\tf_i b'$. 
	When $1 \leq i \leq m$, $\tf_i b'$ is same as in {\it Cases 1, 2} of Example \ref{ex:B infty} by \eqref{eq:abstract crystal parabolic Verma}.
	
	Consider the case when $i=5$.
	We first have $\varphi_5(b_-^{\la_-})=\varepsilon_5(b_-^{\la_-}) + \langle {\rm wt}(b_-^{\la_-}) , \alpha_5^\vee \rangle = 1$,
	since $\varepsilon_5(b_-^{\la_-}) = \varepsilon_5(b_-)=1$ and $\langle {\rm wt}(b_-^{\la_-}) , \alpha_5^\vee \rangle=0$.
	To compute $\varepsilon_5(S)$, we consider
	\begin{equation*}
		\sigma(S) = ( +^{c_{15}} , -^{c_{16}} , +^{c_{25}} , -^{c_{26}} , +^{c_{35}} , -^{c_{36}}) = ( \, \cdot \, , - , + , - , + , \, \cdot \, ), \quad \sigma^{\rm red}(S) =( \, \cdot \, , - , \, \cdot \, , \, \cdot \, , + , \, \cdot \, ),
	\end{equation*}
	which implies $\varepsilon_5(S)=1$ and $\varphi_5(S)=1$.
	Since $\varphi_5(b_-^{\la_-})=1 \leq \varepsilon_5(S)=1$, we have by Proposition \ref{prop:tensor product rule},

	\begin{equation*}
		\tf_5 b' = (\tf_5 S, b_+, b_-^{\la_-}) = \quad
		\scalebox{0.8}{$
		\begin{tikzpicture}[baseline=(current  bounding  box.center), every node/.style={scale=0.9}]
			\node (root_34) at (0,0) {$\color{gray}{0}$};
			\node (root_35) at (0.7,0.7) {$\red{\bm 0}$};
			\node (root_36) at (1.4,1.4) {$\red{\bm 1}$};
			\node (root_37) at (2.1,2.1) {$\color{gray}{1}$};
			\node (root_24) at (0.7,-0.7) {$\color{gray}{1}$};
			\node (root_25) at (1.4,0) {${\bm 1}$};
			\node (root_26) at (2.1,0.7) {${\bm 1}$};
			\node (root_27) at (2.8,1.4) {$\color{gray}{0}$};
			\node (root_14) at (1.4,-1.4) {$\color{gray}{0}$};
			\node (root_15) at (2.1,-0.7) {${\bm 0}$};
			\node (root_16) at (2.8,0) {${\bm 1}$};
			\node (root_17) at (3.5,0.7) {$\color{gray}{0}$};
			\node (root_23) at (3.5,2.1) {$\color{gray}{2}$};
			\node (root_13) at (4.2,1.4) {$\color{gray}{1}$};
			\node (root_12) at (4.9,2.1) {$\color{gray}{2}$};	
			\node (root_45) at (2.8,-1.4) {${\bm 2}$};
			\node (root_46) at (3.5,-0.7) {${\bm 1}$};
			\node (root_47) at (4.2,0) {$\color{gray}{1}$};
			\node (root_56) at (4.2,-1.4) {${\bm 1}$};
			\node (root_57) at (4.9,-0.7) {$\color{gray}{2}$};
			\node (root_67) at (5.6,-1.4) {$\color{gray}{1}$};
		\end{tikzpicture}
		$}\quad
		,
	\end{equation*}
	\smallskip

\noindent
where the entry $1$ at $\de_3 - \de_5$ corresponding to the leftmost $+$ in $\sigma^{\rm red}(S)$ is moved to the place at $\de_3-\de_6$.
Note that $\tf_5 b' $ is not equal to $\tf_5 b$ in Example \ref{ex:B infty} (see Remark \ref{rem: compatibility}).

Next consider $\ms{B}(K(\la))$.
We can check that $b_+^{\la_+}:=b_+ \ot t_{\la_+} \in \ms{B}_{m|0}(\la_+)$ by \eqref{eq: iso 0nla}, and may regard $b'':=(S,b_+^{\la_+},b_-^{\la_-}) $ as an element in $\ms{B}(K(\la))$,
where $b_-^{\la_-} = b_- \ot t_{\la_-}$.
When $i \geq m$, $\tf_i b''$ is the same as $\tf_i b'$ by Proposition \ref{prop: description of BKla}.
When $i = 1$, $\tf_1 b''$ is the same as $\tf_1 b$ \eqref{eq:f1b}.
However, we have $\tf^2_1 b'' = (\tf_1 S, \tf_1 b_+^{\la_+}, b_-^{\la_-}) = {\bf 0}$,
since $\varepsilon^*_1(\tf_1 b_+)=3 >2$ and hence $\tf_1 b_+^{\la_+} = {\bf 0} $, while $\tf^2_1 b \ne {\bf 0}$ \eqref{eq:f1b2}.
	}
	
\end{ex}

\appendix
\section{The algebra $B_q(U)$ and tensor product rule} \label{sec: q-Boson and tensor product}\label{sec:app}

\subsection{The algebra $B_q(U)$}\label{app:B_q}

Let $U$ be given as in \eqref{eq: def of U}.
Let $B_q(U)=B_q$ be the associative $\Bbbk$-algebra with $1$ 
generated by $k_\mu$, $e_i'$, $f_i$ for $\mu \in P$ and $i \in I$ with the following relations:
\begin{gather*}
	k_0 = 1, \quad k_{\mu+\mu'} = k_\mu k_{\mu'} \quad (\mu, \mu' \in P), \\
	k_\mu e_i' k_\mu^{-1} = q^{\langle \mu, \alpha_i^\vee \rangle} e_i' \quad 
	k_\mu f_i k_\mu^{-1} = q^{-\langle \mu, \alpha_i^\vee \rangle} f_i \quad (\mu \in P, i \in I),\\
	e_i'f_j = q^{\langle \alpha_j, \alpha_i^\vee \rangle} f_j e_i' + \delta_{ij} \quad (i, j \in I), \\
	e_i' e_j' -  e_j' e_i' = f_i f_j -  f_j f_i =0 \quad (|i-j|>1), \\
	{e'}_i^2 e_j'- [2] {e'}_i e_j' e_i' + e_j' {e'}_i^2= f_i^2 f_j- [2] f_i f_j f_i+f_j f_i^2 = 0 \quad (|i-j|=1),
\end{gather*}
which is isomorphic to $B_q(\gl_n)$ in \cite{Na}.
We denote by $B_q^\circ$ the subalgebra of $B_q$ generated by $e_i'$ and $f_i$ for $i \in I$ \cite{Kas91}.
Note that $U^-$ is a left $B_q$-module where the actions of $e'_i$, $f_i$, and $k_\mu$ are given by \eqref{eq: derivation}, left multiplication, and conjugation, respectively.
Indeed, $U^-$ is irreducible since it is an irreducible $B_q^\circ$-module \cite[Lemma 3.4.2, Corollary 3.4.9]{Kas91}.

Let $\mc{O}$ be the category of ${\Boson}$-modules $V$ such that $V$ has a weight space decomposition $V=\bigoplus_{\mu\in P} V_\mu$ with ${\rm dim}\,V_\mu < \infty$, and for any $v \in V$, there exists $l$ such that $e_{i_1}'\dots e_{i_l}' v = 0$ for any $i_1, \dots, i_l \in I$.
Note that $U^-$ is the highest weight ${\Boson}$-module in $\mc{O}$ with highest weight $0$.
The following is known in \cite[Remark 3.4.10]{Kas91}, \cite[Propositions 2.3 and 2.4]{Na}.
\begin{prop} \label{prop: basic prop of B-modules}
The category $\mc{O}$ is semisimple, where each irreducible module in $\mc{O}$ is a highest weight module and it is isomorphic to $U^-$ as a $\Boson^\circ$-module.
\end{prop}

Let $V$ be a ${\Boson}$-module in $\mc{O}$.
For a weight vector $u \in V$ and $i\in I$, 
we have $u = \sum_{k \ge 0} f_i^{(k)}u_k$ with $e_i'(u_k) = 0$ for all $k \ge 0$, 
and define $\te_i u$ and $\tf_i u$ as in \eqref{eq:Kashiwara operators for U^-}.
We may define a crystal base $(L,B)$ of $V$ satisfying the conditions (C1)--(C5) in Section \ref{subsec:crystal base of hw gl q=0} with respect to $\te_i$ and $\tf_i$ $(i \in I)$.
For example, $(L(\infty),B(\infty))$ in \eqref{eq: CB for U-} is a crystal base of $U^{-}$.
\begin{cor}$($\cite[Remark 3.5.1]{Kas91}$)$ \label{cor: Boson module has crystal}
Any ${\Boson}$-module in $\mc{O}$ has a crystal base, which is isomorphic to a direct sum of $(L(\infty), B(\infty))$'s up to shift of weights.
\end{cor}

Let $\cmB : \Boson \longrightarrow U \ot \Boson$
be a homomorphism of $\Bbbk$-algebra given by
\begin{equation} \label{eq: comul on B}
\begin{split}
	\cmB(k_\mu) &= k_\mu \ot k_\mu, \\
	\cmB(e_i') &= (q^{-1}-q) k_ie_i \ot 1 + k_i \ot e_i', \\
	\cmB(f_i) &= f_i \ot 1 + k_i \ot f_i,
\end{split}
\end{equation}
for $\mu \in P$ and $i \in I$, which satisfies the coassociativity law \cite[Remark 3.4.11]{Kas91} (see also \cite[Proposition 1.2]{Na}). 
Let $V_1$ be a finite-dimensional $U$-module and let $V_2$ be a ${\Boson}$-module in $\mc{O}$.
Then $V_1 \ot V_2$ is a $\Boson$-module in $\mc{O}$ via \eqref{eq: comul on B}.

\subsection{Tensor product rule}
The following is an analogue of \cite[Theorem 1]{Kas91}.
\begin{thm}$($\cite[Remark 3.5.1]{Kas91}$)$ \label{thm: tensor product rule for Boson}
Let $V_1$ be a finite-dimensional $U$-module and let $V_2$ be a ${\Boson}$-module in $\mc{O}$.
Let $(L_k, B_k)$ be a crystal base of $V_k$ for $k = 1, 2$. 
Then $V_1 \ot V_2$ is in $\mc{O}$, and the pair $(L_1 \ot L_2,B_1 \ot B_2)$ is a crystal base of $V_1 \ot V_2$. 
Moreover, $\te_i$ and $\tf_i$ $(i \in I)$ act on $B_1 \ot B_2$ by
{\allowdisplaybreaks
\begin{equation}\label{eq:tensor product rule for Boson}
\begin{split}
&\te_i(b_1\otimes b_2)= \begin{cases}
\te_i b_1 \otimes b_2, & \text{if $\varphi_i(b_1)\geq\varepsilon_i(b_2)$}, \\ 
 b_1 \otimes \te_ib_2, & \text{if $\varphi_i(b_1)<\varepsilon_i(b_2)$},\\
\end{cases}
\\
&\tf_i(b_1\otimes b_2)=
\begin{cases}
\tf_ib_1 \otimes  b_2, & \text{if $\varphi_i(b_1)>\varepsilon_i(b_2)$}, \\
b_1 \otimes \tf_i  b_2, & \text{if $\varphi_i(b_1)\leq\varepsilon_i(b_2)$}, 
\end{cases}
\end{split}
\end{equation}}
\!\!for $b_1 \ot b_2 \in B_1 \ot B_2$,
where $\varphi_i(b_1) =\max\{\, k \,|\, \tf_i^k b_1  \neq {\bf 0} \, \}$
and
$\varepsilon_i(b_2)=\max\{\, k \,|\, \te_i^k b_2  \neq {\bf 0} \, \}$.
\end{thm}

\begin{cor} \label{cor:Bla ot Binf}
	For $\la \in P^+$, we have
	\begin{equation*}
		B(\la) \ot B(\infty) \cong \bigsqcup_{b \in B(\la)} B(\infty) \ot T_{{\rm wt}(b)},
	\end{equation*}
	as a $U$-crystal.
\end{cor}
\pf
	Let $\left( B(\la) \ot B(\infty) \right) ^h := \{\, b^\circ \in B(\la) \ot B(\infty) \, | \, \te_i(b^\circ) = 0  \text{ for } i \in I \, \}$.
	By Theorem \ref{thm: tensor product rule for Boson} and Corollary \ref{cor: Boson module has crystal},
	we have 
	\begin{equation*}
		B(\la) \ot B(\infty) \cong \bigsqcup_{b^\circ \in \left( B(\la) \ot B(\infty) \right) ^h } B(\infty) \ot T_{{\rm wt}(b^\circ)},
	\end{equation*}
	since the connected component of $b^\circ \in \left( B(\la) \ot B(\infty) \right) ^h$ is isomorphic to $B(\infty) \ot T_{{\rm wt}(b^\circ)}$.
	By \eqref{eq:tensor product rule for Boson}, we have that $b^\circ = b_1 \ot b_2 \in \left( B(\la) \ot B(\infty) \right) ^h$
	if and only if $b_1=u_\la$ and $\varepsilon_i(b_2) \le \langle \la, \alpha_i^\vee \rangle$ for all $i \in I$.
	Let $*$ denote the involution on $B(\infty)$ induced by an involution $*$ on $U^-$ \cite[Theorem 2.1.1]{Kas93}.
	Note that
	${\rm wt}(b_2) = {\rm wt}(b_2^*)$ and $\varepsilon_i(b_2) = \varepsilon_i^*(b_2^*)$ for $i \in I$, where $\varepsilon^*_i(b')=\varepsilon_i((b')^*)$ for $b' \in B(\infty)$.
	Since $B(\la) = \{ \, b \ot t_\la \in B(\infty) \ot T_\la \, | \, \varepsilon^*_i(b) \leq \langle \la , \alpha^\vee_i \rangle \text{ for } i \in I \, \}$, 
	we have a weight preserving bijection from $\left( B(\la) \ot B(\infty) \right) ^h$ to $B(\la)$ sending $b^\circ=u_\la \ot b_2$ to $b_2^* \ot t_\la$.
	Therefore
	\begin{equation*} 
		B(\la) \ot B(\infty) \cong \bigsqcup_{b \in B(\la)} B(\infty) \ot T_{{\rm wt}(b)}.
	\end{equation*}
	
\qed

For the reader's convenience, we give a simple self-contained proof of Theorem \ref{thm: tensor product rule for Boson},
which also works for the case associated with symmetrizable Kac-Moody algebras in \cite{Kas91} (cf.~\cite[17.1]{Lu}).

\begin{lem} \label{lem:identity for binomial coefficients}
	For integers $a \ge b \ge 0$, we have
	\begin{equation*}
		\begin{bmatrix} a \\ b \end{bmatrix} = q^{-b} \begin{bmatrix} a-1 \\ b \end{bmatrix} + q^{a-b} \begin{bmatrix} a-1 \\ b-1 \end{bmatrix} \text{.}
	\end{equation*}
	Here, for $c,d \in \Z$, $\begin{bmatrix} c \\ d \end{bmatrix} := \cfrac{[c]!}{[c-d]![d]!}$ if $c \geq d \geq 0$, and $\begin{bmatrix} c \\ d \end{bmatrix} := 0$ otherwise.
\end{lem}
\pf
	It follows directly from the definition of $q$-binomial coefficient.
\qed

\begin{lem} \label{lem:Akito identity}
	For $a, b \in \Z_{+}$, we have
	\begin{equation*} 
		\sum^{b}_{i=\max(0,b-a)} q^{(a-1)(b-i)} C_i(q) \begin{bmatrix} a \\ b-i \end{bmatrix}=q^{2ab},
	\end{equation*}
	where $C_i(q)=\prod_{i+1 \leq k \leq b} (q^{2k}-1)$ for $i< b$ and $C_b(q)=1$.
\end{lem}
\pf
Put
\begin{equation*}
	P(a,b) = \sum^{b}_{i=\max(0,b-a)} q^{(a-1)(b-i)} C_i(q) \begin{bmatrix} a \\ b-i \end{bmatrix}.
\end{equation*}
Note that $P(0,b) = 1$ for $b \in \Z_{+}$.
We use induction on $a$.
Suppose that $P(a,b) = q^{2ab}$ for $b \in \Z_{+}$.
Since $P(a+1,0) = 1$, we may assume that $b > 0$.
By Lemma \ref{lem:identity for binomial coefficients}, we have the following recursive formula:
\begin{equation*}
	P(a+1,b) = P(a,b) + q^{2a}(q^{2b}-1) P(a,b-1).
\end{equation*}
Then by induction hypothesis, we have
\begin{equation*}
	\begin{split}
		P(a+1,b) = q^{2ab}+q^{2a}(q^{2b}-1)q^{2a(b-1)} = q^{2(a+1)b}.
	\end{split}
\end{equation*}
This completes the induction.
\qed
\smallskip

\begin{proof}[Proof of Theorem \ref{thm: tensor product rule for Boson}] 
We divide the proof into the following steps.

\noindent {\em Step 1.}
Let $U = U_q(\mathfrak{sl}_2)=\langle e,f,k,k^{-1} \rangle$ and let $B=\Boson(U)=\langle e',f,k,k^{-1} \rangle$.
Let 
\begin{equation*}
	V_1 = \bigoplus_{k = 0}^l \Bbbk f^{(k)} v_1, \qquad
	V_2 = \bigoplus_{k \ge 0} \Bbbk f^{(k)} v_2,
\end{equation*}
be irreducible representations of $U$ and $B$ generated by $v_1$ and $v_2$, respectively, where $e \, v_1 = 0$ and $e' \, v_2 = 0$.
Let
\begin{equation*}
	L_1 = \bigoplus_{k = 0}^l A_0 f^{(k)} v_1, \qquad
	L_2 = \bigoplus_{k \ge 0} A_0 f^{(k)} v_2.
\end{equation*}
be the crystal lattices of $V_1$ and $V_2$, respectively.
For $0 \le t \le l$, we let
\begin{equation*}
	E_t = \sum_{i=0}^{t} a_{t,i}(q) f^{(i)}v_1 \otimes f^{(t-i)}v_2 \in L_1 \otimes L_2,
\end{equation*}
where $a_{t,i}(q) \in A_0$ is given by
\begin{equation*}
a_{t,i}(q) = \prod_{0 \leq j \leq i-1} \frac{q^{l-t+1}}{q^{2(l-j)}-1}, \quad a_{t,0}(q)=1.
\end{equation*}
We can check that $E_t$ is a highest weight vector of $V_1 \otimes V_2$, that is, $e' E_t = \cmB (e')(E_t) = 0$, and $\{ E_t \, | \, 0 \le t \le l \}$ is a $\Bbbk$-basis of $\mathrm{Ker} \, e'$.
Since $l-t+1 > 0$, we also have
\begin{equation*}
	\qquad E_t \equiv v_1 \otimes f^{(t)} v_2 \quad (\mathrm{mod} \,\,\, q\, L_1 \otimes L_2).
\end{equation*}

By \eqref{eq: comul on B}, we have for $s \ge 0$
\begin{equation*} \label{eq:comulti formula for fs}
	\cmB (f^{(s)}) = \sum_{i=0}^{s} q^{-i(s-i)}f^{(s-i)}k^i \otimes f^{(i)},
\end{equation*}
and
\begin{equation*}
	\begin{split}
		\cmB(f^{(s)})(E_t) & = \sum^{t}_{i=0} \sum^{s}_{j=0} a_{t,i}(q) q^{-j(s-j)}f^{(s-j)}k^j f^{(i)} v_1 \otimes f^{(j)} f^{(t-i)} v_2 \\
		& = \sum^{t}_{i=0} \sum^{s}_{j=0} a_{t,i}(q) q^{(l-2i-s+j)j} \begin{bmatrix} s-j+i \\ i \end{bmatrix} \begin{bmatrix} t-i+j \\ j \end{bmatrix} f^{(s-j+i)} v_1 \otimes f^{(t-i+j)} v_2.
	\end{split}
\end{equation*}
\smallskip

Now, it is enough to show that
\begin{equation} \label{eq:claim1}
	\cmB(f^{(s)})(E_t) \equiv 
	\begin{cases}
		f^{(s)}v_1\otimes f^{(t)}v_2  & \text{if $s \le l-t$,} \\
		f^{(l-t)}v_1\otimes f^{(2t+s-l)}v_2  & \text{if $s > l-t$,}
	\end{cases}
	\pmod{qL_1 \ot L_2}.
\end{equation}
\smallskip

{\it Case 1.} $s\leq l-t$. In this case,
we can check that 
$$ a_{t,i}(q) q^{(l-2i-s+j)j} \begin{bmatrix} s-j+i \\ i \end{bmatrix} \begin{bmatrix} t-i+j \\ j \end{bmatrix} \in q A_0 $$
if $(i,j)\ne (0,0)$. Therefore $ \cmB(f^{(s)})(E_t) \equiv f^{(s)}v_1\otimes f^{(t)}v_2	\pmod{qL_1 \otimes L_2}$.
\smallskip

{\it Case 2.} $s> l-t$. It is enough to consider the case of $j \geq s+i-l$ since we have $f^{(s-j+i)}v_1=0$ if $s-j+i>l$.
Thus we may write
\begin{equation} \label{eq:equation of Csk}
	\begin{split}
		& \cmB(f^{(s)})E_t  \\
		&= \sum^{t}_{i=0} \sum^{s}_{j=\max(0,s+i-l)} a_{t,i}(q) q^{(l-2i-s+j)j} 
		\begin{bmatrix} s-j+i \\ i \end{bmatrix} \begin{bmatrix} t-i+j \\ j \end{bmatrix} 
		f^{(s-j+i)}v_1 \otimes f^{(t-i+j)}v_2 \\
		&= \sum^{l}_{k=0} C(s,k) f^{(l-k)}v_1 \otimes f^{(t+s-l+k)}v_2 \text{,}
	\end{split}
\end{equation}
where
\begin{equation} \label{eq:def of Csk}
	C(s,k) = \sum^{\min(t,l-k)}_{i=\max(0,l-s-k)} a_{t,i}(q) q^{(k-i)(i-l+s+k)} \begin{bmatrix} l-k \\ i \end{bmatrix} \begin{bmatrix} t-l+s+k \\ t-i \end{bmatrix}.
\end{equation}

Let us call an integer $d$ a {\em minimal degree} of $f(q) \in \mathbb{Q}(q)$ if $d$ is maximal such that $f(q) \in q^d A_0 $.
Let $d(s,k)$ be the minimal degree of $C(s,k)$. To prove \eqref{eq:claim1}, we show
\begin{equation} \label{eq:claim2}
	d(s,k)=
	\begin{cases}
	0 & \text{if  } k = t \text{,} \\
	(s+k-l)(k-t) & \text{if  } k > t \text{,} \\
	(s+t+1-l)(t-k) & \text{if  } k < t \text{.}
	\end{cases}
\end{equation}

First, assume that $k \geq t$.
Note that the minimal degree of
$$a_{t,i}(q) q^{(k-i)(i-l+s+k)} \begin{bmatrix} l-k \\ i \end{bmatrix} \begin{bmatrix} t-l+s+k \\ t-i \end{bmatrix}$$ 
is 
\begin{equation} \label{eq:degree of k>t}
	i^2+(2k-2t+1)i-(l-s-k)(k-t)\text{.}
\end{equation}
Since $(2k-2t+1) \geq 1$, the minimal value of \eqref{eq:degree of k>t} is $ -(l-s-k)(k-t) \geq 0 $ which appears only when $i=0$ for $0=\max (0,l-s-k) \le i \le \min (t,l-k)$.
Therefore, \eqref{eq:claim2} holds in this case.
\smallskip
Next, assume that $k<t$. We need the following recursive formula for $C(s,k)$:
\begin{equation} \label{eq:recursive for Bmstk}
	C(s+1,k) = \frac{1}{[s+1]} (q^{2k-l}[t+s-l+k+1]C(s,k)+[l-k]C(s,k+1)),
\end{equation}
which follows from \eqref{eq:equation of Csk} and the identity
\begin{equation*}
	\cmB(f^{(s+1)}) = \frac{1}{[s+1]} \cmB f \cdot \cmB (f^{(s)}).
\end{equation*}
We use induction on $k$ to verify \eqref{eq:claim2}. 
Note that one can check directly from \eqref{eq:def of Csk} and Lemma \ref{lem:Akito identity} that
\begin{equation*}
	C(s,0) = \frac{q^{t(s+t-l+1)}}{(q^{2t}-1)\cdots (q^{2}-1)}, \quad
	d(s,0) = t(s+t-l+1).
\end{equation*}
By \eqref{eq:recursive for Bmstk}, we have
$$C(s,k+1) = \frac{1}{[l-k]}([s+1]C(s+1,k)-q^{2k-l}[t+s-l+k+1]C(s,k)),$$
for $k<t$ and $s>l-t$.
Since the minimal degree of $\frac{[s+1]}{[l-k]}C(s+1,k)$ is 
$$Y=(l-k-s-1)+(s+1+t-l+1)(t-k),$$
and that of $\frac{q^{2k-l}[t+s-l+k+1]}{[l-k]}C(s,k)$ is
\begin{equation*}
	\begin{split}
		Z&=(l-k-(t+s-l+k+1))+(2k-l)+(s+t-l+1)(t-k)\\
		&=(s+t-l+1)(t-(k+1)),
	\end{split}
\end{equation*}
by induction hypothesis,
we have $Y-Z=2(t-k)>0$, that is, $Y>Z$.
Therefore, $d(s,k+1)=Z=(s+t-l+1)(t-(k+1))$.
We conclude that $C(s,k) \in q A_0$ if $k \ne t$ by \eqref{eq:claim2}, and $C(s,t) \in 1+ q A_0$,
which implies $\cmB(f^{(s)})E_t \equiv f^{(l-t)}v_1\otimes f^{(2t+s-l)}v_2 \,(\mathrm{mod} \,\,\, q\,L_1 \otimes L_2)$.

The proof of \eqref{eq:claim1} completes by {\em Case 1} and {\em Case 2}. 
\smallskip

\noindent {\em Step 2.} 
Consider the case when $U=U_q(\gl_n)$.
Let us check that $(L, B)$ in Theorem \ref{thm: tensor product rule for Boson} satisfies the conditions (C1)--(C5) and \eqref{eq:tensor product rule for Boson}.
Note that $(V_1 \ot V_2)_{\lambda} = \bigoplus_{\mu + \nu = \lambda} (V_1)_\mu \ot (V_2)_\nu$ for $\la \in P$ with respect to $\cmB$ \eqref{eq: comul on B}.
It is clear that (C1)--(C3) hold.
Let us consider (C4). 
Fix $i\in I$. Let $u_1 \in L_1$ and $u_2 \in L_2$ be weight vectors such that $u_1 + qL_1 =: b_1 \in B_1$ and $u_2 + qL_2 =: b_2 \in B_2$.
We may assume that $u_1 = f_i^{(m_1)}v_1$ with $e_i\, v_1 = 0 $ and $u_2 = f_i^{(m_2)}v_2$ with $e_i'(v_2) = 0$.
Let $U_i=\langle e_i,f_i,k_i,k_i^{-1} \rangle \subset U$. 
Then $U_i\cong U_q(\frak{sl}_2)$ and $B_q(U_i) \cong \langle e'_i,f_i,k_i,k_i^{-1} \rangle \subset B_q(U)$. 
Consider the $U_i$-submodule $V'_1$ of $V_1$ generated by $v_1$ and the $B_q(U_i)$-submodule $V'_2$ of $V_2$ generated by $v_2$.
If we let
\begin{gather*}
	L'_1 = \bigoplus_{k = 0}^{\langle \mathrm{wt}(v_1), \alpha^\vee_i \rangle} A_0 f^{(k)}_i v_1, \quad
	L'_2 = \bigoplus_{k \ge 0} A_0 f^{(k)}_i v_2, \\
	B'_1 = \{ \, f^{(k)}_i v_1 \!\!\! \pmod{q L'_1} \, | \, 0 \le k \le \langle \mathrm{wt}(v_1), \alpha^\vee_i \rangle \, \},  \quad 
	B'_2 = \{ \, f^{(k)}_i v_2 \!\!\! \pmod{q L'_2} \, | \, k \ge 0 \, \},
\end{gather*}
then $(L'_t ,B'_t)$ is the crystal base of $V'_t$ for $t=1,2$.
By {\em Step 1}, we have
\begin{gather*}
	\tilde{x}_i(u_1 \ot u_2) \in L'_1 \ot L'_2 \subset L_1 \ot L_2 = L, \\
	\tilde{x}_i (b_1 \ot b_2) \in B'_1 \ot B'_2 \sqcup \{0\} \subset (B_1 \ot B_2) \sqcup \{0\},
\end{gather*}
$(x=e,f)$ 
together with \eqref{eq:tensor product rule for Boson}.
Thus (C4) is satisfied.
The condition (C5) follows immediately from \eqref{eq:tensor product rule for Boson}.
\end{proof}

\end{document}